\documentclass{article}

\usepackage{graphicx}
\usepackage{color}
\usepackage{amssymb}
\usepackage{amsthm}
\usepackage{amsmath}
\usepackage{mathrsfs}
\usepackage{wasysym}
\usepackage{braket}		% defines \middle similar to \left and \right
\usepackage{mdwlist}		% interrupting enumerated lists: \suspend and \resume
\usepackage{mathtools}
\usepackage{array}
\usepackage{enumerate}
\usepackage{calc}
\usepackage[normalem]{ulem}
\usepackage[marginparwidth=7em]{geometry}
\usepackage{tikz}
\usetikzlibrary{matrix}
\usepackage{tikz}
\usetikzlibrary{arrows,decorations.pathmorphing,decorations.shapes,shapes.geometric}
\usepackage{ifthen}
\usepackage[pdfborder={0 0 0 [3 3]}]{hyperref}
\theoremstyle{plain}
\newtheorem{theorem}{Theorem}
\newtheorem{corollary}[theorem]{Corollary}
\newtheorem{lemma}[theorem]{Lemma}
\newtheorem{proposition}[theorem]{Proposition}
\newtheorem{conjecture}[theorem]{Conjecture}
\theoremstyle{definition}

\newtheorem{example}[theorem]{Example}

\newtheorem{question}[theorem]{Question}

\newcommand{\exampleqed}{\hfill$\ocircle$}
\newcommand{\xSx}[1]{\mathscr{#1}}		% for structures
\newcommand{\xSp}[1]{\mathcal{#1}}		% a space, or ...
\newcommand{\family}[1]{\mathfrak{#1}}	% e.g., sigma-algebra

\newcommand{\event}[1]{\mathcal{#1}}		% event
\newcommand{\xv}[1]{\mathbf{#1}}		% Vectors and Random Variables
\newcommand{\xZ}{\mathbb{Z}}			% The set of integers
\newcommand{\xN}{\mathbb{N}}			% The set of natural numbers
\newcommand{\xR}{\mathbb{R}}			% The set of reals
\newcommand{\xL}{\mathbb{L}}			% Lattice
\newcommand{\xT}{\mathbb{T}}			% Unit circle
\newcommand{\xO}{\circ}				% Function composition
\newcommand{\xCmpl}[1]{%				% Complement of a set
	#1^\textsf{c}%
}

\newcommand{\IsDef}{\triangleq}			% "Is Defined" symbol
\newcommand{\abs}[1]{%					% Absolute value
	\left\lvert#1\right\rvert%
}
\newcommand{\norm}[1]{%					% Norm
	\left\lVert#1\right\rVert%
}
\newcommand{\xRest}[2]{%				% Restriction of #1 to #2
	\left.#1\right\rvert_{#2}%
}
\newcommand{\supp}{%					% Support
	\operatorname{\mathrm{supp}}%
}

\newcommand{\diff}{%					% Difference set
	\operatorname{\mathrm{diff}}%
}

\newcommand{\xd}{\mathrm{d}}			% Differential
\newcommand{\xe}{\mathrm{e}}			% Euler's constant
\newcommand{\smallo}{o}					% Small o
\newcommand{\successor}{%					% Successor (in a total ordering)
	\operatorname{\mathrm{succ}}%
}

\newcommand{\xvar}{%					% variation
	\operatorname{\mathrm{var}}%
}

\newcommand{\xQuiet}{\diamondsuit}		% Blank configuration
\newcommand{\xVer}{\mathsf{V}}			% Vertical shift direction (for 2d shift)
\newcommand{\xHor}{\mathsf{H}}			% Horizontal shift direction (for 2d shift)

\newcommand{\OverLabel}[2]{%				% Label over a column of a matrix: \OverLabel{top-most element}{label}
	\newdimen\height
	\setbox0=\hbox{#1}
	\height=\ht0 \advance\height by \dp0
	\raisebox{0pt}[\dp0][0pt]{$\overbracket[0pt][3pt]{#1}^{#2}$}
}

\DeclareMathAlphabet{\mathpzc}{OT1}{pzc}{m}{it}

\newcommand{\symb}[1]{\mathtt{#1}}	% Symbol
\newcommand{\tsymb}[1]{\scriptstyle{\mathtt{#1}}}	% Symbol (small version)
\newcommand{\csymb}[1]{\makebox[2ex]{$\symb{#1}$}}		% symbol of a cell

\newlength{\contourbaseline}
\newlength{\configlinesep}
\newcommand{\ContourEmpty}{%				% Left down
	\begin{tikzpicture}[baseline=\contourbaseline, very thin]
		\clip (-1ex,-1ex) rectangle (1ex,1ex);%
		\draw (-1ex,-1ex) rectangle (1ex,1ex);%
	\end{tikzpicture}%
}

\newcommand{\ContourV}{%				% Left down
	\begin{tikzpicture}[baseline=\contourbaseline, very thin]
		\clip (-1ex,-1ex) rectangle (1ex,1ex);%
		\draw (-1ex,-1ex) rectangle (1ex,1ex);%
		\draw[thick] (0,1ex) -- (0,-1ex);%
	\end{tikzpicture}%
}
\newcommand{\ContourH}{%				% Left down
	\begin{tikzpicture}[baseline=\contourbaseline, very thin]
		\clip (-1ex,-1ex) rectangle (1ex,1ex);%
		\draw (-1ex,-1ex) rectangle (1ex,1ex);%
		\draw[thick] (-1ex,0) -- (1ex,0);%
	\end{tikzpicture}%
}
\newcommand{\ContourLD}{%				% Left down
	\begin{tikzpicture}[baseline=\contourbaseline, very thin]
		\clip (-1ex,-1ex) rectangle (1ex,1ex);%
		\draw (-1ex,-1ex) rectangle (1ex,1ex);%
		\draw[thick] (-1ex,0) .. controls(0,0) .. (0,-1ex);%
	\end{tikzpicture}%
}
\newcommand{\ContourRD}{%				% Left down
	\begin{tikzpicture}[baseline=\contourbaseline, very thin]
		\clip (-1ex,-1ex) rectangle (1ex,1ex);%
		\draw (-1ex,-1ex) rectangle (1ex,1ex);%
		\draw[thick] (1ex,0) .. controls(0,0) .. (0,-1ex);%
	\end{tikzpicture}%
}
\newcommand{\ContourLU}{%				% Left down
	\begin{tikzpicture}[baseline=\contourbaseline, very thin]
		\clip (-1ex,-1ex) rectangle (1ex,1ex);%
		\draw (-1ex,-1ex) rectangle (1ex,1ex);%
		\draw[thick] (-1ex,0) .. controls(0,0) .. (0,1ex);%
	\end{tikzpicture}%
}
\newcommand{\ContourRU}{%				% Left down
	\begin{tikzpicture}[baseline=\contourbaseline, very thin]
		\clip (-1ex,-1ex) rectangle (1ex,1ex);%
		\draw (-1ex,-1ex) rectangle (1ex,1ex);%
		\draw[thick] (1ex,0) .. controls(0,0) .. (0,1ex);%
	\end{tikzpicture}%
}
\newcommand{\ContourX}{%				% Left down
	\begin{tikzpicture}[baseline=\contourbaseline, very thin]
		\clip (-1ex,-1ex) rectangle (1ex,1ex);%
		\draw (-1ex,-1ex) rectangle (1ex,1ex);%
		\draw[thick] (-1ex,0) .. controls(0,0) .. (0,-1ex);%
		\draw[thick] (1ex,0) .. controls(0,0) .. (0,1ex);%
	\end{tikzpicture}%
}

\newcommand{\CellBlock}[4]{%
	\makebox[4ex]{%
		$\begin{array}{c@{}c}
			\csymb{#1} & \csymb{#2} \\ [\configlinesep]
			\csymb{#3} & \csymb{#4}
		\end{array}$%
	}
}
\newcommand{\SymbCellBlock}[4]{%
	\CellBlock{\symb{#1}}{\symb{#2}}{\symb{#3}}{\symb{#4}}%
}

\allowdisplaybreaks[3] 			% allow page break within equations
\makeatletter
\renewcommand\maketag@@@[1]{%
	\,\rlap{\hspace{\marginparsep}\m@th\normalfont\footnotesize#1}%
}
\makeatother

\makeatletter
\renewcommand{\@seccntformat}[1]{%
	\llap{\csname the#1\endcsname\hspace{\marginparsep}}%
}
\makeatother

\makeatletter
\renewcommand{\@biblabel}[1]{%
	\llap{[#1]\hspace{\marginparsep}}\!\!\!%
}
\makeatother

\begin{document}

\title{%
	\huge
	Statistical Mechanics\\
	of Surjective Cellular Automata%
	\footnotetext{Last update:~\today}
}

\author{%
	Jarkko Kari\\
	{\normalsize Department of Mathematics, University of Turku, Finland}\bigskip\\
	Siamak Taati\\
	{\normalsize Mathematical Institute, Utrecht University, The Netherlands}
}

\date{}

\maketitle

\begin{abstract}
	Reversible cellular automata are seen as microscopic physical models,
	and their states of macroscopic equilibrium are described using invariant probability measures.
	We establish a connection between the invariance of
	Gibbs measures
	and the conservation of additive quantities
	in surjective cellular automata.
	Namely, we show that the simplex of shift-invariant Gibbs measures
	associated to a Hamiltonian is invariant under a surjective cellular automaton
	if and only if the cellular automaton conserves the Hamiltonian.
	A special case is the (well-known) invariance of the uniform Bernoulli measure
	under surjective cellular automata, which corresponds to the conservation of the trivial Hamiltonian.
	As an application, we obtain results indicating the lack of (non-trivial) Gibbs or Markov invariant measures
	for ``sufficiently chaotic'' cellular automata.
	We discuss the relevance of the randomization property of algebraic cellular automata
	to the problem of approach to macroscopic equilibrium,
	and pose several open questions.
	
	As an aside, a shift-invariant pre-image of a Gibbs measure
	under a pre-injective factor map between shifts of finite type
	turns out to be always a Gibbs measure.
	We provide a sufficient condition under which the image of a Gibbs measure
	under a pre-injective factor map is not a Gibbs measure.
	We point out a potential application of pre-injective factor maps as a tool
	in the study of phase transitions in statistical mechanical models.
%	\keywords{cellular automata; Gibbs measures; conservation laws; macroscopic equilibrium}
%	\PACS{02.50.-r \and 05.45.-a}
%	\subclass{37B15 \and 37A60 \and 37D35 \and 37B10 \and 82Bxx \and 82Cxx}
%\end{abstract}
%\null\vfil
\renewcommand{\contentsname}{\vspace{-2em}}
{\footnotesize
\tableofcontents
}
\end{abstract}

%xxxxxxxxxxxxxxxxxxxxxxxxxxxxxxxxxxxxxxxxxxxxxxxxxxxxxxxxxxxxxxxxxxxxxxxx
\section{Introduction}
%------------------------------------------------------------------------
\label{sec:intro}

Reversible cellular automata are deterministic, spatially extended,
microscopically reversible dynamical systems.
They provide a suitable framework --- an alternative to Hamiltonian dynamics ---
to examine the dynamical foundations of
statistical mechanics with simple caricature models.
The intuitive structure of cellular automata makes them attractive to mathematicians,
and their combinatorial nature makes them amenable to perfect simulations
and computational study.

Some reversible cellular automata have long been observed, in simulations,
to exhibit ``thermodynamic behavior'': starting from a random configuration,
they undergo a transient dynamics until they reach a state of macroscopic (statistical) equilibrium.
Which of the equilibrium states the system is going to settle in
could often be guessed on the basis of few statistics of the initial configuration.

One such example is the \emph{Q2R} cellular automaton~\cite{Vic84},
which is a deterministic dynamics on top of the Ising model.
Like the standard Ising model, a configuration of the Q2R model
consists of an infinite array of symbols
$\symb{+}$ (representing an upward magnetic \emph{spin}) and
$\symb{-}$ (a downward spin)
arranged on the two-dimensional square lattice.
The symbols are updated in iterated succession of
two updating stages: at the first stage, the symbols on the even sites
(the black cells of the chess board) are updated,
and at the second stage, the symbols on the odd sites.
The updating of a symbol is performed according to a simple rule:
a spin is flipped if and only if among the four neighboring spins,
there are equal numbers of upward and downward spins.
The dynamics is clearly \emph{reversible} (changing the order of
the two stages, we could traverse backward in time).
It also conserves the Ising energy (i.e., the number of pairs of adjacent spins
that are anti-aligned).

Few snapshots from a simulation are shown in Figure~\ref{fig:ising:simulation}.
Starting with a random configuration in which the direction of each spin
is determined by a biased coin flip, the Q2R cellular automaton evolves
towards a state of apparent equilibrium that resembles
a sample from the Ising model at the corresponding temperature.\footnote{%
	The Q2R model has no temperature parameter.  A correspondence can however
	be made with the temperature at which the Ising model
	has the same expected energy density.
}
More sophisticated variants of the Q2R model show numerical agreement
with the phase diagram of the Ising model, at least away from the critical point~\cite{Cre86}.
See~\cite{TofMar87}, Chapter 17, for further simulations and an interesting discussion.
\begin{figure}
	\begin{center}
		\begin{tabular}{ccc}
			\begin{minipage}[c]{0.3\textwidth}
				\centering
				\includegraphics[width=\textwidth]{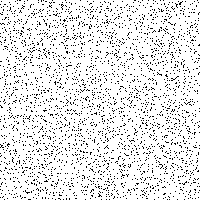}
			\end{minipage} &
			\begin{minipage}[c]{0.3\textwidth}
				\centering
				\includegraphics[width=\textwidth]{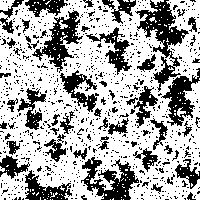}
			\end{minipage} &
			\begin{minipage}[c]{0.3\textwidth}
				\centering
				\includegraphics[width=\textwidth]{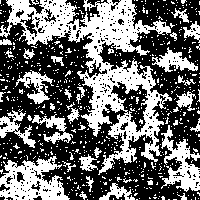}
			\end{minipage} \medskip\\
			$t=0$ & $t=1000$ & $t=10000$
		\end{tabular}
	\end{center}
	\caption{%
		Simulation of the Q2R cellular automaton on a $200\times 200$ toroidal lattice.
		Black represents an upward spin.
		The initial configuration is chosen according to a Bernoulli distribution
		with $1{\,:\,}9$ bias.
		The ``macroscopic appearance'' of the configuration does not vary
		significantly after time~$10000$.
	}
	\label{fig:ising:simulation}
\end{figure}

Wolfram was first to study cellular automata from the point of view of statistical mechanics~\cite{Wol83,Wol84}
(see also~\cite{Wol02}).
He made a detailed heuristic analysis of the so-called elementary cellular automata
(those with two states per site and local rule depending on three neighboring sites
in one dimension) using computer simulations.
One of Wolfram's observations
(the randomizing property of the XOR cellular automaton)
was mathematically confirmed by Lind~\cite{Lin84},
although the same result had also been obtained independently by Miyamoto~\cite{Miy79}.
Motivated by the problem of foundations of statistical mechanics,
Takesue made a similar study of elementary \emph{reversible} cellular automata
and investigated their ergodic properties and
thermodynamic behavior~\cite{Tak87,Tak89,Tak90}.
Recognizing the role of conservation laws in presence or absence of thermodynamic behavior,
he also started a systematic study of additive conserved quantities in cellular automata~\cite{HatTak91,Tak95}.

This article concerns the ``states of macroscopic equilibrium''
and their connection with conservation laws
in a class of cellular automata including the reversible ones.

As in statistical mechanics, we identify the ``macroscopic states'' of
lattice configurations with probability measures on the space of all such configurations.
The justification and proper interpretation of this formulation is
beyond the scope of this article.
We content ourselves with recalling two distinct points of view:
the \emph{subjective} interpretation (probability measures are meant to
describe the partial states of knowledge of an observer; see~\cite{Jay57})
and the \emph{frequentist} interpretation (a probability measure represents
a class of configurations sharing the same statistics).
See~\cite{Skl93} for comparison and discussion.
If we call tail-measurable observables ``macroscopic'', a probability measure
that is trivial on the tail events would give a full description of a macroscopic state
(see Paragraph~(7.8) of~\cite{Geo88}).
On the other hand,
restricting ``macroscopic'' observables to statistical averages
(i.e., averages of local observables over the lattice), one could identify
the macroscopic states with probability measures that are shift-invariant and ergodic.
The configurations in the ergodic set of a shift-ergodic probability measure
(i.e., the generic points in its support; see~\cite{Oxt52}) may then be considered
as ``typical'' microscopic states for the identified macroscopic state.

The interpretation of ``equilibrium'' is another unsettling issue
that we leave open.
Equilibrium statistical mechanics postulates
that the equilibrium states (of a lattice model described by interaction energies)
are suitably described by Gibbs measures (associated with the interaction energies)~\cite{Rue04,Isr79,Geo88}.
One justification (within the subjective interpretation)
is the variational principle that characterizes the shift-invariant Gibbs measures
as measures that maximize entropy under a fixed expected energy density constraint.
Within a dynamical framework, on the other hand,
the system is considered
to be in macroscopic equilibrium if its internal fluctuations
are not detected by macroscopic observables.
One is therefore tempted to identify the equilibrium states of
a cellular automaton with (tail-trivial or shift-ergodic) probability measures
that are time-invariant.
Unfortunately, there are usually an infinity of invariant measures that
do not seem to be of physical relevance.
For instance, in any cellular automaton,
the uniform distribution on the shift and time orbit of
a jointly periodic configuration is time-invariant and shift-ergodic,
but may hardly be considered a macroscopic equilibrium state.
Other conditions such as ``smoothness'' or ``attractiveness'' therefore might be needed.

Rather than reversible cellular automata (i.e., those whose trajectories
can be traced backward by another cellular automaton),
we work with the broader class of surjective cellular automata
(i.e., those that act surjectively on the configuration space).
Every reversible cellular automaton is surjective, but there are many
surjective cellular automata that are not reversible.
Surjective cellular automata are nevertheless ``almost injective'' in that
the average amount of information per site they erase in each time step is vanishing.
They are precisely those cellular automata that
preserve the uniform Bernoulli measure
(cf.\ Liouville's theorem for Hamiltonian systems).
Even if not necessarily physically relevant,
they provide a richer source of interesting examples, which could be used in case studies.
For instance, most of the known examples of the randomization phenomenon
(which, we shall argue, could provide an explanation of approach to equilibrium)
are in non-reversible surjective cellular automata.

The invariance of Gibbs measures under surjective cellular automata
turns out to be associated with their conservation laws.
More precisely, if an additive energy-like quantity,
formalized by a Hamiltonian, is conserved by a surjective cellular automaton,
the cellular automaton maps
the simplex of shift-invariant Gibbs measures corresponding to that Hamiltonian
onto itself (Theorem~\ref{thm:conservation-invariance:1}).
The converse is true in a stronger sense: if a surjective cellular automaton
maps a (not necessarily shift-invariant) Gibbs measure for a Hamiltonian to a Gibbs measure
for the same Hamiltonian, the Hamiltonian must be conserved by the cellular automaton
(Corollary~\ref{cor:conservation-invariance:1:non-shift-invariant}).
The proof of this correspondence is an immediate consequence of the variational characterization of
shift-invariant Gibbs measures and the fact that surjective cellular automata
preserve the average entropy per site of shift-invariant probability measures
(Theorem~\ref{thm:pentropy:pre-injective}).
An elementary proof of a special case was presented earlier~\cite{KarTaa11}.

Note that if a conserved Hamiltonian has a unique Gibbs measure,
then that unique Gibbs measure will be invariant under the cellular automaton.
This is the case, for example, in one dimension,
or when the Hamiltonian does not involve the interaction of
more than one site (the Bernoulli case).
An important special case is the \emph{trivial} Hamiltonian
(all configurations on the same ``energy'' level) which is obviously conserved
by every surjective cellular automaton.
The uniform Bernoulli measure is the unique Gibbs measure
for the trivial Hamiltonian, and we recover the well-known
fact that every surjective cellular automaton preserves
the uniform Bernoulli measure on its configuration space (i.e., Corollary~\ref{cor:balance:weak}).
If, on the other hand, the simplex of shift-invariant Gibbs measures
for a conserved Hamiltonian has more than one element, the cellular automaton
does not need to preserve individual Gibbs measures in this simplex
(Example~\ref{exp:ising:ca:invariant-non-invariant}).

We do not know whether, in general, a surjective cellular automaton maps
the non-shift-invariant Gibbs measures for a conserved Hamiltonian
to Gibbs measures for the same Hamiltonian, but this is known
to be the case for a proper subclass of surjective cellular automata
including the reversible ones (Theorem~\ref{thm:gibbs:factor-map:complete}), following a result of Ruelle.

The essence of the above-mentioned connection between conservation laws
and invariant Gibbs measures comes about in a more abstract setting,
concerning the pre-injective factor maps between strongly irreducible shifts of finite type.
We show that a shift-invariant pre-image of a (shift-invariant) Gibbs measure
under such a factor map is again a Gibbs measure (Corollary~\ref{thm:equilibrium:factor-map}).
We find a simple sufficient condition under which
a pre-injective factor map transforms a shift-invariant Gibbs measure
into a measure that is not Gibbs (Proposition~\ref{prop:gibbs-non-gibbs}).
An example of a surjective cellular automaton is given that
eventually transforms every starting Gibbs measure into a non-Gibbs measure
(Example~\ref{exp:xor:non-gibbs:contd}).
As an application in the study of phase transitions in equilibrium statistical mechanics,
we demonstrate how the result of Aizenman and Higuchi regarding the structure
of the simplex of Gibbs measures for the two-dimensional Ising model
could be more transparently formulated using a pre-injective factor map
(Example~\ref{exp:ising-contour}).

The correspondence between invariant Gibbs measures and conservation laws
allows us to reduce the problem of invariance of Gibbs measures
to the problem of conservation of additive quantities.
Conservation laws in cellular automata have been studied by many
from various points of view (see e.g.~\cite{Pom84,HatTak91,BocFuk98,Piv02,
DurForRok03,ForGra03,MorBocGol04,Ber07,Gar11,ForKarTaa11}).
For example, simple algorithms have been proposed to find all additive quantities
of up to a given interaction range that are conserved by a cellular automaton.
Such an algorithm can be readily applied to find all the
full-support Markov measures that are invariant under a surjective cellular automaton
(at least in one dimension).
We postpone the study of this and similar algorithmic problems to a separate occasion.

A highlight of this article is the use of this correspondence
to obtain severe restrictions on the existence of invariant Gibbs measures
in two interesting classes of cellular automata with strong chaotic behavior.
First, we show that a strongly transitive cellular automaton
cannot have any invariant Gibbs measure other than the uniform Bernoulli measure
(Corollary~\ref{cor:ca:strongly-transitive:rigidity}).
%(which we already know is invariant).
The other result concerns the class of one-dimensional reversible cellular automata
that are obtained by swapping the role of time and space in positively expansive
cellular automata.  For such reversible cellular automata, we show that
the uniform Bernoulli measure is the unique invariant Markov measure
with full support (Corollary~\ref{cor:ca:pexp-transpose:rigidity}).

Back to the interpretation of shift-ergodic probability measures
as macroscopic states, one might interpret the latter results
as an indication of ``absence of phase transitions'' in the cellular automata in question.
Much sharper results have been obtained by others for
narrower classes of cellular automata having algebraic structures
%(see~\cite{Piv09}).
(see the references in Example~\ref{exp:xor:ca}).

A mathematical description of approach to equilibrium (as observed in the Q2R example)
seems to be very difficult in general.
The randomization property of algebraic cellular automata
(the result of Miyamoto and Lind and its extensions; see Example~\ref{exp:xor:randomization})
however provides a partial explanation of approach to equilibrium
in such cellular automata.
Finding ``physically interesting'' cellular automata
with similar randomization property is an outstanding open problem.

The structure of the paper is as follows.
Section~\ref{sec:background} is dedicated to the development of the setting and background material.
Given the interdisciplinary nature of the subject, we try to be as self-contained as possible.
Basic results regarding the pre-injective factor maps between
shifts of finite type as well as two degressing applications appear in Section~\ref{sec:entropy-preserving}.
In Section~\ref{sec:ca}, we apply the results of the previous section
on cellular automata.
Conservation laws in cellular automata are discussed in Section~\ref{sec:ca:conservation-laws}.
Proving the absence of non-trivial conservation laws
in two classes of chaotic cellular automata in Section~\ref{sec:ca:rigidity}, we obtain
results regarding the rigidity of invariant measures for these two classes.
Section~\ref{sec:ca:randomization} contains a discussion of the problem of approach to equilibrium.

%xxxxxxxxxxxxxxxxxxxxxxxxxxxxxxxxxxxxxxxxxxxxxxxxxxxxxxxxxxxxxxxxxxxxxxxx
\section{Background}
%------------------------------------------------------------------------
\label{sec:background}

%========================================================================
\subsection{Observables, Probabilities, and Dynamical Systems}
%------------------------------------------------------------------------
Let $\xSp{X}$ be a compact metric space.
By an \emph{observable} we mean a Borel measurable function
$f:\xSp{X}\to\xR$.
The set of continuous observables on $\xSp{X}$
will be denoted by $C(\xSp{X})$.
This is a Banach space with the uniform norm.
The default topology on $C(\xSp{X})$
is the topology of the uniform norm.
The set of Borel probability measures on $\xSp{X}$
will be denoted by $\xSx{P}(\xSp{X})$.
The expectation operator of a Borel probability measure $\pi\in\xSx{P}(\xSp{X})$
is a positive linear (and hence continuous) functional on $C(\xSp{X})$.
Conversely, the Riesz representation theorem states that
every normalized positive linear functional on $C(\xSp{X})$
is the expectation operator of a unique probability measure on $\xSp{X}$.
Therefore, the Borel probability measures can equivalently be identified
as normalized positive linear functionals on $C(\xSp{X})$.
We assume that $\xSx{P}(\xSp{X})$ is topologized
with the weak topology.
This is the weakest topology with respect to which,
for every observable $f\in C(\xSp{X})$,
the mapping $\pi\mapsto \pi(f)$ is continuous.
The space $\xSx{P}(\xSp{X})$ under the weak topology
is compact and metrizable.
If $\delta_x$ denotes the Dirac measure concentrated at $x\in\xSp{X}$,
the map $x\mapsto\delta_x$ is an embedding of $\xSp{X}$ into $\xSx{P}(\xSp{X})$.
The Dirac measures are precisely the extreme elements of
the convex set $\xSx{P}(\xSp{X})$,
and by the Krein-Milman theorem,
$\xSx{P}(\xSp{X})$ is the closed convex hull of the Dirac measures.

Let $\xSp{X}$ and $\xSp{Y}$ be compact metric spaces
and $\Phi:\xSp{X}\to\xSp{Y}$ a continuous mapping.
We denote the induced mapping $\xSx{P}(\xSp{X})\to\xSx{P}(\xSp{Y})$
by the same symbol $\Phi$;
hence $(\Phi\pi)(E)\IsDef\pi(\Phi^{-1}E)$.
The dual map $C(\xSp{Y})\to C(\xSp{X})$
is denoted by $\Phi^*$; that is, $(\Phi^*f)(x)\IsDef f(\Phi x)$.
The following lemma is well-known.
\begin{lemma}
\label{lem:onto-one-to-one}
	Let $\xSp{X}$ and $\xSp{Y}$ be compact metric spaces
	and $\Phi:\xSp{X}\to\xSp{Y}$ a continuous map.
	\begin{enumerate}[ \rm a)]
		\item $\Phi:\xSp{X}\to\xSp{Y}$ is one-to-one
			if and only if
			$\Phi^*:C(\xSp{Y})\to C(\xSp{X})$ is onto,
			which in turn holds
			if and only if
			$\Phi:\xSx{P}(\xSp{X})\to\xSx{P}(\xSp{Y})$ is one-to-one.
		\item $\Phi:\xSp{X}\to\xSp{Y}$ is onto
			if and only if
			$\Phi^*:C(\xSp{Y})\to C(\xSp{X})$ is one-to-one,
			which in turn holds
			if and only if
			$\Phi:\xSx{P}(\xSp{X})\to\xSx{P}(\xSp{Y})$ is onto.
	\end{enumerate}
\end{lemma}
\begin{proof}\ \\
	\begin{enumerate}[ a)]
		\item
			Suppose that $\Phi:\xSp{X}\to\xSp{Y}$ is one-to-one.
			Let $g\in C(\xSp{X})$ be arbitrary.
			Define $\xSp{Y}_0\IsDef\Phi\xSp{X}$.
			Then $g\xO\Phi^{-1}:\xSp{Y}_0\to\xR$ is a continuous real-valued
			function on a closed subset of a compact metric space.
			Hence, by the Tietze extension theorem, it has an extension
			$f:\xSp{Y}\to\xR$.
			We have $f\xO\Phi=g\xO\Phi^{-1}\xO\Phi=g$.
			Therefore, $\Phi^*:C(\xSp{Y})\to C(\xSp{X})$ is onto.
			
			The other implications are trivial.
		\item
			Suppose that $\Phi:\xSp{X}\to\xSp{Y}$ is not onto.
			Let $\xSp{Y}_0\IsDef\Phi\xSp{X}$ and
			pick an arbitrary $y\in\xSp{Y}\setminus\xSp{Y}_0$.
			Using the Tietze extension theorem,
			we can find $f,f'\in C(\xSp{Y})$ such that
			$\xRest{f}{\xSp{Y}_0}=\xRest{f'}{\xSp{Y}_0}$
			but $f(y)\neq f'(y)$.
			Then $f\xO\Phi=f'\xO\Phi$.
			Hence, $\Phi^*:C(\xSp{Y})\to C(\xSp{Y})$ is not one-to-one.

			Next, suppose that $\Phi:\xSp{X}\to\xSp{Y}$ is onto.
			By the Krein-Milman theorem, the set $\xSx{P}(\xSp{Y})$
			is the closed convex hull of Dirac measures on $\xSp{Y}$.
			Let $\pi=\sum_i \lambda_i\delta_{y_i}\in\xSx{P}(\xSp{Y})$
			be a convex combination of Dirac measures.
			Pick $x_i\in\xSp{X}$ such that $\Phi x_i=y_i$, and
			define $\nu\IsDef\sum_i\lambda_i\delta_{x_i}\in\xSx{P}(\xSp{X})$.
			Then $\Phi\nu=\pi$.
			Therefore, $\Phi\xSx{P}(\xSp{X})$ is dense in $\xSx{P}(\xSp{Y})$.
			Since $\Phi\xSx{P}(\xSp{X})$ is also closed, we obtain that
			$\Phi\xSx{P}(\xSp{X})=\xSx{P}(\xSp{Y})$.
			That is,
			$\Phi:\xSx{P}(\xSp{X})\to\xSx{P}(\xSp{Y})$ is onto.
			
			The remaining implication is trivial.			
	\end{enumerate}
\end{proof}

By a \emph{dynamical system} we shall mean
a compact metric space $\xSp{X}$ together with
a continuous action $(i,x)\mapsto \varphi^i x$ of a discrete commutative finitely generated
group or semigroup $\xL$ on $\xSp{X}$.
In case of a cellular automaton, $\xL$ is the set of non-negative integers $\xN$
(or the set of integers $\xZ$ if the cellular automaton is reversible).
For a $d$-dimensional shift,
$\xL$ is the $d$-dimensional hyper-cubic lattice~$\xZ^d$.
Every dynamical system $(\xSp{X},\varphi)$ has at least one invariant measure, that is,
a probability measure $\pi\in\xSx{P}(\xSp{X})$ such that $\varphi^i\pi=\pi$ for every $i\in\xL$.
In fact, every non-empty, closed and convex subset of $\xSx{P}(\xSp{X})$ that
is closed under the application of $\varphi$ contains an invariant measure.
We will denote the set of invariant measures of $(\xSp{X},\varphi)$ by $\xSx{P}(\xSp{X},\varphi)$.

We also define $C(\xSp{X},\varphi)$
as the closed linear subspace of $C(\xSp{X})$
generated by the observables of the form
$g\xO \varphi^i - g$ for $g\in C(\xSp{X})$ and $i\in\xL$; that is,
\begin{align}
	C(\xSp{X},\varphi) &\IsDef
		\overline{
			\langle
				g\xO\varphi^i - g: g\in C(\xSp{X}) \text{ and } i\in\xL
			\rangle
		} \;
\end{align}
(see~\cite{Rue04}, Sections~4.7--4.8).
Then $\xSx{P}(\xSp{X},\varphi)$ and $C(\xSp{X},\varphi)$
are annihilators of each other:
\begin{lemma}[see e.g.~\cite{KatRob01}, Proposition~2.13]
\label{lem:annihilators}
	Let $(\xSp{X},\varphi)$ be a dynamical system.
	Then,
	\begin{enumerate}[ \rm a)]
		\item $\xSx{P}(\xSp{X},\varphi) = \left\{
				\pi\in\xSx{P}(\xSp{X}):
				\text{$\pi(f)=0$ for every $f\in C(\xSp{X},\varphi)$}
			\right\}$.
		\item $C(\xSp{X},\varphi) = \left\{
				f\in C(\xSp{X}):
				\text{$\pi(f)=0$ for every $\pi\in\xSx{P}(\xSp{X},\varphi)$}
			\right\}$.
	\end{enumerate}
\end{lemma}
\begin{proof}\ \\
	\begin{enumerate}[ a)]
		\item A probability measure $\pi$ on $\xSp{X}$
			is in $\xSx{P}(\xSp{X},\varphi)$
			if and only if
			$\pi(g\xO\varphi^i - g)=0$
			for every $g\in C(\xSp{X})$ and $i\in\xL$.
			Furthermore, for each $\pi$,
			the set $\{f\in C(\xSp{X}): \pi(f)=0\}$
			is a closed linear subspace of $C(\xSp{X})$.
			Therefore, the equality in~(a) holds.
		\item
			Let us denote the righthand side of the claimed equality in~(b) by $D$.
			The set $D$ is closed and linear,
			and contains all the elements of the form
			$g\xO\varphi^i-g$ for all $g\in C(\xSp{X})$ and $i\in\xL$.
			Therefore, $C(\xSp{X},\varphi)\subseteq D$.
			Conversely, let $f\in C(\xSp{X})\setminus C(\xSp{X},\varphi)$.
			Then every element $h\in\langle f, C(\xSp{X},\varphi)\rangle$
			has a unique representation $h=a_h f + u_h$
			where $a_h\in\xR$ and $u_h\in C(\xSp{X},\varphi)$.
			Define the linear functional
			$J:\langle f, C(\xSp{X},\varphi)\rangle\to\xR$
			by $J(h)=J(a_h f + u_h)\IsDef a_h$.
			Then $J$ is bounded
			($\norm{h}=\norm{a_h f + u_h}\geq a_h\delta=\abs{J(h)}\delta$,
			where $\delta>0$ is the distance between $f$ and $C(\xSp{X},\varphi)$),
			and hence, by the Hahn-Banach theorem, has
			a bounded linear extension $\widehat{J}$ on $C(\xSp{X})$.
			According to the Riesz representation theorem,
			there is a unique signed measure $\pi$ on $\xSp{X}$
			such that $\pi(h)=\widehat{J}(h)$ for every $h\in C(\xSp{X})$.
			Let $\pi=\pi^+-\pi^-$ be
			the Hahn decomposition of $\pi$.
			Since $\pi(f)=1$, either
			$\pi^+(f)>0$ or $\pi^-(f)>0$.
			If $\pi^+(f)>0$, define $\pi^\ast\IsDef \frac{1}{\pi^+(\xSp{X})}\pi^+$;
			otherwise $\pi^\ast\IsDef \frac{1}{\pi^-(\xSp{X})}\pi^-$.
			Then $\pi^*$ is a probability measure
			with $\pi^*(u)=0$ for every $u\in C(\xSp{X},\varphi)$,
			which according to part~(a), ensures that
			$\pi^*\in\xSx{P}(\xSp{X},\varphi)$.
			On the other hand $\pi^*(f)>0$, and hence $f\notin D$.
			We conclude that $C(\xSp{X},\varphi)=D$.
	\end{enumerate}
\end{proof}

If $K(\xSp{X})$ is a dense subspace of $C(\xSp{X})$,
the subspace $C(\xSp{X},\varphi)$ can also be expressed in terms of $K(\xSp{X})$.
Namely, if we define
\begin{align}
\label{eq:def:local-coboundaries}
	K(\xSp{X},\varphi) &\IsDef
		\langle
			h\xO\varphi^i - h: h\in K(\xSp{X}) \text{ and } i\in\xL
		\rangle \; ,
\end{align}
then $C(\xSp{X},\varphi)=\overline{K(\xSp{X},\varphi)}$.
If $D_0$ is a finite generating set for the group/semigroup $\xL$,
the subspace $K(\xSp{X},\varphi)$ may also be expressed as
\begin{align}
	\Big\{
		\sum_{i\in D_0} (h_i\xO\varphi^i - h_i) : h_i\in K(\xSp{X})
	\Big\} \;.
\end{align}
In particular, if $\xL=\xZ$ or $\xL=\xN$, then every element of $K(\xSp{X},\varphi)$
is of the form $h\xO\varphi-h$ for some $h\in K(\xSp{X})$.

A morphism between two dynamical systems $(\xSp{X},\varphi)$ and $(\xSp{Y},\psi)$
is a continuous map $\Theta:\xSp{X}\to\xSp{Y}$ such that
$\Theta\varphi=\psi\Theta$.
An epimorphism (i.e., an onto morphism) is also called a \emph{factor map}.
If $\Theta:\xSp{X}\to\xSp{Y}$ is a factor map, then $(\xSp{Y},\psi)$ is 
a \emph{factor} of $(\xSp{X},\varphi)$, and $(\xSp{X},\varphi)$ is
an \emph{extension} of $(\xSp{Y},\psi)$.
A monomorphism (i.e., a one-to-one morphism) is also known as an \emph{embedding}.
If $(\xSp{Y},\psi)$ is embedded in $(\xSp{X},\varphi)$ by the inclusion map,
$(\xSp{Y},\psi)$ is called a \emph{subsystem} of $(\xSp{X},\varphi)$.
A \emph{conjugacy} between dynamical systems is the same as an isomorphism;
two systems are said to be conjugate if they are isomorphic.

%========================================================================
\subsection{Shifts and Cellular Automata}
%------------------------------------------------------------------------
\label{sec:shifts-ca}
A cellular automaton is a dynamical system on symbolic configurations on a lattice.
The configuration space itself has translational symmetry and can be considered
as a dynamical system with the shift action.
We allow constraints on the local arrangement of symbols
to include models with so-called hard interactions,
such as the hard-core model (Example~\ref{exp:hard-core}) or the contour model (Example~\ref{exp:contour}).
Such a restricted configuration space is modeled by
a (strongly irreducible) shift of finite type.

The \emph{sites} of the $d$-dimensional (hypercubic) \emph{lattice} are indexed by the elements
of the group~$\xL\IsDef\xZ^d$.
A \emph{neighborhood} is a non-empty finite set $N\subseteq\xL$ and specifies a notion of closeness
between the lattice sites.
The \emph{$N$-neighborhood} of a site $a\in\xL$ is the set
$N(a)\IsDef a+N = \{a+i: i\in N\}$.  Likewise, the $N$-neighborhood of
a set $A\subseteq\xL$ is $N(A)\IsDef A+N \IsDef \{a+i: a\in A \text{ and } i\in N\}$.
The (symmetric) \emph{$N$-boundary} of a set $A\subseteq\xL$
is $\partial N(A)\IsDef N(A)\cap N(\xL\setminus A)$.
For a set $A\subseteq\xL$, we denote by $N^{-1}(A)\IsDef A-N$
the set of all sites $b\in\xL$ that have an $N$-neighbor in $A$.

A \emph{configuration} is an assignment $x:\xL\to S$ of symbols from a finite set $S$ to the lattice sites.
The symbol $x(i)$ assigned to a site $i\in\xL$ is also called the \emph{state} of site $i$ in $x$.
For two configurations $x,y:\xL\to S$,
we denote by $\diff(x,y)\IsDef\{i\in\xL: x(i)\neq y(i)\}$
the set of sites on which $x$ and $y$ disagree.
Two configurations $x$ and $y$ are said to be \emph{asymptotic}
(or tail-equivalent)
if $\diff(x,y)$ is finite.
If $D\subseteq\xL$ is finite, an assignment $p:D\to S$ is called a \emph{pattern} on $D$.
If $p:D\to S$ and $q:E\to S$ are two patterns (or partial configurations)
that agree on $D\cap E$, we denote by $p\lor q$ the pattern (or partial configuration)
that agrees with $p$ on $D$ and with $q$ on $E$.

Let $S$ be a finite set of symbols with at least two elements.
The set $S^\xL$ of all configurations of symbols from $S$ on $\xL$
is given the product topology, which is compact and metrizable.
The convergence in this topology is equivalent to site-wise eventual agreement.
If $D\subseteq\xL$ is a finite set and $x$ a configuration (or a partial configuration whose domain includes $D$),
the set
\begin{align}
	[x]_D &\IsDef \{y \in S^\xL: \xRest{y}{D}=\xRest{x}{D}\}
\end{align}
is called a \emph{cylinder} with \emph{base} $D$.
If $p:D\to S$ is a pattern, we may write more concisely $[p]$ rather than $[p]_D$.
In one dimension (i.e., if $\xL=\xZ$), we may also use words to specify cylinder sets:
if $u=u_0u_1\cdots u_{n-1}\in S^*$ is a word over the alphabet $S$ and $k\in\xZ$,
we write $[u]_k$ for the set of configurations $x\in S^\xZ$ such that
$x_{k+i}=u_i$ for each $0\leq i<n$.
The cylinders are clopen (i.e., both open and close) and form a basis for the product topology.
The Borel $\sigma$-algebra on $S^\xL$ is denoted by $\family{F}$.
For $A\subseteq\xL$, the sub-$\sigma$-algebra of events occurring in $A$
(i.e., the $\sigma$-algebra generated by the cylinders whose base is a subset of $A$)
will be denoted by $\family{F}_A$.

Given a configuration $x:\xL\to S$ and an element $k\in\xL$,
we denote by $\sigma^k x$ the configuration obtained by \emph{shifting} (or \emph{translating})
$x$ by vector $k$; that is, $(\sigma^k x)(i)\IsDef x(k+i)$ for every $i\in\xL$.
The dynamical system defined by the action of the shift $\sigma$
on $S^\xL$ is called the \emph{full shift}.
A closed shift-invariant set $\xSp{X}\subseteq S^\xL$ is called a \emph{shift space}
and the subsystem of $(S^\xL,\sigma)$ obtained by restricting $\sigma$ to $\xSp{X}$
is called a \emph{shift} system.  We shall use the same symbol $\sigma$ for the shift action
of all shift systems.  This will not lead to confusion, as the domain
will always be clear from the context.

A shift space $\xSp{X}\subseteq S^\xL$ is uniquely determined by its \emph{forbidden} patterns,
that is, the patterns $p:D\to S$ such that $[p]_D\cap\xSp{X}=\varnothing$.
Conversely, every set $F$ of patterns defines a shift space
by forbidding the occurrence of the elements of $F$; that is,
\begin{align}
	\xSp{X}_{F} &\IsDef S^\xL \setminus
		\big(\bigcup_{i\in\xL}\bigcup_{p\in\xSp{F}} \sigma^{-i}[p]\,\big) \;.
\end{align}
The set of patterns $p:D\to S$ that are allowed in $\xSp{X}$
(i.e., $[p]\cap\xSp{X}\neq\varnothing$) is denoted by $L(\xSp{X})$.
If $D\subseteq\xL$ is finite, we denote by $L_D(\xSp{X})\IsDef L(\xSp{X})\cap S^D$
the set of patterns on $D$ that are allowed in $\xSp{X}$.
For every pattern $p\in L_D(\xSp{X})$, there is a configuration $x\in\xSp{X}$
such that $p\lor \xRest{x}{\xL\setminus D}\in\xSp{X}$.
Given a finite set $D\subseteq\xL$ and a configuration $x\in\xSp{X}$,
we write $L_D(\xSp{X}\,|\,x)$ the set of patterns $p\in L_D(\xSp{X})$
such that $p\lor \xRest{x}{\xL\setminus D}\in\xSp{X}$.

A shift $(\xSp{X},\sigma)$ (or a shift space $\xSp{X}$)
is \emph{of finite type}, if $\xSp{X}$ can be identified
by forbidding a finite set of patterns, that is, $\xSp{X}=\xSp{X}_F$ for a finite set $F$.
The shifts of finite type have the following gluing property:
for every shift of finite type $(\xSp{X},\sigma)$,
there is a neighborhood $0\in M\subseteq\xL$ such that
for every two sets $A,B\subseteq\xL$ with $M(A)\cap M(B)=\varnothing$
and every two configurations $x,y\in\xSp{X}$ that agree outside $A\cup B$,
there is another configuration $z\in\xSp{X}$ that agrees with~$x$ outside~$B$
and with~$y$ outside~$A$.
A similar gluing property is the strong irreducibility:
a shift $(\xSp{X},\sigma)$ is \emph{strongly irreducible} if
there is a neighborhood $0\in M\subseteq\xL$ such that
for every two sets $A,B\subseteq\xL$ with $M(A)\cap M(B)=\varnothing$
and every two configurations $x,y\in\xSp{X}$,
there is another configuration $z\in\xSp{X}$ that agrees with~$x$ in $A$
and with~$y$ in~$B$.
Note that strong irreducibility is a stronger version of
topological mixing.  A dynamical system $(\xSp{X},\varphi)$
is (topologically) \emph{mixing} if for every two non-empty open sets $U,V\subseteq\xSp{X}$,
$U\cap \varphi^{-t}V\neq\varnothing$ for all but finitely many $t$.
A one-dimensional shift of finite type is strongly irreducible
if and only if it is mixing.
Our primary interest in this article will be the shifts of finite type
that are strongly irreducible, for these are sufficiently broad to encompass
the configuration space of most physically interesting lattice models.

The morphisms between shift systems are the same as the sliding block maps.
A map $\Theta:\xSp{X}\to\xSp{Y}$ between two shift spaces
$\xSp{X}\subseteq S^\xL$ and $\xSp{Y}\subseteq T^\xL$ is a \emph{sliding block map}
if there is a neighborhood $0\in N\subseteq\xL$ (a \emph{neighborhood} for $\Theta$)
and a function $\theta:L_N(\xSp{X})\to T$ (a \emph{local rule} for $\Theta$)
such that
\begin{align}
	(\Theta x)(i) &\IsDef \theta\big(\!\xRest{(\sigma^i x)}{N}\!\big)
\end{align}
for every configuration $x\in\xSp{X}$ and every site $i\in\xL$.
Any sliding block map is continuous and commutes with the shift, and hence, is a morphism.
Conversely, every morphism between shift systems is a sliding block map.
Finite type property and strong irreducibility are both conjugacy invariants.
A morphism $\Theta:\xSp{X}\to\xSp{Y}$
between two shifts $(\xSp{X},\sigma)$ and $(\xSp{Y},\sigma)$
is said to be \emph{pre-injective} if
for every two distinct asymptotic configuration $x,y\in\xSp{X}$,
the configurations $\Theta x$ and $\Theta y$ are distinct.

A \emph{cellular automaton} on a shift space $\xSp{X}$
is a dynamical system identified by
an endomorphism $\Phi:\xSp{X}\to\xSp{X}$ of $(\xSp{X},\sigma)$.
The evolution of a cellular automaton starting from a configuration $x\in\xSp{X}$
is seen as synchronous updating of the state of different sites in $x$ using
the local rule of $\Phi$.
A cellular automaton $(\xSp{X},\Phi)$ is said to be \emph{surjective}
(resp., \emph{injective}, \emph{pre-injective}, \emph{bijective})
if $\Phi$ is surjective (resp., injective, pre-injective, bijective).
If $\Phi$ is bijective, the cellular automaton is further
said to be \emph{reversible}, for $(\xSp{X},\Phi^{-1})$ is also a cellular automaton.
In this article, we only work with cellular automata that are defined over
strongly irreducible shifts of finite type.
It is well-known that for cellular automata over strongly irreducible shifts of finite type,
surjectivity and pre-injectivity are equivalent
(the \emph{Garden-of-Eden} theorem; see below).
In particular, every injective cellular automaton is also surjective, and hence reversible.

Let $\xSp{X}\subseteq S^\xL$ be a shift space.
A linear combination of characteristic functions of cylinder sets is
called a \emph{local} observable.
An observable $f:\xSp{X}\to\xR$ is local if and only if
it is $\family{F}_D$-measurable for a finite set $D\subseteq\xL$.
A finite set $D$ with such property is a \emph{base} for $f$;
the value of $f$ at a configuration $x$ can be evaluated by looking
at $x$ ``through the window $D$''.
The set of all local observables on $\xSp{X}$, denoted by $K(\xSp{X})$,
is dense in $C(\xSp{X})$.
The set of all local observables on $\xSp{X}$ with base $D$
is denoted by $K_D(\xSp{X})$.

Let $D\subseteq\xL$ be a non-empty finite set.
The \emph{$D$-block presentation} of a configuration $x:\xL\to S$
is a configuration $x^{[D]}:\xL\to S^D$, where
$x^{[D]}(i)\IsDef \xRest{x}{i+D}$.
If $\xSp{X}$ is a shift space, the set of $D$-block presentations
of the elements of $\xSp{X}$ is called
the $D$-block presentation of $\xSp{X}$,
and is denoted by $\xSp{X}^{[D]}$.
The shifts $(\xSp{X},\sigma)$ and $(\xSp{X}^{[D]},\sigma)$
are conjugate via the map $x\mapsto x^{[D]}$.

More background on shifts and cellular automata (from the view point of dynamical systems) can be found
in the books~\cite{LinMar95,Kit98,Kur03b}.

%========================================================================
\subsection{Hamiltonians and Gibbs Measures}
%------------------------------------------------------------------------
\label{sec:hamiltonian-gibbs}
We will use the term Hamiltonian in more or less
the same sense as in the Ising model or
other lattice models from statistical mechanics,
except that we do not require it to be interpreted
as ``energy''.  A Hamiltonian formalizes the concept of
a local and additive quantity, be it energy, momentum
or a quantity with no familiar physical interpretation.

Let $\xSp{X}$ be an arbitrary set.  A \emph{potential difference}
on $\xSp{X}$ is a partial mapping
$\Delta:\xSp{X}\times\xSp{X}\to\xR$ such that
\begin{enumerate}[ a)]
	\item $\Delta(x,x)=0$ for every $x\in\xSp{X}$,
	\item $\Delta(y,x)=-\Delta(x,y)$ whenever $\Delta(x,y)$ exists, and
	\item $\Delta(x,z)=\Delta(x,y)+\Delta(y,z)$
		whenever $\Delta(x,y)$ and $\Delta(y,z)$ both exist.
\end{enumerate}
Let $\xSp{X}\subseteq S^\xL$ be a shift space.
A potential difference $\Delta$ on $\xSp{X}$
is a (relative) \emph{Hamiltonian}
if
\begin{enumerate}[ a)]
\setcounter{enumi}{3}
	\item $\Delta(x,y)$ exists precisely when $x$ and $y$ are asymptotic,
	\item $\Delta(\sigma^a x,\sigma^a y)=\Delta(x,y)$
		whenever $\Delta(x,y)$ exists and $a\in\xL$, and
	\item 
		For every finite $D\subseteq\xL$, $\Delta$ is continuous when restricted
		to pairs $(x,y)$ with $\diff(x,y)\subseteq D$.
		\label{item:hamiltonian:continuous}
\end{enumerate}
Note that due to the compactness of $\xSp{X}$,
the latter continuity
is uniform among all pairs $(x,y)$ with $\diff(x,y)\subseteq D$.
If the condition~(\ref{item:hamiltonian:continuous}) is strengthened by the following condition,
we say that $\Delta$ is a \emph{finite-range} Hamiltonian.
\begin{enumerate}[ a$\!\>'$)]
\setcounter{enumi}{5}
	\item \label{item:hamiltonian:range}%
		There exists a neighborhood $0\in M\subseteq\xL$
		(the \emph{interaction} neighborhood of $\Delta$)
		such that $\Delta(x,y)$ depends only on
		the restriction of~$x$ and~$y$ to
		$M(\diff(x,y))$.
\end{enumerate}

Hamiltonians in statistical mechanics are usually constructed
by assigning interaction energies to different local arrangements
of site states.
Equivalently, they can be constructed using observables.
A local observable $f\in K(\xSp{X})$ defines a finite-range Hamiltonian
$\Delta_f$ on $\xSp{X}$ via
\begin{equation}
\label{eq:hamiltonian:observable}
	\Delta_f(x,y)\IsDef \sum_{i\in\xL}\left[
		f(\sigma^i y) - f(\sigma^i x)
	\right]
\end{equation}
for every two asymptotic configurations $x,y\in\xSp{X}$.
The value of $f\xO\sigma^i$ is then interpreted as the contribution of
site $i$ to the energy-like quantity formalized by $\Delta_f$.
The same construction works for non-local observables
that are ``sufficiently short-ranged'' (i.e., whose dependence on faraway sites decays rapidly).
The \emph{variation} of an observable $f:\xSp{X}\to\xR$
relative to a finite set $A\subseteq\xL$ is defined as
\begin{align}
	\xvar_A(f) &= \sup_{\substack{x,y\in\xSp{X}\\ \text{$x=y$ on $A$}}}\abs{f(y)-f(x)} \;,
\end{align}
where the supremum is taken over all pairs of configurations $x$ and $y$ in $\xSp{X}$
that agree on $A$.
A continuous observable $f$ is said to have \emph{summable variations}
if
\begin{align}
	\sum_{n=0}^\infty \abs{\partial I_n}\xvar_{I_n}(f) &< \infty \;,
\end{align}
where $I_n\IsDef[-n,n]^d$ and $\partial I_n\IsDef I_{n+1}\setminus I_n$.
Every observable $f$ that has summable variations
defines a Hamiltonian via~(\ref{eq:hamiltonian:observable}),
in which the sum is absolutely convergent.
We denote the set of observables with summable variations with $SV(\xSp{X})$.
Note that $K(\xSp{X})\subseteq SV(\xSp{X})\subseteq C(\xSp{X})$.

\begin{question}
	Is every Hamiltonian on a strongly irreducible shift of finite type
	generated by an observable with summable variations
	via~(\ref{eq:hamiltonian:observable})?
	Is every finite-range Hamiltonian on a strongly irreducible shift of finite type
	generated by a local observable?
\end{question}

\begin{proposition}
\label{prop:Hamiltonian:observable:local}
	Every finite-range Hamiltonian on a full shift is generated by a local observable.
\end{proposition}
\begin{proof}
	The idea is to write the Hamiltonian as a telescopic sum
	(see e.g.~\cite{HatTak91}, or~\cite{Kar05}, Section~5).
	
	Let $\Delta$ be a finite-range Hamiltonian with interaction range $M$.
	Let $\xQuiet$ be an arbitrary uniform configuration.
	Let $\preceq$ be the lexicographic order on $\xL=\xZ^d$, and denote
	by $\successor(k)$, the successor of site $k\in\xL$ in this ordering.
	For every configuration $z$ that is asymptotic to $\xQuiet$, we can write
	\begin{align}
		\Delta(\xQuiet,z) &= \sum_{k\in\xL} \Delta(z_k,z_{\successor(k)}) \;,
	\end{align}
	where $z_k$ is the configuration that agrees with $z$ on every site $i\prec k$
	and with $\xQuiet$ on every site $i\succeq k$.
	Note that all but a finite number of terms in the above sum are~$0$.

	For every configuration $z$, we define $f(z)\IsDef \Delta(z_0,z_{\successor(0)})$
	with the same definition for $z_k$ as above.
	This is clearly a local observable with base $M$.
	If $z$ is asymptotic to $\xQuiet$, the above telescopic expansion shows that
	$\Delta(\xQuiet,z)=\Delta_f(\xQuiet,z)$.
	If $x$ and $y$ are arbitrary asymptotic configurations,
	we have $\Delta(x,y)=\Delta(\hat{x},\hat{y})=\Delta(\xQuiet,\hat{y})-\Delta(\xQuiet,\hat{x})$,
	where $\hat{x}$ and $\hat{y}$ are the configurations that agree, respectively,
	with $x$ and $y$ on $M^{-1}(M(\diff(x,y)))$ and with $\xQuiet$ everywhere else.
	Therefore, we can write $\Delta(x,y)=\Delta(\hat{x},\hat{y})=\Delta_f(\hat{x},\hat{y})=\Delta_f(x,y)$.
\end{proof}
Whether the above proposition extends to finite-range Hamiltonians on
strongly irreducible shifts of finite type is not known,
but in~\cite{ChaMey13}, examples of shifts of finite type are given
on which not every finite-range Hamiltonian is generated by a local observable.
On the other hand, the main result of~\cite{ChaHanMarMeyPav13} implies that
on a one-dimensional mixing shift of finite type,
every finite-range Hamiltonian can be generated by a local observable.

The \emph{trivial} Hamiltonian on $\xSp{X}$ (i.e., the Hamiltonian $\Delta$ for which
$\Delta(x,y)=0$ for all asymptotic $x,y\in\xSp{X}$) plays a special role
as it identifies an important notion of equivalence between observables (see Section~\ref{sec:physical-equivalence}).

Another important concept regarding Hamiltonians is that of ground configurations.
Let $\xSp{X}\subseteq S^\xL$ be a shift space and $\Delta$ a Hamiltonian on $\xSp{X}$.
A \emph{ground} configuration for $\Delta$ is a configuration $z\in\xSp{X}$
such that $\Delta(z,x)\geq 0$ for every configuration $x\in\xSp{X}$ that is asymptotic to $z$.
The existence of ground configurations is well known.
We shall use it later in the proof of Theorem~\ref{thm:strongly-transitive-CA:no-cons-law}.
\begin{proposition}
\label{prop:Hamiltonian:ground}
	Every Hamiltonian on a shift space of finite type has at least one ground configuration.
\end{proposition}
\begin{proof}
	Let $\xSp{X}$ be a shift space of finite type and $\Delta$ a Hamiltonian on $\xSp{X}$.
	Let $I_1\subseteq I_2\subseteq \cdots$ be a chain of finite subsets of $\xL$
	that is exhaustive (i.e., $\bigcup_n I_n=\xL$).
	For example, we could take $I_n=[-n,n]^d$ in $\xL=\xZ^d$.
	Let $z_0\in\xSp{X}$ be an arbitrary configuration,
	and construct a sequence of configurations $z_1,z_2,\ldots\in\xSp{X}$ as follows.
	
	For each $n$, choose $z_n\in\xSp{X}$ to be a configuration with $\diff(z_0,z_n)\subseteq I_n$
	such that $\Delta(z_0,z_n)$ is minimum
	(i.e., $\Delta(z_0,z_n)\leq\Delta(z_0,x)$ for all $x\in\xSp{X}$ with $\diff(z_0,x)\subseteq I_n$).
	The minimum exists because $L_{I_n}(\xSp{X}\,|\,z_0)$ is finite.
	By compactness, there is a subsequence $n_1<n_2<\cdots$ such that
	$z_{n_i}$ converges.
	The limit $z\IsDef \lim_{i\to\infty} z_{n_i}$ is a ground configuration.
	
	To see this, let $x\in\xSp{X}$ be asymptotic to $z$, and choose $k$ such that
	$I_k\supseteq\diff(z,x)$.  Since $\xSp{X}$ is of finite type,
	there is a $l\geq k$ such that for every two configuration $u,v\in\xSp{X}$
	that agree on $I_l\setminus I_k$, there is a configuration $w\in\xSp{X}$
	that agrees with $u$ on $I_l$ and with $v$ outside $I_k$.
	In particular, for every sufficiently large $i$,
	$x$ (and $z$) agree with $z_{n_i}$ on $I_l\setminus I_k$, and
	hence there is a configuration $x_{n_i}$ that agrees with
	$x$ on $I_l$ and with $z_{n_i}$ outside $I_k$.
	Then, $\diff(z_{n_i},x_{n_i})=\diff(z,x)$.
	Since $z_{n_i}\to z$, we also get $x_{n_i}\to x$.
	The continuity property of $\Delta$ now implies that
	$\Delta(z_{n_i},x_{n_i})\to\Delta(z,x)$.
	On the other hand,
	$\Delta(z_{n_i},x_{n_i})=\Delta(z_0,x_{n_i})-\Delta(z_0,z_{n_i})\geq 0$.
	Therefore, $\Delta(z,x)\geq 0$.
\end{proof}

\begin{example}[Ising model]
\label{exp:ising}
	The \emph{Ising model} is a simple model on the lattice designed to
	give a statistical explanation of the phenomenon of
	spontaneous magnetization in ferromagnetic material (see e.g.~\cite{Vel09,Geo88}).
	The configuration space of the $d$-dimensional
	Ising model is the full shift $\xSp{X}\IsDef\{\symb{-},\symb{+}\}^{\xL}$,
	where $\xL=\xZ^d$, and
	where having $\symb{+}$ and~$\symb{-}$ at a site $i$ is interpreted as
	an upward or downward magnetization of the tiny segment of the material approximated by site $i$.
	The state of site $i$ is called the \emph{spin} at site $i$.
	
	The interaction between spins is modeled by associating an interaction energy $-1$
	to every two adjacent spins that are aligned (i.e., both are upward or both downward)
	and energy $+1$ to every two adjacent spins that are not aligned.
	Alternatively, we can specify the energy using the energy observable $f\in K(\xSp{X})$
	defined by
	\begin{align}
		f(x) &\IsDef \begin{cases}
			\frac{1}{2}(n^{\symb{-}}(x) - n^{\symb{+}}(x)) & \text{if $x(0)=\symb{+}$,} \\
			\frac{1}{2}(n^{\symb{+}}(x) - n^{\symb{-}}(x)) & \text{if $x(0)=\symb{-}$,}
		\end{cases}
	\end{align}
	where $n^{\symb{+}}(x)$ and $n^{\symb{-}}(x)$ are, respectively, the number of upward and downward
	spins adjacent to site $0$.
	This defines a Hamiltonian $\Delta_f$.
	The two uniform configurations (all sites $\symb{+}$ and all sites $\symb{-}$)
	are ground configurations for $\Delta_f$,
	although $\Delta_f$ has many other ground configurations.
	\exampleqed
\end{example}

\begin{example}[Contour model]
\label{exp:contour}
	The \emph{contour model} was originally used to study phase transition
	in the Ising model.  Each site of two-dimensional lattice $\xL=\xZ^2$
	may take a state from the set
	\begin{align}
		T &\IsDef \{
			\,\ContourEmpty\,, \,\ContourH\,, \,\ContourV\,,
			\,\ContourLD\,, \,\ContourRD\,, \,\ContourRU\,, \,\ContourLU\,, \,\ContourX\,
		\} \;.
	\end{align}
	Not all configurations are allowed.
	The allowed configurations are those in which the state of adjacent sites match
	in the obvious fashion.  For example, $\ContourLD$ can be placed on the right side of $\ContourH$
	but not on top of it, and $\ContourEmpty$ can be placed on the left side of $\ContourRU$
	but not on its right side.  The allowed configurations depict decorations of the lattice
	formed by closed or bi-infinite paths (see Figure~\ref{fig:ising-contour}b).  These
	paths are referred to as \emph{contours}.
	
	The space of allowed configurations $\xSp{Y}\subseteq T^\xZ$ is a shift space of finite type.
	It is also easy to verify that $(\xSp{Y},\sigma)$ is strongly irreducible.
	Define the local observable $g\in K(\xSp{Y})$, where
	\begin{align}
		g(x) &\IsDef \begin{cases}
			0		&\text{if $x(0)=\ContourEmpty$,}\\
			1		&\text{if $x(0)\in \{
						\,\ContourH\,, \,\ContourV\,,
						\,\ContourLD\,, \,\ContourRD\,, \,\ContourRU\,, \,\ContourLU\,
					\}$,} \\
			2		&\text{if $x(0)=\ContourX$.}
		\end{cases}
	\end{align}
	The Hamiltonian $\Delta_g$ simply compares the length of the contours in two asymptotic configurations.
	The uniform configuration in which every site is in state $\ContourEmpty$
	is a ground configuration for $\Delta_g$.
	Any configuration with a single bi-infinite horizontal (or vertical) contour is also
	a ground configuration for $\Delta_g$.
	\exampleqed
\end{example}

\begin{example}[Hard-core gas]
\label{exp:hard-core}
	Let $0\subseteq W\subseteq\xL$ be a neighborhood,
	and define a shift space $\xSp{X}\subseteq\{\symb{0},\symb{1}\}^\xL$ consisting of all configurations $x$
	for which $W(i)\cap W(j)=\varnothing$ for every distinct $i,j\in\xL$ with $x(i)=x(j)=\symb{1}$.
	This is the configuration space of the \emph{hard-core gas} model.  A site having state $\symb{1}$
	is interpreted as containing a particle, whereas a site in state $\symb{0}$
	is thought of to be empty.  It is assumed that each particle occupies a volume $W$
	and that the volume of different particles
	cannot overlap.  The one-dimensional version of the hard-core shift
	with volume $W=\{0,1\}$ is also known as the \emph{golden mean} shift.
	
	The hard-core shift is clearly of finite type.  It is also strongly irreducible.
	In fact, $\xSp{X}$ has a stronger irreducibility property:
	for every two asymptotic configurations $x,y\in\xSp{X}$,
	there is a sequence $x=x_0,x_1,\ldots,x_n=y$ of configurations in $\xSp{X}$ such that
	$\diff(x_i,x_{i+1})$ is singleton.  In particular, Proposition~\ref{prop:Hamiltonian:observable:local}
	can be adapted to cover the Hamiltonians on~$\xSp{X}$.
	
	Let $h(x)\IsDef 1$ if $x(0)=\symb{1}$ and $h(x)\IsDef 0$ otherwise.
	The Hamiltonian $\Delta_h$ compares the number of particles on two asymptotic configurations.
	The empty configuration is the unique ground configuration for $\Delta_h$.%
	\exampleqed
\end{example}

Gibbs measures are a class of probability measures identified by Hamiltonians.
Let $\xSp{X}$ be a shift space. 
A \emph{Gibbs measure} for a finite-range Hamiltonian $\Delta$ is a probability measure
$\pi\in\xSx{P}(\xSp{X})$ satisfying
\begin{align}
\label{eq:gibbs:def:local}
	\pi([y]_E) &= \xe^{-\Delta(x,y)}\,\pi([x]_E)
\end{align}
for every two asymptotic configurations $x,y\in\xSp{X}$ and
all sufficiently large $E$.
(If $M$ is the interaction neighborhood of $\Delta$,
the above equality will hold for every $E\supseteq M(\diff(x,y))$.)
More generally, if $\Delta$ is an arbitrary Hamiltonian on $\xSp{X}$,
a probability measure $\pi\in\xSx{P}(\xSp{X})$ is said to be a Gibbs measure for $\Delta$ if
\begin{align}
\label{eq:gibbs:def}
	\lim_{E\nearrow\xL} \frac{\pi([y]_E)}{\pi([x]_E)}
		&= \xe^{-\Delta(x,y)}
\end{align}
for every configuration $x\in\xSp{X}$ that is in the support of $\pi$ and every configurations $y\in\xSp{X}$
that is asymptotic to $x$.  The limit is taken along the directed family of finite subsets of $\xL$ with inclusion.\footnote{%
	The original definition of a Gibbs measure given by Dobrushin, Lanford and Ruelle 
	is via conditional probabilities (see e.g.~\cite{Geo88}).
	The definition given here can be shown to be equivalent to the original definition
	using the martingale convergence theorem. See Appendix~\ref{apx:gibbs:def:equivalence}.
}
The above limit is in fact uniform among all pairs of configurations $x,y$
in the support of $\pi$ whose disagreements $\diff(x,y)$ are included in a finite set $D\subseteq\xL$
(see Appendix~\ref{apx:gibbs:def:equivalence}).
Note also that if $\xSp{X}$ is strongly irreducible, every Gibbs measure on $\xSp{X}$ has full support,
and therefore, the relation~(\ref{eq:gibbs:def}) must hold for every two asymptotic $x,y\in\xSp{X}$.
The set of Gibbs measures for a Hamiltonian $\Delta$,
denoted by $\xSx{G}_\Delta(\xSp{X})$,
is non-empty, closed and convex.
According to the Krein-Milman theorem,
the set $\xSx{G}_\Delta(\xSp{X})$ coincides with the closed convex hull of its extremal elements.
The extremal elements of $\xSx{G}_\Delta(\xSp{X})$ are mutually singular.
The subset $\xSx{G}_\Delta(\xSp{X},\sigma)$ of shift-invariant elements
of $\xSx{G}_\Delta(\xSp{X})$ is also non-empty
(using convexity and compactness), closed and convex,
and hence equal to the closed convex hull of its extremal elements.
The extremal elements of $\xSx{G}_\Delta(\xSp{X},\sigma)$ are precisely
its ergodic elements, and hence again mutually singular.

The Gibbs measures associated to finite-range Hamiltonians
have the Markov property.
A measure $\pi$ on a shift space $\xSp{X}\subseteq S^\xL$
is called a \emph{Markov measure} if there is
a neighborhood $0\in M\subseteq\xL$ such that
for every two finite sets $D,E\subseteq\xL$ with
$M(D)\subseteq E$ and every pattern $p:E\to S$
with $\pi([p]_{E\setminus D})>0$ it holds
\begin{equation}
	\pi\left([p]_D\,\middle|\, [p]_{E\setminus D}\right) =
	\pi\left([p]_D\,\middle|\, [p]_{M(D)\setminus D}\right) \;.
\end{equation}
The data contained in the conditional probabilities
$\pi\left([p]_D\,\middle|\, [p]_{M(D)\setminus D}\right)$
for all choices of $D$, $E$ and $p$ is called
the \emph{specification} of the Markov measure $\pi$.
The specification of a Gibbs measure associated to a finite-range Hamiltonian
is \emph{positive} (i.e., all the conditional distributions are positive) and \emph{shift-invariant}.
Conversely, every positive shift-invariant
Markovian specification is the specification of a Gibbs measure.
In fact, Equation~(\ref{eq:gibbs:def:local}) identifies
a one-to-one correspondence between
finite-range Hamiltonians and the
positive shift-invariant Markovian specifications.

The uniform Bernoulli measure on a full shift $\xSp{X}\subseteq S^\xL$
is the unique Gibbs measure for the trivial Hamiltonian on $\xSp{X}$.
More generally, the shift-invariant Gibbs measures on a strongly irreducible shift of finite type
associated to the trivial Hamiltonian are precisely the measures
that maximize the entropy (see below).

We shall call a Gibbs measure \emph{regular} if
its corresponding Hamiltonian is 
generated by an observable with summable variations.

%========================================================================
\subsection{Entropy, Pressure, and The Variational Principle}
%------------------------------------------------------------------------
\label{sec:entropy-pressure}
Statistical mechanics attempts to explain the macroscopic behaviour
of a physical system by statistical analysis of its microscopic details.
In the subjective interpretation (see~\cite{Jay57}), 
the probabilities reflect the partial knowledge of an observer.
A suitable choice for a probability distribution over the possible microscopic
states of a system is therefore
one which, in light of the available partial observations, is least presumptive.

The standard approach to pick the least presumptive probability distribution
is by maximizing entropy.  The characterization of the uniform probability
distribution over a finite set
as the probability distribution that maximizes entropy is widely known.
Maximizing entropy subject to partial observations
leads to Boltzmann distribution.
The infinite systems based on lattice configurations
have a similar (though more technical) picture.
Below, we give a minimum review necessary for our discussion.
Details and more information can be found in
the original monographs and textbooks~\cite{Rue04,Isr79,Geo88,Sim93,Kel98}.
The equilibrium statistical mechanics, which can be built upon
the maximum entropy postulate, has been enormously successful
in predicting the absence or presence of phase transitions,
and in describing the qualitative features of the phases; see~\cite{Geo88}.

Let $\xSp{X}\subseteq S^\xL$ be a shift space
and $\pi\in\xSx{P}(\xSp{X})$ a 
probability measure on $\xSp{X}$.
The \emph{entropy} of a finite set $A\subseteq\xL$ of sites
under $\pi$ is
\begin{equation}
	H_\pi(A) \IsDef -\sum_{p\in L_A(\xSp{X})} \pi([p]_A)\log\pi([p]_A) \;.
\end{equation}
(By convention, $0\log 0\IsDef 0$.)
This is the same as the Shannon entropy of the random variable $\xv{x}_A$
in the probability space $(\xSp{X},\pi)$, where $\xv{x}_A$
is the projection $x\mapsto\xRest{x}{A}$.
Let us recall few basic properties of the entropy.
The entropy $H(\xv{x})$ of a random variable $\xv{x}$ is non-negative.
If $\xv{x}$ takes its values in a finite set of cardinality $n$,
then $H(\xv{x})\leq\log n$.
The entropy is sub-additive, meaning that
$H((\xv{x},\xv{y}))\leq H(\xv{x})+H(\xv{y})$ for every two random variables $\xv{x}$ and $\xv{y}$.
If $\xv{y}=f(\xv{x})$ depends deterministically on $\xv{x}$,
we have $H(f(\xv{x}))\leq H(\xv{x})$.

Let $I_n\IsDef [-n,n]^d\subseteq\xL$ be the centered
$(2n+1)\times (2n+1)\times \cdots\times (2n+1)$ box in the lattice.
If $\pi$ is shift-invariant, the sub-additivity of $A\mapsto H_\pi(A)$
ensures that the limit
\begin{equation}
	h_\pi(\xSp{X},\sigma)\IsDef
		\lim_{n\to\infty} \frac{H_\pi(I_n)}{\abs{I_n}} =
		\inf_{n\geq 0} \frac{H_\pi(I_n)}{\abs{I_n}}
\end{equation}
exists (Fekete's lemma).
The limit value $h_\pi(\xSp{X},\sigma)$ is the average entropy per site of $\pi$ over $\xSp{X}$.
It is also referred to as the \emph{(Kolmogorov-Sinai) entropy} of
the dynamical system $(\xSp{X},\sigma)$ under $\pi$ (see~\cite{Wal82}, Theorem~4.17).

The entropy functional $\pi \mapsto h_\pi(\xSp{X},\sigma)$ is non-negative
and affine.  Although it is not continuous, it is upper semi-continuous.
\begin{proposition}[Upper Semi-continuity]
	If $\lim_{i\to\infty}\pi_i=\pi$, then
	$\limsup_{i\to\infty}h_{\pi_i}(\xSp{X},\sigma) \leq h_\pi(\xSp{X},\sigma)$.
\end{proposition}
\begin{proof}
	The pointwise infimum of a family of continuous functions
	is upper semi-continuous.
\end{proof}
The entropy functional is also bounded.
Due to the compactness of $\xSx{P}(\xSp{X},\sigma)$
and the upper semi-continuity of $\pi\mapsto h_\pi(\xSp{X},\sigma)$,
the entropy $h_\pi(\xSp{X},\sigma)$ takes its maximum value
at some measures $\pi\in\xSx{P}(\xSp{X},\sigma)$.
This maximum value coincides
with the \emph{topological entropy} of the shift $(\xSp{X},\sigma)$, defined by
\begin{align}
	h(\xSp{X},\sigma)\IsDef
		\lim_{n\to\infty} \frac{\log\abs{L_{I_n}(\xSp{X})}}{\abs{I_n}} =
		\inf_{n\geq 0} \frac{\log\abs{L_{I_n}(\xSp{X})}}{\abs{I_n}} \;,
\end{align}
which is the average combinatorial entropy per site of $\xSp{X}$.

The following propositions are easy to prove, and are indeed valid
for arbitrary dynamical systems.
\begin{proposition}[Factoring]
\label{prop:entropy:factoring}
	Let $\Phi:\xSp{X}\to\xSp{Y}$ be a factor map between two shifts $(\xSp{X},\sigma)$ and $(\xSp{Y},\sigma)$
	and $\pi\in\xSx{P}(\xSp{X},\sigma)$ a probability measure on $\xSp{X}$.
	Then, $h_{\Phi\pi}(\xSp{Y},\sigma)\leq h_\pi(\xSp{X},\sigma)$.
\end{proposition}

\begin{proposition}[Embedding]
	Let $\Phi:\xSp{Y}\to\xSp{X}$ be an embedding of a shift $(\xSp{Y},\sigma)$ in a shift $(\xSp{X},\sigma)$
	and $\pi\in\xSx{P}(\xSp{Y},\sigma)$ a probability measure on $\xSp{Y}$.
	Then, $h_\pi(\xSp{Y},\sigma)= h_{\Phi\pi}(\xSp{X},\sigma)$.
\end{proposition}

Given a continuous observable $f\in C(\xSp{X})$,
the mapping $\pi\in\xSx{P}(\xSp{X},\sigma)\mapsto \pi(f)$
is continuous and affine.  Its range is closed, bounded, and convex, that is,
a finite closed interval $[e_{\min},e_{\max}]\subseteq\xR$.
For each $e\in [e_{\min},e_{\max}]$, let us define
\begin{equation}
	s_f(e) \IsDef \sup\left\{
		h_\pi(\xSp{X},\sigma):
		\pi\in\xSx{P}(\xSp{X},\sigma) \text{ and }
		\pi(f)=e
	\right\} \;.
\end{equation}
Let $\xSx{E}_{\langle f\rangle=e}(\xSp{X},\sigma)$ denote the set of
measures $\pi\in\xSx{P}(\xSp{X},\sigma)$ with
$\pi(f)=e$ and $h_\pi(\xSp{X},\sigma)=s_f(e)$,
that is, the measures $\pi$ that maximize entropy under the constraint $\pi(f)=e$. 
By the compactness of $\xSx{P}(\xSp{X},\sigma)$ and
the upper semi-continuity of $\pi\mapsto h_\pi(\xSp{X},\sigma)$,
the set $\xSx{E}_{\langle f\rangle=e}(\xSp{X},\sigma)$ is non-empty (as long as $e\in [e_{\min},e_{\max}]$).
The mapping $s_f(\cdot)$ is concave and continuous.
The measures in $\xSx{E}_{\langle f\rangle=e}(\xSp{X},\sigma)$
(and more generally, the solutions of similar entropy maximization problems
with multiple contraints $\pi(f_1)=e_1$, $\pi(f_2)=e_2$, \ldots, $\pi(f_n)=e_n$)
could be implicitly identified after a Legendre transform.

The \emph{pressure} associated to $f\in C(\xSp{X})$ could be defined as
\begin{align}
	\label{eq:pressure:def}
	P_f(\xSp{X},\sigma) &=
		\sup_{\nu\in\xSx{P}(\xSp{X},\sigma)}[
			h_\nu(\xSp{X},\sigma) - \nu(f)
		] \;.
\end{align}
The functional $f\mapsto P_f(\xSp{X},\sigma)$ is convex and Lipschitz continuous.
It is the convex conjugate of the entropy functional $\nu\mapsto h_\nu(\xSp{X},\sigma)$
(up to a negative sign), and we also have
\begin{align}
	h_\pi(\xSp{X},\sigma) &=
		\inf_{g\in C(\xSp{X})}[
			P_g(\xSp{X},\sigma) + \pi(g)
		]
\end{align}
(see~\cite{Rue04}, Theorem~3.12).
Note that the pressure $P_0(\xSp{X},\sigma)$ associated to $0$
is the same as the topological entropy of $(\xSp{X},\sigma)$.
Again, the compactness of $\xSx{P}(\xSp{X},\sigma)$
and the upper semi-continuity of $\nu\mapsto h_\nu(\xSp{X},\sigma)$
ensure that the supremum in~(\ref{eq:pressure:def}) can be achieved.
The set of shift-invariant probability measures $\pi\in\xSx{P}(\xSp{X},\sigma)$
for which the equality in
\begin{align}
	h_\pi(\xSp{X},\sigma) - P_f(\xSp{X},\sigma) \leq \pi(f)
\end{align}
is satisfied will be denoted by $\xSx{E}_f(\xSp{X},\sigma)$.
Following the common terminology of statistical mechanics and ergodic theory,
we call the elements of $\xSx{E}_f(\xSp{X},\sigma)$
the \emph{equilibrium measures} for $f$.
Let us emphasize that this terminology lacks a dynamical justification
that we are striving for.
The Bayesian justification is further clarified below.

A celebrated theorem of Dobrushin, Lanford and Ruelle
characterizes the equilibrium measures
(for ``short-ranged'' observables over strongly irreducible shift spaces of finite type)
as the associated shift-invariant Gibbs measures.
\begin{theorem}[%
	Characterization of Equilibrium Measures;
	see~\cite{Rue04}, Theorem~4.2,
	and~\cite{Kel98}, Sections~5.2 and~5.3, and~\cite{Mey13}%
]
\label{thm:equilibrium:Gibbs}
	Let $\xSp{X}\subseteq S^\xL$ be a strongly irreducible shift space of finite type.
	Let $f\in SV(\xSp{X})$ be an observable with summable variations and
	$\Delta_f$ the Hamiltonian it generates.
	The set of equilibrium measures for $f$ coincides with
	the set of shift-invariant Gibbs measures for $\Delta_f$.
\end{theorem}

Consider now an observable $f\in C(\xSp{X})$,
and as before, let $[e_{\min},e_{\max}]$ be the set of possible values $\nu(f)$
for $\nu\in\xSx{P}(\xSp{X},\sigma)$.
For every $\beta\in\xR$, we have
\begin{align}
	P_{\beta f}(\xSp{X},\sigma) &=
		\sup\{s_f(e) - \beta e: e\in [e_{\min},e_{\max}]\} \;.
\end{align}
That is, $\beta\in\xR\mapsto P_{\beta f}$ is the Legendre transform of $e\mapsto s_f(e)$.
If $f$ has summable variations and the Hamiltonian $\Delta_f$ is not trivial,
it can be shown that $\beta\mapsto P_{\beta f}$ is strictly convex
(see~\cite{Rue04}, Section~4.6, or~\cite{Isr79}, Section~III.4).
It follows that $e\mapsto s_f(e)$ is continuously differentiable everywhere
except at $e_{\min}$ and $e_{\max}$, and
\begin{align}
	s_f(e) &= \inf\{P_{\beta f}(\xSp{X},\sigma)+\beta e: \beta\in\xR\}
\end{align}
for every $e\in (e_{\min},e_{\max})$.
For $e\in (e_{\min},e_{\max})$,
the above theorem identifies the elements of
$\xSx{E}_{\langle f\rangle=e}(\xSp{X},\sigma)$
as the shift-invariant Gibbs measures for $\beta_e f$,
where $\beta_e\in\xR$ is the unique value at which $\beta\mapsto -P_{\beta f}(\xSp{X},\sigma)$
has a tangent with slope~$e$.
The mapping $e\mapsto\beta_e$ is continuous and non-increasing.
The set of slopes of tangents to $\beta\mapsto -P_{\beta f}(\xSp{X},\sigma)$ at
a point $\beta\in\xR$ is a closed interval
$[e^-_\beta,e^+_\beta]\subseteq(e_{\min},e_{\max})$.
We have
\begin{equation}
	\xSx{E}_{\beta f}(\xSp{X},\sigma)=
		\bigcup_{e\in[e^-_\beta,e^+_\beta]}
			\xSx{E}_{\langle f\rangle=e}(\xSp{X},\sigma) \;.
\end{equation}

When $f$ is interpreted as the energy contribution of
a single site, $1/\beta$ is interpreted as
the temperature and $e$ as the mean energy per site.
By a Bayesian reasoning,
if $\xSx{E}_{\langle f\rangle=e}(\xSp{X},\sigma)$ is singleton,
its unique element is an appropriate choice of
the probability distribution of the system in thermal equilibrium
when the mean energy per site is $e$.
If $\xSx{E}_{\beta f}(\xSp{X},\sigma)$ is singleton,
the unique element is interpreted as a description
of the system in thermal equilibrium at temperature $1/\beta$.
The existence of more than one element in
$\xSx{E}_{\beta f}(\xSp{X},\sigma)$ (or in $\xSx{E}_{\langle f\rangle=e}(\xSp{X},\sigma)$) is interpreted
as the existence of more than one phase (e.g., liquid or gas)
at temperature $1/\beta$ (resp., with energy density $e$).  The presence of distinct tangents to
$\beta\mapsto-P_{\beta f}(\xSp{X},\sigma)$ at a given inverse temperature $\beta$
implies the existence of distinct phases at temperature $1/\beta$
having different mean energy per site.

Note that since the elements of $\xSx{E}_{\beta f}(\xSp{X},\sigma)=\xSx{G}_{\beta\Delta_f}(\xSp{X},\sigma)$
are shift-invariant, they only offer a description of the equilibrium states
that respect the translation symmetry of the model.
By extrapolating the interpretation, one could consider
the Gibbs measures $\pi\in\xSx{G}_{\beta\Delta_f}(\xSp{X})$
that are not shift-invariant as states of equilibrium in which the translation symmetry is broken.%

%========================================================================
\subsection{Physical Equivalence of Observables}
%------------------------------------------------------------------------
\label{sec:physical-equivalence}

Let $\xSp{X}\subseteq S^{\xL}$ be
a strongly irreducible shift space of finite type.
Every local observable generates a finite-range Hamiltonian
via Equation~(\ref{eq:hamiltonian:observable}).
However, different local observables may generate the same Hamiltonians.
Two local observables $f,g\in K(\xSp{X})$
are \emph{physically equivalent} (see~\cite{Rue04}, Sections~4.6-4.7,
\cite{Isr79}, Sections~I.4 and~III.4, or~\cite{Geo88}, Section~2.4),
$f\sim g$ in symbols,
if they identify the same Hamiltonian, that is, if $\Delta_f=\Delta_g$.
The following proposition gives an alternate characterization
of physical equivalence, which will allow us to extend the notion of
physical equivalence to~$C(\xSp{X})$.
\begin{proposition} 
\label{prop:physical-equivalence:local}
	Let $(\xSp{X},\sigma)$ be a strongly irreducible shift of finite type.
	Two observables $f,g\in K(\xSp{X})$
	are physically equivalent, if and only if
	there is a constant $c\in\xR$ such that
	$\pi(f)=\pi(g)+c$ for every %shift-invariant
	probability measure $\pi\in\xSx{P}(\xSp{X},\sigma)$.
\end{proposition}
\begin{proof}\ \\
	\begin{enumerate}[ $\Rightarrow$)]
	\item[ $\Rightarrow$)]
	Let $h\IsDef f-g$.
	Let us pick an arbitrary configuration $\xQuiet\in\xSp{X}$
	with the property that the spatial average
	\begin{equation}
		c\IsDef \lim_{n\to\infty}
			\frac{\sum_{i\in I_n} h(\sigma^i \xQuiet)}{\abs{I_n}}
	\end{equation}
	(where $I_n\IsDef[-n,n]^d\subseteq\xL$)
	exists.
	That such a configuration exists follows, for example,
	from the ergodic theorem.\footnote{%
		We could simply choose $\xQuiet$ to be a periodic configuration
		if we knew such a configuration existed.
		Unfortunately, it is not known whether every strongly irreducible shift of finite type
		(in more than two dimensions)
		has a periodic configuration.
	}
	We claim that indeed
	\begin{equation}
		\lim_{n\to\infty}
			\frac{\sum_{i\in I_n} h(\sigma^i x)}{\abs{I_n}} = c
	\end{equation}
	for every configuration $x\in\xSp{X}$.
	This follows from the fact that $\Delta_h=\Delta_f-\Delta_g=0$.
	More specifically,
	let $0\in M\subseteq\xL$ be a neighborhood that witnesses
	the strong irreducibility of $\xSp{X}$,
	and let $D\subseteq\xL$ be
	a finite base for $h$ (i.e., $h$ is $\family{F}_D$-measurable).
	For each configuration $x\in\xSp{X}$ and each $n\geq 0$,
	let $x_n$ be a configuration that agrees
	with $x$ on $I_n+D$ and with $\xQuiet$ off $I_n+D+M-M$.
	Then
	\begin{align}
		\sum_{i\in I_n} h(\sigma^i x) &=
			\sum_{i\in I_n} h(\sigma^i x_n) \\
		&=
			\sum_{i\in I_n} h(\sigma^i\xQuiet) +
			\Delta_h(\xQuiet,x_n) + \smallo(\abs{I_n}) \\
		&=
			\sum_{i\in I_n} h(\sigma^i\xQuiet) + \smallo(\abs{I_n}) \;,
	\end{align}
	and the claim follows.
	Now, 
	the dominated convergence theorem
	concludes that $\pi(f)-\pi(g)=\pi(h)=c$, for every $\pi\in\xSx{P}(\xSp{X},\sigma)$.
	
	\item[ $\Leftarrow$)]
	Following the definition,
	$P_f(\xSp{X},\sigma)=P_g(\xSp{X},\sigma)-c$,
	and $f$ and $g$ have the same equilibrium measures.
	Theorem~\ref{thm:equilibrium:Gibbs} then implies that
	the shift-invariant Gibbs measures of $\Delta_f$
	and $\Delta_g$ coincide, which in turn concludes that
	$\Delta_f=\Delta_g$.
	\end{enumerate}
\end{proof}

As a corollary, the physical equivalence relation is closed in $K(\xSp{X})\times K(\xSp{X})$:
\begin{corollary}
	Let $(\xSp{X},\sigma)$ be a strongly irreducible shift of finite type.
	Let $h_1,h_2,\ldots$ be local observables on $\xSp{X}$
	such that $\Delta_{h_i}=0$ for each $i$.
	If $h_i$ converge to a local observable $h$, then $\Delta_h=0$.
\end{corollary}
The continuous extension of this relation
(i.e., the closure of $\sim$ in $C(\xSp{X})\times C(\xSp{X})$)
gives a notion of physical equivalence of arbitrary continuous observables.
\begin{proposition}
\label{prop:physical-equivalence:continuous}
	Let $(\xSp{X},\sigma)$ be a strongly irreducible shift of finite type.
	Two observables $f,g\in C(\xSp{X})$
	are physically equivalent if and only if
	there is a constant $c\in\xR$ such that
	$\pi(f)=\pi(g)+c$ for every 
	probability measure $\pi\in\xSx{P}(\xSp{X},\sigma)$.
\end{proposition}
\begin{proof}
	First, suppose that $f$ and $g$ are physically equivalent.
	Then, there exist sequences $f_1,f_2,\ldots$ and $g_1,g_2,\ldots$
	of local observables such that $f_i\to f$, $g_i\to g$ and $f_i\sim g_i$.
	By Proposition~\ref{prop:physical-equivalence:local},
	there are real numbers $c_i$ such that for every $\pi\in\xSx{P}(\xSp{X},\sigma)$,
	$\pi(f_i)-\pi(g_i)=c_i$.  Taking the limits as $i\to\infty$,
	we obtain $\pi(f)-\pi(g)=c$, where $c\IsDef\lim_i c_i$ is independent of $\pi$.
	
	Conversely, suppose there is a constant $c\in\xR$ such that
	$\pi(f)=\pi(g)+c$ for every probability measure $\pi\in\xSx{P}(\xSp{X},\sigma)$.
	Let $h\IsDef f-g-c$.
	Then, according to Lemma~\ref{lem:annihilators},
	$h\in C(\xSp{X},\sigma)$.
	Therefore, by the denseness of $K(\xSp{X})$ in $C(\xSp{X})$,
	there exists a sequence of \emph{local} observables $h_i\in C(\xSp{X},\sigma)$
	such that $h_i\to h$.
	Choose another sequence of local observables $g_i$ that converges to $g$,
	and set $f_i\IsDef h_i+g_i+c$.
	By Lemma~\ref{lem:annihilators}, $\pi(f_i)=\pi(g_i)+c$
	for every $\pi\in\xSx{P}(\xSp{X},\sigma)$, which along with Proposition~\ref{prop:physical-equivalence:local},
	implies that $\Delta_{f_i}=\Delta_{g_i}$.
	Taking the limit, we obtain that $f$ and $g$ are physically equivalent.
\end{proof}

Using Lemma~\ref{lem:annihilators}, we also get the following characterization.
\begin{corollary}
	Let $(\xSp{X},\sigma)$ be a strongly irreducible shift of finite type.
	Two observables $f,g\in C(\xSp{X})$ are physically equivalent if and only if
	$f-g-c\in C(\xSp{X},\sigma)$ for some $c\in\xR$.
\end{corollary}

Physically equivalent observables define the same
set of equilibrium measures.
Moreover, the equilibrium measures of two observables
with summable variations
that are not physically equivalent are disjoint. 
(However, continuous observables that are not physically equivalent
might in general share equilibrium measures; see~\cite{Rue04}, Corollary~3.17.)

\begin{proposition} 
\label{prop:equilibrium:equivalence}
	Let $(\xSp{X},\sigma)$ be a strongly irreducible shift of finite type.
	If two observables $f,g\in C(\xSp{X})$ are physically
	equivalent, they have the same set of equilibrium measures.
	Conversely, if two observables $f,g\in SV(\xSp{X})$
	with summable variations share an equilibrium measure, they are physically equivalent.
\end{proposition}
%\begin{proof}
%	First, let $f,g\in C(\xSp{X})$ be two physically equivalent observables.
%	Let $\pi$ and $\nu$ be equilibrium measures
%	for $f$ and $g$, respectively.
%	According to the definition, we have
%	\begin{align}
%		h_\pi - P_f &= \pi(f)\;,		& h_\nu - P_g &= \nu(g)\;, \\
%		\label{eq:equivalence:equilibrium}
%		h_\pi - P_g &\leq \pi(g)\;, 	& h_\nu - P_f &\leq \nu(f)\;.
%	\end{align}
%	Since $f$ and $g$ are physically equivalent, we have
%	$\pi(g-f)=\nu(g-f)$ and hence
%	\begin{equation}
%		P_f-P_g\leq \pi(g-f)=\nu(g-f)\leq P_f-P_g \;.
%	\end{equation}
%	Therefore, both equalities in~(\ref{eq:equivalence:equilibrium})
%	must hold.
%	That is, $\pi$ is also an equilibrium measure for $g$
%	and $\nu$ is also an equilibrium measure for $f$.
%
%	Next, let $f,g\in SV(\xSp{X})$ two observables with summable variations
%	and suppose that there is a measure $\pi\in\xSx{P}(\xSp{X},\sigma)$
%	that is an equilibrium measure for both $f$ and $g$.
%	Then, by Theorem~\ref{thm:equilibrium:Gibbs},
%	$\pi$ is a Gibbs measure for both $f$ and $g$.
%	It follows that the Hamiltonians $\Delta_f$ and $\Delta_g$
%	generated by $f$ and $g$ are equal.
%	That is, $f$ and $g$ are physically equivalent.
%\end{proof}
\begin{proof}
	The first claim is an easy consequence of the characterization of
	physical equivalence given in Proposition~\ref{prop:physical-equivalence:continuous}.
	The converse follows from the characterization of equilibrium measures
	as Gibbs measures (Theorem~\ref{thm:equilibrium:Gibbs}).
\end{proof}

%xxxxxxxxxxxxxxxxxxxxxxxxxxxxxxxxxxxxxxxxxxxxxxxxxxxxxxxxxxxxxxxxxxxxxxxx
\section{Entropy-preserving Maps}
%------------------------------------------------------------------------
\label{sec:entropy-preserving}

%========================================================================
\subsection{Entropy and Pre-injective Maps}
%------------------------------------------------------------------------

The Garden-of-Eden theorem states that
a cellular automaton over a strongly irreducible shift of finite type
is surjective if and only if it is pre-injective~\cite{Moo62,Myh63,Hed69,CecMacSca99,Fio03}.
This is one of the earliest results in the theory of cellular automata,
and gives a characterization of when a cellular automaton
has a so-called \emph{Garden-of-Eden}, that is,
a configuration with no pre-image.
The Garden-of-Eden theorem can be proved by a counting argument.
Alternatively, the argument can be phrased in terms of entropy
(see~\cite{LinMar95}, Theorem~8.1.16 and~\cite{CecCor10}, Chapter~5).

\begin{theorem}[see~\cite{LinMar95}, Theorem~8.1.16 and~\cite{MeeSte01}, Theorem~3.6]
\label{thm:entropy:pre-injective}
	Let $\Phi:\xSp{X}\to\xSp{Y}$ be a factor map
	from a strongly irreducible shift of finite type $(\xSp{X},\sigma)$
	onto a shift $(\xSp{Y},\sigma)$.
	Then, $h(\xSp{Y},\sigma)\leq h(\xSp{X},\sigma)$
	with equality if and only if $\Phi$ is pre-injective.
\end{theorem}
\begin{theorem}[see~\cite{CovPau74}, Theorem~3.3, and~\cite{MeeSte01}, Lemma~4.1, and~\cite{Fio03}, Lemma~4.4]
\label{thm:entropy:embedding}
	Let $(\xSp{X},\sigma)$ be a strongly irreducible shift of finite type
	and $\xSp{Y}\subseteq\xSp{X}$ a proper subsystem.
	Then, $h(\xSp{Y},\sigma)<h(\xSp{X},\sigma)$.
\end{theorem}

\begin{corollary}[Garden-of-Eden Theorem~\cite{Moo62,Myh63,CecMacSca99,Fio03}]
	Let $\Phi:\xSp{X}\to\xSp{X}$ be a cellular automaton
	on a strongly irreducible shift of finite type $(\xSp{X},\sigma)$.
	Then, $\Phi$ is surjective if and only if it is pre-injective.
\end{corollary}
\begin{proof}
	\begin{align}
		\begin{array}{c@{\ }c@{\ }c@{\ }c@{\ }c}
			\xSp{X} & \xrightarrow{\ \quad\Phi\ \quad} & \Phi\xSp{X} & \xrightarrow{\ \quad\subseteq\ \quad} & \xSp{X} \;, \\[1ex]
			\mathclap{h(\xSp{X},\sigma)} & \geq &
				\mathclap{h(\Phi\xSp{X},\sigma)} & \leq &
				\mathclap{h(\xSp{X},\sigma)} \;.
		\end{array}
	\end{align}
\end{proof}

Another corollary of Theorem~\ref{thm:entropy:pre-injective}
(along with Lemma~\ref{lem:onto-one-to-one} and Proposition~\ref{prop:entropy:factoring})
is the so-called \emph{balance} property of pre-injective cellular automata.
\begin{corollary}[see~\cite{CovPau74}, Theorem~2.1, and~\cite{MeeSte01}, Theorems~3.3 and~3.6]
\label{cor:balance:weak}
	Let $\Phi:\xSp{X}\to\xSp{Y}$ be a pre-injective factor map
	from a strongly irreducible shift of finite type $(\xSp{X},\sigma)$
	onto a shift $(\xSp{Y},\sigma)$.
	Every maximum entropy measure $\nu\in\xSx{P}(\xSp{Y},\sigma)$ has a
	maximum entropy pre-image $\pi\in\xSx{P}(\xSp{X},\sigma)$.
\end{corollary}
\noindent In particular, a cellular automaton on a full shift is surjective
if and only if it preserves the uniform Bernoulli measure~\cite{Hed69,MarKim76}.
In Section~\ref{sec:ca:invariance}, we shall find a generalization of this property.

The probabilistic version of Theorem~\ref{thm:entropy:pre-injective}
states that the pre-injective factor maps preserve the entropy
of shift-invariant probability measures,
and seems to be part of the folklore (see e.g.~\cite{HelLinNor07}).
\begin{theorem}
\label{thm:pentropy:pre-injective}
	Let $\Phi:\xSp{X}\to\xSp{Y}$ be a factor map
	from a shift of finite type $(\xSp{X},\sigma)$ onto a shift $(\xSp{Y},\sigma)$.
	Let $\pi\in\xSx{P}(\xSp{X},\sigma)$ be a probability measure.
	Then, $h_{\Phi\pi}(\xSp{Y},\sigma)\leq h_\pi(\xSp{X},\sigma)$
	with equality if $\Phi$ is pre-injective.
\end{theorem}
\begin{proof}
	For any factor map $\Phi$, the inequality $h_{\Phi\pi}(\xSp{Y},\sigma)\leq h_\pi(\xSp{X},\sigma)$
	holds by Proposition~\ref{prop:entropy:factoring}.  Suppose that $\Phi$ is pre-injective.
	It is enough to show that $h_{\Phi\pi}(\xSp{Y},\sigma)\geq h_\pi(\xSp{X},\sigma)$.

	Let $0\subseteq M\subseteq\xL$ be a neighborhood for $\Phi$
	and a witness for the finite-type gluing property of $\xSp{X}$ (see Section~\ref{sec:shifts-ca}).
	Let $A\subseteq\xL$ be a finite set.  By the pre-injectivity of $\Phi$, for every $x\in\xSp{X}$,
	the pattern $\xRest{x}{A\setminus M(\xCmpl{A})}$ is uniquely determined by
	$\xRest{x}{\partial M(A)}$ and $\xRest{(\Phi x)}{A}$.
	Indeed, suppose that $x'\in\xSp{X}$ is another configuration
	with $\xRest{x'}{\partial M(A)}=\xRest{x}{\partial M(A)}$
	and $\xRest{(\Phi x')}{A}=\xRest{(\Phi x)}{A}$.  Then, the configuration $x''$ that agrees with $x$
	on $M(\xCmpl{A})$ and with $x'$ on $M(A)$
	is in $\xSp{X}$ and asymptotic to~$x$.  Since $\xRest{(\Phi x'')}{A}=\xRest{(\Phi x')}{A}=\xRest{(\Phi x)}{A}$,
	it follows that $\Phi x''=\Phi x$.  Therefore, $x''=x$, and in particular,
	$\xRest{x'}{A\setminus M(\xCmpl{A})}=\xRest{x''}{A\setminus M(\xCmpl{A})}=\xRest{x}{A\setminus M(\xCmpl{A})}$.

	From the basic properties of the Shannon entropy, it follows that
	\begin{align}
		H_\pi(\partial M(A)) + H_{\Phi\pi}(A) &\geq H_\pi(A\setminus M(\xCmpl{A})) \;.
	\end{align}
	Now, choose the neighborhood $M$ to be $I_r\IsDef [-r,r]^d\subseteq\xL$ for a sufficiently large $r$.
	For $A\IsDef I_n=[-n,n]^d$, we obtain
	\begin{align}
		H_{\Phi\pi}(I_n) &\geq H_\pi(I_{n-r}) - H_\pi(I_{n+r}\setminus I_{n-r}) \;.
	\end{align}
	Dividing by $\abs{I_n}$ we get
	\begin{align}
		\frac{H_{\Phi\pi}(I_n)}{\abs{I_n}} &\geq
			\frac{\abs{I_{n-r}}}{\abs{I_n}}\cdot\frac{H_\pi(I_{n-r})}{\abs{I_{n-r}}} -
			\frac{\smallo(\abs{I_n})}{\abs{I_n}} \;,
	\end{align}
	which proves the theorem by letting $n\to\infty$.
\end{proof}

From Theorem~\ref{thm:pentropy:pre-injective} and Lemma~\ref{lem:onto-one-to-one},
it immediately follows that the functionals $f\mapsto P_f$
and $f\mapsto s_f(\cdot)$ are preserved under the dual of a pre-injective factor map.
\begin{corollary}
\label{cor:pressure:pre-injective}
	Let $\Phi:\xSp{X}\to\xSp{Y}$ be a factor map
	from a shift of finite type $(\xSp{X},\sigma)$ onto a shift $(\xSp{Y},\sigma)$.
	Let $f\in C(\xSp{Y})$ be an observable.
	Then, $P_{f\xO\Phi}(\xSp{X},\sigma)\geq P_f(\xSp{Y},\sigma)$
	with equality if $\Phi$ is pre-injective.
\end{corollary}
\begin{corollary}
\label{cor:bayesian:entropy:factor-map}
	Let $\Phi:\xSp{X}\to\xSp{Y}$ be a factor map
	from a shift of finite type $(\xSp{X},\sigma)$ onto a shift $(\xSp{Y},\sigma)$.
	Let $f\in C(\xSp{Y})$ be an observable.
	Then, $s_{f\xO\Phi}(\cdot)\geq s_f(\cdot)$
	with equality if $\Phi$ is pre-injective.
\end{corollary}
%\begin{proof}
%	Using Theorem~\ref{thm:pentropy:pre-injective}
%	and Lemma~\ref{lem:onto-one-to-one}, we have
%	\begin{align}
%		s_{f\xO\Phi}(e) &=
%			\sup\left\{
%				h_\pi(\xSp{X},\sigma): \pi\in\xSx{P}(\xSp{X},\sigma)
%				\text{ and } \pi(f\xO\Phi)=e
%			\right\} \\
%		&\geq
%			\sup\left\{
%				h_{\Phi\pi}(\xSp{Y},\sigma): \pi\in\xSx{P}(\xSp{X},\sigma)
%				\text{ and } (\Phi\pi)(f)=e
%			\right\} \\
%		&=
%			\sup\left\{
%				h_\nu(\xSp{Y},\sigma): \nu\in\xSx{P}(\xSp{Y},\sigma)
%				\text{ and } \nu(f)=e
%			\right\} \\
%		&=
%			s_f(e) \;.
%	\end{align}
%\end{proof}
%
Central to this article is the the following correspondence
between the equilibrium (Gibbs) measures of a model and its pre-injective factors.
\begin{corollary}
\label{thm:equilibrium:factor-map}
	Let $\Phi:\xSp{X}\to\xSp{Y}$ be a pre-injective factor map
	from a shift of finite type $(\xSp{X},\sigma)$ onto a shift $(\xSp{Y},\sigma)$.
	Let $f\in C(\xSp{Y})$ be an observable and
	$\pi\in \xSx{P}(\xSp{X},\sigma)$ a probability measure.
	Then $\pi\in\xSx{E}_{f\xO\Phi}(\xSp{X},\sigma)$
	if and only if
	$\Phi\pi\in\xSx{E}_f(\xSp{Y},\sigma)$.
\end{corollary}
%\begin{proof}
%	Using Theorems~\ref{cor:pressure:pre-injective}
%	and~\ref{thm:pentropy:pre-injective} we have
%	\begin{equation}
%		P_{f\xO\Phi}(\xSp{X},\sigma)-h_\pi(\xSp{X},\sigma)=\pi(f\xO\Phi)
%	\end{equation}
%	if and only if
%	\begin{equation}
%		P_f(\xSp{Y},\sigma)-h_{\Phi\pi}(\xSp{Y},\sigma)=(\Phi\pi)(f) \;.
%	\end{equation}
%\end{proof}
\begin{corollary}
\label{thm:bayesian:solutions:factor-map}
	Let $\Phi:\xSp{X}\to\xSp{Y}$ be a pre-injective factor map
	from a shift of finite type $(\xSp{X},\sigma)$ onto a shift $(\xSp{Y},\sigma)$.
	Let $f\in C(\xSp{Y})$ be an observable,
	$\pi\in \xSx{P}(\xSp{X},\sigma)$ a probability measure,
	and $e\in\xR$.
	Then, $\pi\in\xSx{E}_{\langle f\xO\Phi\rangle=e}(\xSp{X},\sigma)$
	if and only if
	$\Phi\pi\in\xSx{E}_{\langle f\rangle=e}(\xSp{Y},\sigma)$.
\end{corollary}
%\begin{proof}
%	Using Theorem~\ref{thm:pentropy:pre-injective}
%	and Corollary~\ref{cor:bayesian:entropy:factor-map} we have
%	\begin{equation}
%		h_\pi(\xSp{X},\sigma)=s_{f\xO\Phi}(e)
%		\qquad\text{and}\qquad
%		\pi(f\xO\Phi)=e
%	\end{equation}
%	if and only if
%	\begin{equation}
%		h_{\Phi\pi}(\xSp{X},\sigma)=s_f(e)
%		\qquad\text{and}\qquad
%		(\Phi\pi)(f)=e \;.
%	\end{equation}
%\end{proof}

\begin{example}
	Let $\xSp{X}\subseteq\{\symb{0},\symb{1},\symb{2}\}^\xZ$ be the shift obtained by forbidding $\symb{1}\symb{1}$
	and $\symb{2}\symb{2}$, and $\xSp{Y}\subseteq\{\symb{0},\symb{1},\symb{2}\}^\xZ$ the shift obtained
	by forbidding $\symb{2}\symb{2}$ and $\symb{2}\symb{1}$.
	Then, both $(\xSp{X},\sigma)$ and $(\xSp{Y},\sigma)$ are mixing shifts of finite type.
	For every configuration $x\in\xSp{X}$, let $\Phi x\in\{\symb{0},\symb{1},\symb{2}\}^\xZ$
	be the configuration in which
	\begin{align}
		(\Phi x)(i) &= \begin{cases}
			\symb{1}	\qquad	& \text{if $x(i)=\symb{2}$ and $x(i+1)=\symb{1}$,}\\
			x(i)				& \text{otherwise.}
		\end{cases}
	\end{align}
	Then, $\Phi$ is a pre-injective factor map from $\xSp{X}$ onto $\xSp{Y}$.
	
	Consider the observables $g_0,g_1,g_2:\xSp{Y}\to\xR$ defined by
	\begin{align}
		g_0(y) &\IsDef 1_{[\symb{0}]_0}(y)\;, & g_1(y) &\IsDef 1_{[\symb{1}]_0}(y)\;,
			&	g_2(y) &\IsDef 1_{[\symb{2}]_0}(y)
	\end{align}
	for $y\in\xSp{Y}$.
	The Hamiltonians $\Delta_{g_0}$, $\Delta_{g_1}$ and $\Delta_{g_2}$ count
	the number of $\symb{0}$s, $\symb{1}$s and $\symb{2}$s, respectively.
	The unique Gibbs measures for $\Delta_{g_0}$, $\Delta_{g_1}$ and $\Delta_{g_2}$
	are, 
	respectively, the distribution of the bi-infinite Markov chains with transition matrices\\
	\begin{align}
		A_0 &\approx
			{\small%
			\begin{matrix}
				\tsymb{0} \\ \tsymb{1} \\ \tsymb{2}
			\end{matrix}
			\begin{bmatrix}
				\OverLabel{0.230}{\symb{0}}
								& \OverLabel{0.626}{\symb{1}}
												& \OverLabel{0.144}{\symb{2}}	\\
				0.230			& 0.626			& 0.144	\\
				1				& 0				& 0	
			\end{bmatrix}
			} \;,
		&
		A_1 &\approx
			{\small%
			\begin{matrix}
				\tsymb{0} \\ \tsymb{1} \\ \tsymb{2}
			\end{matrix}
			\begin{bmatrix}
				\OverLabel{0.528}{\symb{0}}
								& \OverLabel{0.194}{\symb{1}}
												& \OverLabel{0.278}{\symb{2}}	\\
				0.528			& 0.194			& 0.278	\\
				1				& 0				& 0	
			\end{bmatrix}
			} \;,
		&
		A_2 &\approx
			{\small%
			\begin{matrix}
				\tsymb{0} \\ \tsymb{1} \\ \tsymb{2}
			\end{matrix}
			\begin{bmatrix}
				\OverLabel{0.461}{\symb{0}}
								& \OverLabel{0.461}{\symb{1}}
												& \OverLabel{0.078}{\symb{2}}	\\
				0.461			& 0.461			& 0.078	\\
				1				& 1				& 0	
			\end{bmatrix}
			} \;.
	\end{align}
	In general, every finite-range Gibbs measure on a one-dimensional %shift space
	mixing shift of finite type
	is the distribution of a bi-infinite Markov chain and vice versa
	(see~\cite{Geo88}, Theorem~3.5 and~\cite{ChaHanMarMeyPav13}).
	The observables induced by $g_0$, $g_1$ and $g_2$ on $\xSp{X}$ via $\Phi$ satisfy
	\begin{align}
		(g_0\xO\Phi)(x) &= 1_{[\symb{0}]_0}(x)\;, & (g_1\xO\Phi)(x) &\IsDef
			(1_{[\symb{1}]_0}+1_{[\symb{2}\symb{1}]_0})(x)\;,
			&	(g_2\xO\Phi)(x) &\IsDef 1_{[\symb{2}\symb{0}]_0}(x)\;.
	\end{align}
	for every $x\in\xSp{X}$.
	The unique Gibbs measures for $\Delta_{g_0\xO\Phi}$,
	$\Delta_{g_1\xO\Phi}$ and $\Delta_{g_2\xO\Phi}$ are, respectively,
	the distribution of the bi-infinite Markov chains with transition matrices\\
	\begin{align}
		A_0 &\approx
			{\small%
			\begin{matrix}
				\tsymb{0} \\ \tsymb{1} \\ \tsymb{2}
			\end{matrix}
			\begin{bmatrix}
				\OverLabel{0.230}{\symb{0}}
								& \OverLabel{0.385}{\symb{1}}
												& \OverLabel{0.385}{\symb{2}}	\\
				0.374			& 0				& 0.626	\\
				0.374			& 0.626			& 0	
			\end{bmatrix}
			} \;,
		&
		A_1 &\approx
			{\small%
			\begin{matrix}
				\tsymb{0} \\ \tsymb{1} \\ \tsymb{2}
			\end{matrix}
			\begin{bmatrix}
				\OverLabel{0.528}{\symb{0}}
								& \OverLabel{0.163}{\symb{1}}
												& \OverLabel{0.310}{\symb{2}}	\\
				0.630			& 0				& 0.370	\\
				0.898			& 0.102			& 0	
			\end{bmatrix}
			} \;,
		&
		A_2 &\approx
			{\small%
			\begin{matrix}
				\tsymb{0} \\ \tsymb{1} \\ \tsymb{2}
			\end{matrix}
			\begin{bmatrix}
				\OverLabel{0.461}{\symb{0}}
								& \OverLabel{0.316}{\symb{1}}
												& \OverLabel{0.224}{\symb{2}}	\\
				0.673			& 0				& 0.327	\\
				0.350			& 0.650			& 0	
			\end{bmatrix}
			} \;.
	\end{align}
	By Corollary~\ref{thm:equilibrium:factor-map},
	we have $\Phi\pi_0=\nu_0$, $\Phi\pi_1=\nu_1$, and $\Phi\pi_2=\nu_2$.
	\exampleqed
\end{example}

%========================================================================
\subsection{Complete Pre-injective Maps}
%------------------------------------------------------------------------

In this section, we discuss the extension of %the correspondence given in
Corollary~\ref{thm:equilibrium:factor-map} to the case of
non-shift-invariant Gibbs measures
(Conjecture~\ref{conj:gibbs:factor-map}, Theorem~\ref{thm:gibbs:factor-map:complete},
and Corollary~\ref{cor:gibbs:conjugacy}).
We start with an example that deviates from the main line of this article
(i.e., understanding macroscopic equilibrium in surjective cellular automata)
but rather demonstrates an application of factor maps as a tool in the study of
phase transitions in equilibrium statistical mechanics models.
The (trivial) argument used in this example however serves as a model
for the proof of Theorem~\ref{thm:gibbs:factor-map:complete}.

\begin{example}[Ising and contour models]
\label{exp:ising-contour}
	There is a natural correspondence between the two-dimensional Ising model (Example~\ref{exp:ising})
	and the contour model (Example~\ref{exp:contour}).

	As before, let $\xSp{X}=\{\symb{+},\symb{-}\}^{\xZ^2}$ and $\xSp{Y}\IsDef T^{\xZ^2}$
	denote the configuration spaces of the Ising model and the contour model.
	Define a sliding block map $\Theta:\xSp{X}\to T^{\xZ^2}$
	with neighborhood $N\IsDef\{(0,0),(0,1),(1,0),(1,1)\}$ and local rule $\theta:\{\symb{+},\symb{-}\}^N\to T$,
	specified by 
	{\small%
	\begin{align}
		\begin{array}{%
			c@{\>\mapsto\>}c|c@{\>\mapsto\>}c|c@{\>\mapsto\>}c|c@{\>\mapsto\>}c
			|c@{\>\mapsto\>}c|c@{\>\mapsto\>}c|c@{\>\mapsto\>}c|c@{\>\mapsto\>}c
		}
			%\hline
			\SymbCellBlock{+}{+}{+}{+}&\ContourEmpty		&	\SymbCellBlock{+}{+}{+}{-}&\ContourRD		&
			\SymbCellBlock{+}{+}{-}{+}&\ContourLD			&	\SymbCellBlock{+}{+}{-}{-}&\ContourH		&
			\SymbCellBlock{+}{-}{+}{+}&\ContourRU			&	\SymbCellBlock{+}{-}{+}{-}&\ContourV		&
			\SymbCellBlock{+}{-}{-}{+}&\ContourX			&	\SymbCellBlock{+}{-}{-}{-}&\ContourLU		\\ \hline
			\SymbCellBlock{-}{-}{-}{-}&\ContourEmpty		&	\SymbCellBlock{-}{-}{-}{+}&\ContourRD		&
			\SymbCellBlock{-}{-}{+}{-}&\ContourLD			&	\SymbCellBlock{-}{-}{+}{+}&\ContourH		&
			\SymbCellBlock{-}{+}{-}{-}&\ContourRU			&	\SymbCellBlock{-}{+}{-}{+}&\ContourV		&
			\SymbCellBlock{-}{+}{+}{-}&\ContourX			&	\SymbCellBlock{-}{+}{+}{+}&\ContourLU		%\\ \hline
		\end{array}
	\end{align}
	}%
	(see Figure~\ref{fig:ising-contour}).
	Then, $\Theta$ is a factor map onto $\xSp{Y}$
	and is pre-injective.
	In fact, $\Theta$ is $2$-to-$1$: every configuration $y\in\xSp{Y}$
	has exactly two pre-images $x,x'\in\xSp{X}$,
	where $x'=-x$ (i.e., $x'$ is obtained from $x$ by flipping the direction of the spin at every site).
	Moreover, if $f$ denotes the energy observable for the Ising model
	and $g$ the contour length observable for the contour model,
	we have 
	$\Delta_f=2\Delta_{g\xO\Theta}$.
	\begin{figure}
		\begin{center}
			\begin{tabular}{ccc}
				\begin{minipage}[c]{0.4\textwidth}
					\centering
					\includegraphics[width=0.95\textwidth]{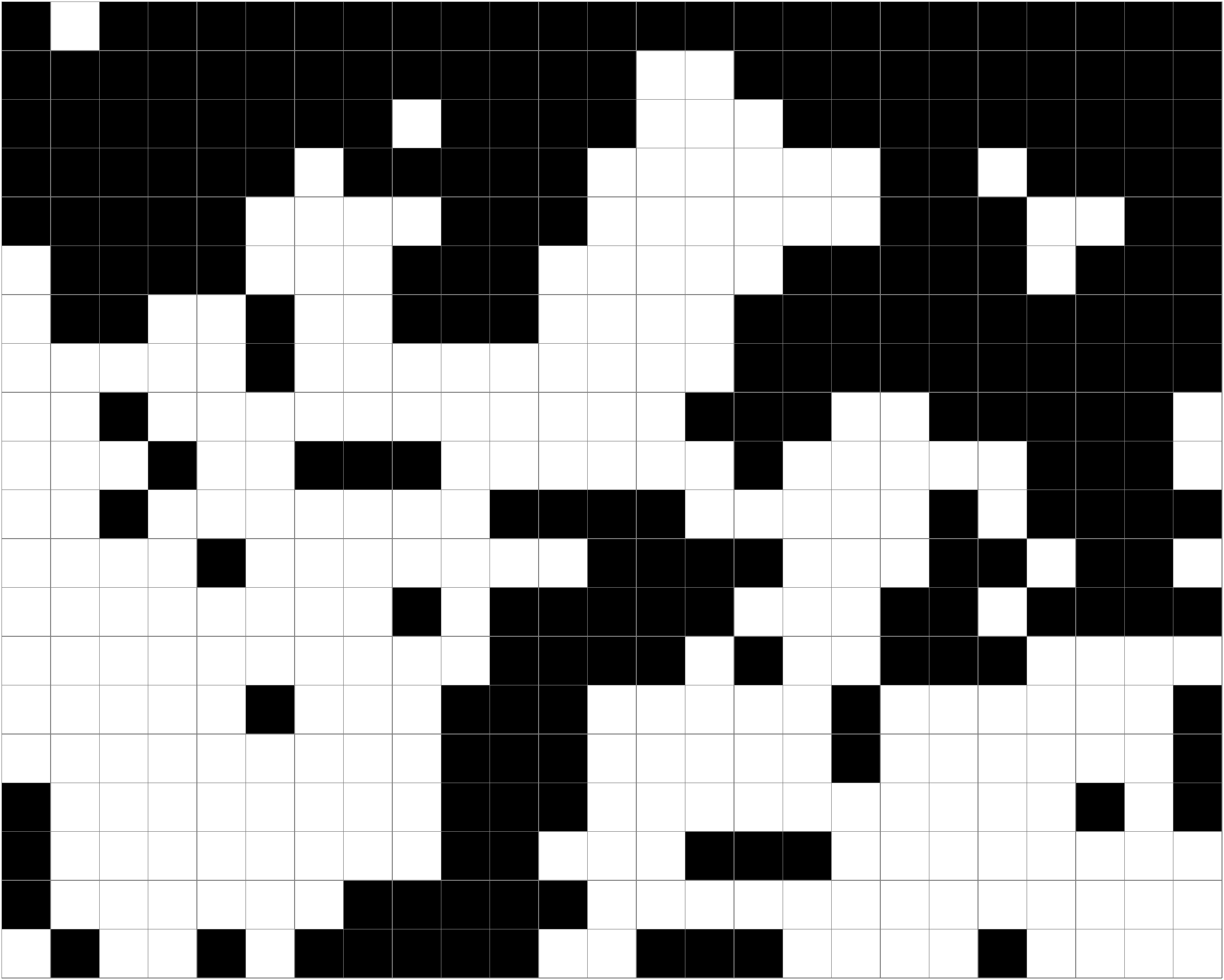}
				\end{minipage} & &
				\begin{minipage}[c]{0.4\textwidth}
					\centering
					\includegraphics[width=0.912\textwidth]{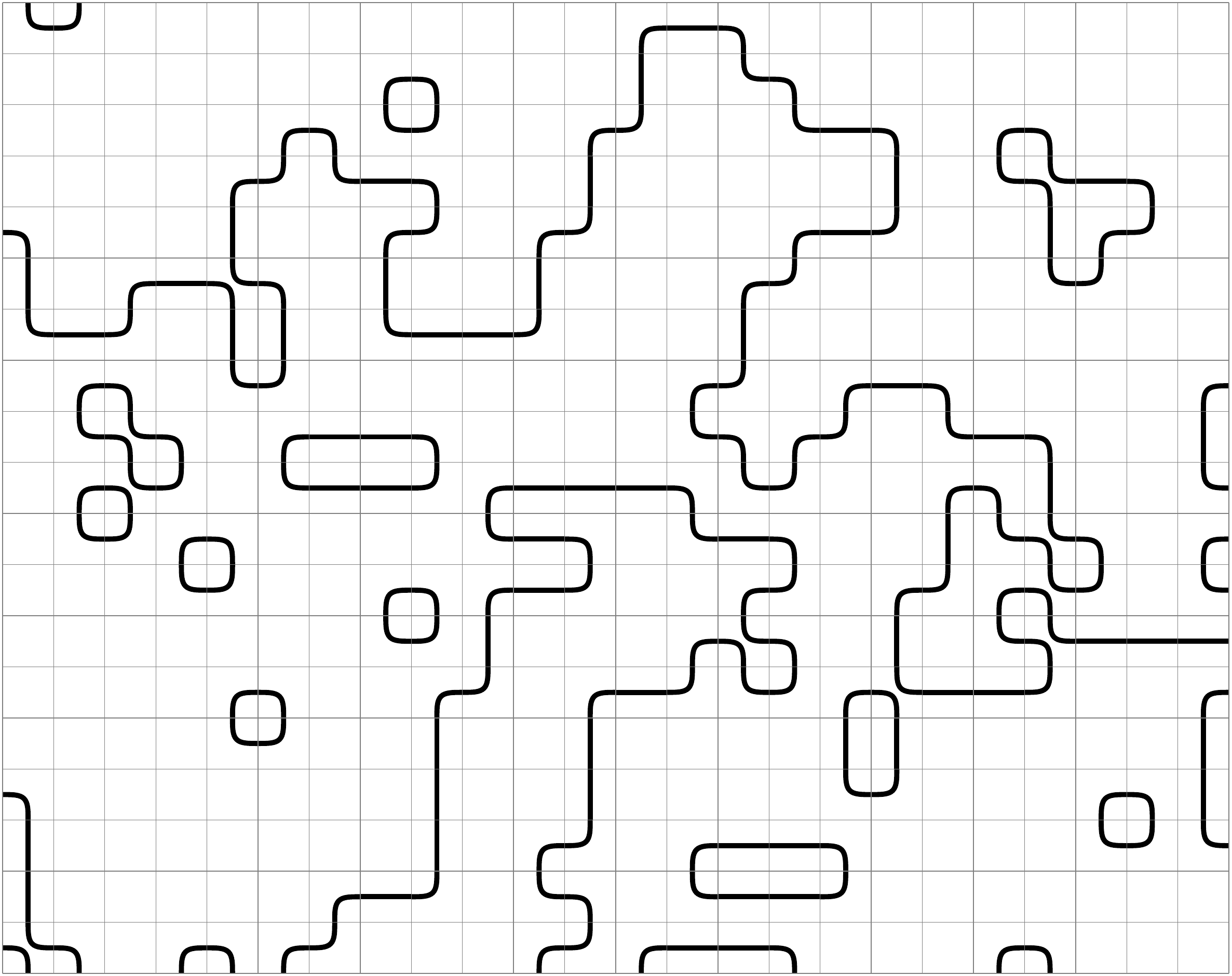}		% 0.95*24/25
				\end{minipage} \medskip\\
				(a) && (b)
			\end{tabular}
		\end{center}
		\caption{%
			A configuration of the Ising model (a)
			and its corresponding contour configuration (b).
			Black represents an upward spin.
		}
		\label{fig:ising-contour}
	\end{figure}
	
	This relationship, which was first discovered by Peierls~\cite{Pei36}, is used to reduce the study of
	the Ising model to the study of the contour model (see e.g.~\cite{Gri64,Vel09}).
	The Gibbs measures for $\beta\Delta_f$ represent the states of thermal equilibrium
	for the Ising model at temperature $1/\beta$.
	According to Corollary~\ref{thm:equilibrium:factor-map} (and Theorem~\ref{thm:equilibrium:Gibbs}),
	the shift-invariant Gibbs measures $\pi\in\xSx{G}_{\beta\Delta_f}(\xSp{X},\sigma)$ are precisely
	the $\Theta$-pre-images of the shift-invariant Gibbs measures
	$\nu\in\xSx{G}_{2\beta\Delta_g}(\xSp{Y},\sigma)$
	for the contour model.

	In fact, in this case it is also easy to show that the $\Theta$-image of every Gibbs measure for $\beta\Delta_f$
	(not necessarily shift-invariant) is a Gibbs measure for $2\beta\Delta_g$.
	Indeed, suppose that $\pi\in\xSx{G}_{\beta\Delta_f}(\xSp{X})$ is a Gibbs measure for $\beta\Delta_f$
	and $\nu\IsDef\Theta\pi$ its image.  Let $y,y'\in\xSp{Y}$ be asymptotic configurations,
	and $E\supseteq\diff(y,y')$ a sufficiently large finite set of sites.
	Let $x_1,x_2\in\xSp{X}$ be the pre-images of $y$, and $x'_1,x'_2\in\xSp{X}$ the pre-images of $y'$.
	Without loss of generality, we can assume that
	$x_1$ is asymptotic to $x'_1$, and $x_2$ is asymptotic $x'_2$.
	It is easy to see that $\Theta^{-1}[y]_E=[x_1]_{N(E)}\cup[x_2]_{N(E)}$
	and $\Theta^{-1}[y']_E=[x'_1]_{N(E)}\cup[x'_2]_{N(E)}$.
	Note that the cylinders $[x_1]_{N(E)}$ and $[x_2]_{N(E)}$ are disjoint,
	and so are the cylinders $[x'_1]_{N(E)}$ and $[x'_2]_{N(E)}$.
	Therefore,
	\begin{align}
		\frac{\nu([y']_E)}{\nu([y]_E)} &= \frac{\pi(\Theta^{-1}[y']_E)}{\pi(\Theta^{-1}[y]_E)} =
			\frac{\pi([x'_1]_{N(E)}) + \pi([x'_2]_{N(E)})}{\pi([x_1]_{N(E)}) + \pi([x_2]_{N(E)})} \;.
	\end{align}
	Since $N(E)$ is large and $\pi$ is a Gibbs measure for $\beta\Delta_f$, we have
	$\pi([x'_1]_{N(E)})=\xe^{-\beta\Delta_f(x_1,x'_1)}\pi([x_1]_{N(E)})$ and
	$\pi([x'_2]_{N(E)})=\xe^{-\beta\Delta_f(x_2,x'_2)}\pi([x_2]_{N(E)})$.
	Since, $\Delta_f(x_1,x'_1)=\Delta_f(x_2,x'_2)=2\Delta_g(y,y')$,
	it follows that $\nu([y']_E)=\xe^{-2\beta\Delta_g(y,y')}\nu([y]_E)$.
	
	It has been proved that for any $0<\beta<\infty$, the contour model with Hamiltonian $2\beta\Delta_g$
	has a unique Gibbs measure~\cite{Aiz80,Hig81}; the main difficulty is to show that the infinite contours
	are ``unstable'', in the sense that, under every Gibbs measure, the probability of appearance
	of an infinite contour is zero.\footnote{%
		In fact, the theorem of Aizenman and Higuchi states that the simplex of
		Gibbs measures for the two-dimensional Ising model at any temperature has at most two extremal elements.
		However, the uniqueness of the Gibbs measure for the contour model is implicit
		in their result, and constitutes the main ingredient of the proof.
	}  Let us denote the unique Gibbs measure for $2\beta\Delta_g$
	by $\nu_\beta$.  It follows that the simplex of Gibbs measures for the Ising model
	at temperature $1/\beta$ is precisely $\Theta^{-1}\nu_\beta$.
	For, the set $\Theta^{-1}\nu_\beta$ includes $\xSx{G}_{\beta\Delta_f}(\xSp{X})$ (by the above observation)
	and is included in $\xSx{G}_{\beta\Delta_f}(\xSp{X},\sigma)$ (because $\nu_\beta$ must be shift-invariant).
	Therefore, the Gibbs measures for the Ising model at any temperature $1/\beta$ are shift-invariant
	and $\xSx{G}_{\beta\Delta_f}(\xSp{X})=\xSx{G}_{\beta\Delta_f}(\xSp{X},\sigma)=\Theta^{-1}\nu_\beta$.
	
	It is not difficult to show that if $\Phi:\xSp{X}\to\xSp{Y}$ is a continuous $k$-to-$1$ map
	between two compact metric spaces, then every probability measure on $\xSp{Y}$
	has at most $k$ mutually singular pre-images under $\Phi$.
	In particular, the simplex
	$\xSx{G}_{\beta\Delta_f}(\xSp{X})=\xSx{G}_{\beta\Delta_f}(\xSp{X},\sigma)=\Theta^{-1}\nu_\beta$
	of Gibbs measures for the Ising model at temperature $1/\beta$ has at most $2$~ergodic elements.
	
	Whether the Ising model at temperature $1/\beta$ has two ergodic Gibbs measures or one
	depends on a specific geometric feature of the typical contour configurations under the measure~$\nu_\beta$.
	Roughly speaking, the contours of a contour configuration divide the two-dimensional plane into disjoint clusters.
	A configuration with no infinite contour generates either one or no infinite cluster,
	depending on whether each site is surrounded by a finite or infinite number of contours.
	Note that since $\nu_\beta$ is ergodic, the number of infinite clusters
	in a random configuration chosen according to $\nu_\beta$ is almost surely constant.
	If $\nu_\beta$-almost every configuration has an infinite cluster, then it follows by symmetry
	that $\Theta^{-1}\nu_\beta$ contains two distinct ergodic measures,
	one in which the infinite cluster is colored with~$\symb{+}$ and one with~$\symb{-}$.
	The converse is also known to be true~\cite{Rus79}:
	if $\nu_\beta$-almost every configuration has no infinite cluster, then
	$\Theta^{-1}\nu_\beta$ has only one element.
	
	Contour representations are used to study a wide range of statistical mechanics models,
	and are particularly fruitful to prove the ``stability'' of ground configurations
	at low temperature (see e.g.~\cite{Sin82,Fer98}).
	\exampleqed
\end{example}

Let $\Phi:\xSp{X}\to\xSp{Y}$ be a pre-injective factor map
between two strongly irreducible shifts of finite type $(\xSp{X},\sigma)$ and $(\xSp{Y},\sigma)$.
Let $f\in SV(\xSp{Y})$ be an observable having summable variations
and $\Delta_f$ the Hamiltonian defined by $f$.
Then, according to Theorem~\ref{thm:equilibrium:Gibbs},
the equilibrium measures of $f$ and $f\xO\Phi$ are precisely
the shift-invariant Gibbs measures
for the Hamiltonians $\Delta_f$ and $\Delta_{f\xO\Phi}$.
A natural question is whether Corollary~\ref{thm:equilibrium:factor-map}
remains valid for arbitrary Gibbs measures (not necessarily shift-invariant).
If $\Phi:\xSp{X}\to\xSp{Y}$ is a morphism between two shifts
and $\Delta$ is a Hamiltonian on $\xSp{Y}$, let us denote by $\Phi^*\Delta$,
the Hamiltonian on $\xSp{X}$ defined by $(\Phi^*\Delta)(x,y)\IsDef \Delta(\Phi x,\Phi y)$.
\begin{conjecture}
\label{conj:gibbs:factor-map}
	Let $\Phi:\xSp{X}\to\xSp{Y}$ be a pre-injective factor map
	between two strongly irreducible shifts of finite type $(\xSp{X},\sigma)$ and $(\xSp{Y},\sigma)$.
	Let $\Delta$ be a Hamiltonian on $\xSp{Y}$, and $\pi$ a probability measure on $\xSp{X}$.
	Then, $\pi$ is a Gibbs measure for~$\Phi^*\Delta$
	if and only if $\Phi\pi$ is a Gibbs measure for~$\Delta$.
\end{conjecture}

One direction of the latter conjecture is known to be true for
a subclass of pre-injective factor maps.
Let us say that a pre-injective factor map $\Phi:\xSp{X}\to\xSp{Y}$
between two shifts is \emph{complete} if for every configuration $x\in\xSp{X}$ and every
configuration $y'\in\xSp{Y}$ that is asymptotic to $y\IsDef\Phi x$,
there is a (unique) configuration $x'\in\xSp{X}$ asymptotic to $x$
such that $\Phi x'= y'$.
\begin{lemma}
\label{lem:pre-injective:complete}
	Let $\Phi:\xSp{X}\to\xSp{Y}$ be a complete pre-injective factor map
	between two shifts of finite type $(\xSp{X},\sigma)$ and $(\xSp{Y},\sigma)$.
	For every finite set $D\subseteq\xL$, there is a finite set $E\subseteq\xL$
	such that every two asymptotic configurations $x,x'\in\xSp{X}$
	with $\diff(\Phi x,\Phi x')\subseteq D$ satisfy $\diff(x,x')\subseteq E$.
\end{lemma}
\begin{proof}
	See Figure~\ref{fig:lemma:complete-pre-injective} for an illustration.

	For a configuration $x\in\xSp{X}$, let $\event{A}_x$ be the set of all configurations
	$x'$ asymptotic to $x$ such that $\diff(\Phi x,\Phi x')\subseteq D$.
	The set $\event{A}_x$ is finite.  Therefore, there is a finite set $E_x$ such that
	all the elements of $\event{A}_x$ agree outside $E_x$.
	We claim that if $C_x\supseteq E_x$ is a large enough finite set of sites,
	then for every configuration $x_1\in [x]_{C_x}$,
	all the elements of $\event{A}_{x_1}$ agree outside $E_x$.
	
	To see this, suppose that $C_x$ is large, and consider a configuration $x_1\in[x]_{C_x}$.
	Let $x'_1$ be a configuration asymptotic to $x_1$ such that $\diff(\Phi x_1,\Phi x'_1)\subseteq D$.
	By the gluing property of $\xSp{Y}$, there is a configuration $y'\in\xSp{Y}$
	that agrees with $\Phi x'_1$ in a large neighborhood of $D$ and with $\Phi x$ outside~$D$.
	Since $\Phi$ is a complete pre-injective factor map, there is a unique configuration $x'$
	asymptotic to $x$ such that $\Phi x'=y'$.
	Now, by the gluing property of $\xSp{X}$, there is a configuration $x''_1$
	that agrees with $x'$ in $C_x$ and with $x_1$ outside $E_x$.
	Since $C_x$ was chosen large, it follows that $\Phi x''_1=\Phi x'_1$.
	Since $x'_1$ and $x''_1$ are asymptotic, the pre-injectivity of $\Phi$ ensures that $x''_1=x'_1$.
	Therefore, $x_1$ and $x'_1$ agree outside $E_x$.
	
	The cylinders $[x]_{C_x}$ form an open cover of $\xSp{X}$.
	Therefore, by the compactness of $\xSp{X}$, there is a finite set $\event{I}\subseteq\xSp{X}$
	such that $\bigcup_{x\in\event{I}} [x]_{C_x}\supseteq\xSp{X}$.
	The set $E\IsDef \bigcup_{x\in\event{I}} E_x$ has the desired property.	
\end{proof}
\begin{figure}
	\begin{center}
		\begin{tikzpicture}[scale=1,line cap=round, >=stealth',
			hachA/.style={
				decorate,decoration={snake,amplitude=1pt,segment length=5pt},thick
			},
			hachB/.style={
				decorate,decoration={crosses,segment length=3pt,shape height=3pt,shape width=3pt},thick
			},
			hachAA/.style={
				decorate,decoration={zigzag,amplitude=1pt,segment length=5pt},thick
			},
			hachBB/.style={
				decorate,decoration={shape backgrounds,shape=diamond,shape size=3pt,shape sep=4pt},thick
			}			
		]
			\def\vdist{0.4}		% distance between lines
			\def\eps{0.1}		% tiny length
			\def\vgap{3}			% gap between time 0 and time 1
			\def\window{6}
			\def\dwin{0.5}		% D
			\def\ewin{1.7}		% E
			\def\cwin{5.3}		% C
			\def\crwin{4.8}		% C interior
			\def\xwin{5.6}		% unknown
			\def\winskip{0.35}

			\coordinate[label=left:$\scriptstyle x$]		(x) at		(-\window,0);
			\coordinate[label=left:$\scriptstyle x_1$]	(x1) at		(-\window,-\vdist);
			\coordinate[label=left:$\scriptstyle x'_1$]	(xp1) at		(-\window,-2*\vdist);
			\coordinate[label=left:$\scriptstyle x'$]		(xp) at		(-\window,-3*\vdist);
			\coordinate[label=left:$\scriptstyle x''_1$]	(xpp1) at	(-\window,-4*\vdist);
			
			\coordinate [label=left:$\scriptstyle \Phi x$]	(fx) at		(-\window,-\vgap);
			\coordinate [label=left:$\scriptstyle \Phi x_1$]	(fx1) at		(-\window,-\vgap-\vdist);
			\coordinate [label=left:$\scriptstyle \Phi x'_1$]	(fxp1) at	(-\window,-\vgap-2*\vdist);
			\coordinate [label=left:$\scriptstyle y'$]		(yp) at		(-\window,-\vgap-3*\vdist);
			
			\draw[very thick] (-\window,0) -- (\window,0);			% x
			
			\draw[hachB] (-\window,-\vdist) -- (-\cwin,-\vdist);	% x_1
			\draw[hachB] (\window,-\vdist) -- (\cwin,-\vdist);
			\draw[very thick] (-\cwin,-\vdist) -- (\cwin,-\vdist);
			
			\draw[hachB] (-\window,-2*\vdist) -- (-\xwin,-2*\vdist);	% x'_1
			\draw[hachB] (\window,-2*\vdist) -- (\xwin,-2*\vdist);
			\draw[dashed] (-\xwin,-2*\vdist) -- (\xwin,-2*\vdist);
			
			\draw[very thick] (-\window,-3*\vdist) -- (-\ewin,-3*\vdist);	% x'
			\draw[very thick] (\window,-3*\vdist) -- (\ewin,-3*\vdist);
			\draw[hachA] (-\ewin,-3*\vdist) -- (\ewin,-3*\vdist);
			
			\draw[hachB] (-\window,-4*\vdist) -- (-\cwin,-4*\vdist);	% x''_1
			\draw[hachB] (\window,-4*\vdist) -- (\cwin,-4*\vdist);
			\draw[very thick] (-\cwin,-4*\vdist) -- (-\ewin,-4*\vdist);
			\draw[very thick] (\cwin,-4*\vdist) -- (\ewin,-4*\vdist);
			\draw[hachA] (-\ewin,-4*\vdist) -- (\ewin,-4*\vdist);
			
			\draw[very thick] (-\window,-\vgap) -- (\window,-\vgap);	% \Phi x
			
			\draw[hachBB] (-\window,-\vgap-\vdist) -- (-\crwin,-\vgap-\vdist);	% \Phi x_1
			\draw[hachBB] (\window,-\vgap-\vdist) -- (\crwin,-\vgap-\vdist);
			\draw[very thick] (-\crwin,-\vgap-\vdist) -- (\crwin,-\vgap-\vdist);
			
			\draw[hachBB] (-\window,-\vgap-2*\vdist) -- (-\crwin,-\vgap-2*\vdist);	% \Phi x'_1
			\draw[hachBB] (\window,-\vgap-2*\vdist) -- (\crwin,-\vgap-2*\vdist);
			\draw[very thick] (-\crwin,-\vgap-2*\vdist) -- (-\dwin,-\vgap-2*\vdist);
			\draw[very thick] (\crwin,-\vgap-2*\vdist) -- (\dwin,-\vgap-2*\vdist);
			\draw[hachAA] (-\dwin,-\vgap-2*\vdist) -- (\dwin,-\vgap-2*\vdist);
			
			\draw[very thick] (-\window,-\vgap-3*\vdist) -- (-\dwin,-\vgap-3*\vdist);	% y'
			\draw[very thick] (\window,-\vgap-3*\vdist) -- (\dwin,-\vgap-3*\vdist);
			\draw[hachAA] (-\dwin,-\vgap-3*\vdist) -- (\dwin,-\vgap-3*\vdist);
			
			\draw[help lines] (-\ewin,\winskip+\eps) -- (-\ewin,-4*\vdist-\eps);
			\draw[help lines] (\ewin,\winskip+\eps) -- (\ewin,-4*\vdist-\eps);	
			\draw[help lines, <->, shorten <=1pt, shorten >=1pt] (-\ewin,\winskip) -- (\ewin,\winskip)
				node[fill=white,pos=0.5] {$\scriptstyle E_x$};
			
			\draw[help lines] (-\cwin,2*\winskip+\eps) -- (-\cwin,-4*\vdist-\eps);
			\draw[help lines] (\cwin,2*\winskip+\eps) -- (\cwin,-4*\vdist-\eps);	
			\draw[help lines, <->, shorten <=1pt, shorten >=1pt] (-\cwin,2*\winskip) -- (\cwin,2*\winskip)
				node[fill=white,pos=0.5] {$\scriptstyle C_x$};
			
			\draw[help lines] (-\dwin,-\vgap+\winskip+\eps) -- (-\dwin,-\vgap-3*\vdist-\eps);
			\draw[help lines] (\dwin,-\vgap+\winskip+\eps) -- (\dwin,-\vgap-3*\vdist-\eps);	
			\draw[help lines, <->, shorten <=1pt, shorten >=1pt] (-\dwin,-\vgap+\winskip) -- (\dwin,-\vgap+\winskip)
				node[fill=white,pos=0.5] {$\scriptstyle D$};
				
			\draw[help lines] (-\crwin,-\vgap+\eps) -- (-\crwin,-\vgap-3*\vdist-\eps);
			\draw[help lines] (\crwin,-\vgap+\eps) -- (\crwin,-\vgap-3*\vdist-\eps);	

		\end{tikzpicture}
	\end{center}
	\caption{%
		Illustration of the proof of Lemma~\ref{lem:pre-injective:complete}.
	}
	\label{fig:lemma:complete-pre-injective}
\end{figure}
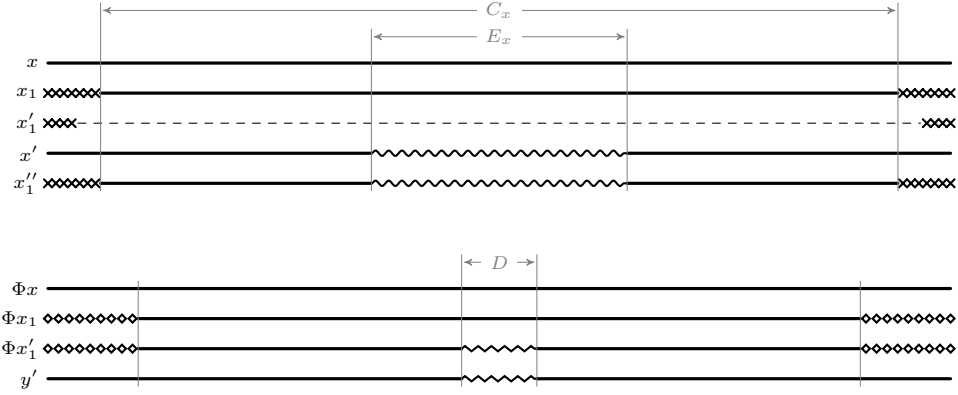

\begin{theorem}[see~\cite{Rue04}, Proposition~2.5]
\label{thm:gibbs:factor-map:complete}
	Let $\Phi:\xSp{X}\to\xSp{Y}$ be a complete pre-injective factor map
	between two strongly irreducible shifts of finite type $(\xSp{X},\sigma)$ and $(\xSp{Y},\sigma)$.
	Let $\Delta$ be a Hamiltonian on $\xSp{Y}$, and $\pi$ a probability measure on $\xSp{X}$.
	If $\pi$ is a Gibbs measure for~$\Phi^*\Delta$,
	then $\Phi\pi$ is a Gibbs measure for~$\Delta$.
\end{theorem}
\begin{proof}
	Let $0\in N\subseteq\xL$ be a neighborhood for $\Phi$.
	Let $0\in M\subseteq\xL$ be a neighborhood that witnesses
	the finite type gluing property of both $\xSp{X}$ and $\xSp{Y}$.
	We write $\tilde{N}\IsDef N^{-1}(N)$ and $\tilde{M}\IsDef M^{-1}(M)$.
	
	Let $y$ and $y'$ be two asymptotic configurations in $\xSp{Y}$,
	and set $D\IsDef\diff(y,y')$.  For every configuration $x\in\Phi^{-1}[y]_{\tilde{M}(D)}$,
	there is a unique configuration $x'\in\Phi^{-1}[y']_{\tilde{M}(D)}$ that is asymptotic to $x$
	and such that $\diff(\Phi x,\Phi x')\subseteq D$.
	(Namely, by the gluing property of $\xSp{Y}$, the configuration $y'_x$ that agrees with
	$y'$ in $\tilde{M}(D)$ and with $\Phi x$ outside $D$ is in $\xSp{Y}$.
	Since $\Phi$ is a complete pre-injective factor map, there is a unique configuration $x'$
	that is asymptotic to $x$ and $\Phi x'=y'_x$.)
	The relation $x\mapsto x'$ is a one-to-one correspondence.
	By Lemma~\ref{lem:pre-injective:complete}, there is a large enough finite set $E\subseteq\xL$
	such that for every $x\in\Phi^{-1}[y]_{\tilde{M}(D)}$, it holds $\diff(x,x')\subseteq E$.
	
	Consider a large finite set $\hat{D}\subseteq\xL$
	and another finite set $\hat{E}\subseteq\xL$ that is much larger than $\hat{D}$.
	(More precisely, we need $\hat{D}\supseteq \tilde{M}(D)$ and $\hat{E}\supseteq N(\hat{D})\cup\tilde{N}(E)$.)
	Let $P_y$ denote the set of patterns $p\in L_{\hat{E}}(\xSp{X})$
	such that $\Phi[p]_{\hat{E}}\subseteq [y]_{\hat{D}}$.
	Then, $\Phi^{-1}[y]_{\hat{D}}=\bigcup_{p\in P_y} [p]_{\hat{E}}$ (provided $\hat{E}\supseteq N(\hat{D})$).
	Let $\event{I}\subseteq\xSp{X}$ be a finite set
	consisting of one representative 
	from each cylinder $[p]_{\hat{E}}$, for $p\in A$.
	Then, $\Phi^{-1}[y]_{\hat{D}}=\bigcup_{x\in\event{I}}[x]_{\hat{E}}$.
	Moreover, $\Phi^{-1}[y']_{\hat{D}}=\bigcup_{x\in\event{I}}[x']_{\hat{E}}$
	(provided $\hat{E}\supseteq\tilde{N}(E)$).
	Since the terms in each of the two latter unions are disjoint, we have
	\begin{align}
		(\Phi\pi)[y]_{\hat{D}} = \sum_{x\in\event{I}} \pi([x]_{\hat{E}})
		\qquad\quad\text{and}\qquad\quad
		(\Phi\pi)[y']_{\hat{D}} = \sum_{x\in\event{I}} \pi([x']_{\hat{E}}) \;.
	\end{align}
	For each $x\in \Phi^{-1}[y]_{\hat{D}}$, we have
	\begin{align}
		\frac{\pi([x']_{\hat{E}})}{\pi([x]_{\hat{E}})} - \xe^{-\Delta(y,y')} &=
			\left(\frac{\pi([x']_{\hat{E}})}{\pi([x]_{\hat{E}})} - \xe^{-(\Phi^*\Delta)(x,x')}\right) +
			\left(\xe^{-\Delta(\Phi x,\Phi x')} - \xe^{-\Delta(y,y')}\right) \;.
	\end{align}
	Let us denote the first term on the righthand side by $\delta_{\hat{E}}(x)$
	and the second term by $\gamma_{\hat{D}}(x)$.
	Note that, since $\pi$ is a Gibbs measure for $\Phi^*\Delta$
	and $\diff(x,x')\subseteq E$, $\delta_{\hat{E}}(x)\to 0$
	uniformly over $\Phi^{-1}[y]_{\hat{D}}$ as $\hat{E}\nearrow\xL$.
	Note also that, 
	by the continuity property of $\Delta$,
	$\gamma_{\hat{D}}(x)\to 0$
	uniformly over $\Phi^{-1}[y]_{\hat{D}}$ as $\hat{D}\nearrow\xL$.
	We can now write
	\begin{align}
		\frac{(\Phi\pi)[y']_{\hat{D}}}{(\Phi\pi)[y]_{\hat{D}}} =
			\frac{\displaystyle{%
				\sum_{x\in\event{I}} \pi([x']_{\hat{E}})
			}}{\displaystyle{%
				\sum_{x\in\event{I}} \pi([x]_{\hat{E}})
			}}
		&=
			\frac{\displaystyle{%
				\sum_{x\in\event{I}} \left(
					\xe^{-\Delta(y,y')} + \delta_{\hat{E}}(x) + \gamma_{\hat{D}}(x)
				\right)\;  \pi([x]_{\hat{E}})
			}}{\displaystyle{%
				\sum_{x\in\event{I}} \pi([x]_{\hat{E}})
			}} \\
		&=
			\xe^{-\Delta(y,y')} + \frac{\displaystyle{%
				\sum_{x\in\event{I}} \left(\delta_{\hat{E}}(x) + \gamma_{\hat{D}}(x)\right)\;  \pi([x]_{\hat{E}})
			}}{\displaystyle{%
				\sum_{x\in\event{I}} \pi([x]_{\hat{E}})
			}} \;.
	\end{align}
	Consider a small number $\varepsilon>0$.
	If $\hat{D}$ is sufficiently large, we have $\abs{\gamma_{\hat{D}}}<\varepsilon/2$.
	Moreover, %if $\hat{E}$ is sufficiently large, we have
	after choosing $\hat{D}$, we can choose $\hat{E}$ large enough so that
	$\abs{\delta_{\hat{E}}}<\varepsilon/2$.
	Therefore, for $\hat{D}$ sufficiently large we get
	\begin{align}
		\abs{\frac{(\Phi\pi)[y']_{\hat{D}}}{(\Phi\pi)[y]_{\hat{D}}} - \xe^{-\Delta(y,y)}} &\leq \varepsilon.
	\end{align}
	It follows that
	\begin{align}
		\abs{\frac{(\Phi\pi)[y']_{\hat{D}}}{(\Phi\pi)[y]_{\hat{D}}} - \xe^{-\Delta(y,y)}} &\to 0
	\end{align}
	as $\hat{D}\nearrow \xL$.
	Since this is valid for every two asymptotic configurations $y,y'\in\xSp{Y}$,
	we conclude that $\Phi\pi$ is a Gibbs measure for $\Delta$.
\end{proof}

\begin{corollary}
\label{cor:gibbs:conjugacy}
	Let $\Phi:\xSp{X}\to\xSp{Y}$ be a conjugacy between
	two strongly irreducible shifts of finite type $(\xSp{X},\sigma)$ and $(\xSp{Y},\sigma)$.
	Let $\Delta$ be a Hamiltonian on $\xSp{Y}$, and $\pi$ a probability measure on $\xSp{X}$.
	Then, $\pi$ is a Gibbs measure for~$\Phi^*\Delta$ if and only if
	$\Phi\pi$ is a Gibbs measure for~$\Delta$.
\end{corollary}

%========================================================================
\subsection{The Image of a Gibbs Measure}
%------------------------------------------------------------------------

Let $\Phi:\xSp{X}\to\xSp{Y}$ be a pre-injective factor map between
two strongly irreducible shifts of finite type.
According to Corollary~\ref{thm:equilibrium:factor-map}, 
a pre-image of a shift-invariant Gibbs measure under the induced map
$\xSx{P}(\xSp{X},\sigma)\to\xSx{P}(\xSp{Y},\sigma)$ is again a Gibbs measure.
The image of a Gibbs measure, however, does not need to be a Gibbs measure
as the following example demonstrates.
\begin{example}[XOR map]
\label{exp:xor:non-gibbs}
	Let $\xSp{X}=\xSp{Y}\IsDef\{\symb{0},\symb{1}\}^{\xZ}$ be the binary full shift
	and $\Phi$ the so-called \emph{XOR map}, defined by
	$(\Phi x)(i)\IsDef x(i)+x(i+1) \pmod{2}$.
	Let $\pi$ be the shift-invariant Bernoulli measure on $\xSp{X}$
	with marginals $\symb{1}\mapsto p$ and $\symb{0}\mapsto 1-p$, where $0< p< 1$.
	This is a Gibbs measure for the Hamiltonian $\Delta_f$,
	where $f:\xSp{X}\to\xR$ is the single-site observable
	defined by
	$f(x)\IsDef-\log p$ if $x(0)=\symb{1}$ and $f(x)\IsDef-\log(1-p)$ if $x(0)=\symb{0}$.
	We claim that unless $p=\frac{1}{2}$,
	$\Phi\pi$ is not a regular Gibbs measure
	(i.e., a Gibbs measure for a Hamiltonian generated by an observable with summable variations).
	
	Suppose, on the contrary, that $p\neq\frac{1}{2}$
	and $\Phi\pi$ is
	a Gibbs measure for $\Delta_g$ for some $g\in SV(\xSp{X})$.
	Then $\pi$ is also an equilibrium measure
	for $g\xO\Phi$ (Corollary~\ref{thm:equilibrium:factor-map}),
	implying that $f$ and $g\xO\Phi$ are physically equivalent
	(Proposition~\ref{prop:equilibrium:equivalence}).
	Consider the two uniform configurations
	$\underline{\symb{0}}$ and $\underline{\symb{1}}$,
	where 
	$\underline{\symb{0}}(i)\IsDef \symb{0}$ and $\underline{\symb{1}}(i)\IsDef \symb{1}$
	for every $i\in\xZ$.
	We have
	$f(\underline{\symb{0}})=-\log p\neq -\log(1-p)=f(\underline{\symb{1}})$,
	whereas
	$g\xO\Phi(\underline{\symb{0}})=g\xO\Phi(\underline{\symb{1}})$.
	If $\delta_{\underline{\symb{0}}}$ and $\delta_{\underline{\symb{1}}}$ are,
	respectively, the probability measures
	concentrated on $\underline{\symb{0}}$ and $\underline{\symb{1}}$, we get that
	$\delta_{\underline{\symb{0}}}(f)-\delta_{\underline{\symb{0}}}(g\xO\Phi)
	\neq
		\delta_{\underline{\symb{1}}}(f)-\delta_{\underline{\symb{1}}}(g\xO\Phi)$.
	This is a contradiction with the physical equivalence of
	$f$ and $g\xO\Phi$, because
	$\delta_{\underline{\symb{0}}},\delta_{\underline{\symb{1}}}\in\xSx{P}(\xSp{X},\sigma)$
	(Proposition~\ref{prop:physical-equivalence:continuous}).
	
	In fact, the same argument shows that none of the $n$-fold iterations
	$\Phi^n\pi$ are regular Gibbs measures,
	because $\Phi^n(\underline{\symb{0}})=\Phi^n(\underline{\symb{1}})$ for every $n\geq 1$.
	On the other hand, it has been shown~\cite{Miy79,Lin84}, that $\Phi^n\pi$ converges in density
	to the uniform Bernoulli measure, which is a Gibbs measure and is invariant under $\Phi$.
	The question of approach to equilibrium will be discussed in Section~\ref{sec:ca:randomization}.%
	\exampleqed
\end{example}

The latter example was first suggested by van~den~Berg (see~\cite{LorMaeVan98}, Section~3.2)
as an example of a measure that is \emph{strongly} non-Gibbsian, in the sense that
attempting to define a Hamiltonian for it via~(\ref{eq:gibbs:def})
would lead to a function $\Delta$ 
for which the continuity property fails \emph{everywhere}.
The question of when a measure is Gibbsian and the study of the symptoms of being non-Gibbsian
is an active area of research as non-Gibbsianness sets boundaries on the applicability of
the so-called renormalization group technique in statistical mechanics (see e.g.~\cite{EntFerSok93,Fer06}).

The observation in Example~\ref{exp:xor:non-gibbs} can be generalized as follows.
\begin{proposition}
\label{prop:gibbs-non-gibbs}
	Let $(\xSp{X},\sigma)$ be a strongly irreducible shift of finite type and
	$\pi\in\xSx{P}(\xSp{X},\sigma)$ a Gibbs measure for a Hamiltonian $\Delta_f$,
	where $f\in SV(\xSp{X})$.
	Suppose that $\Phi:\xSp{X}\to\xSp{Y}$ is a pre-injective factor map from $(\xSp{X},\sigma)$
	onto another shift of finite type $(\xSp{Y},\sigma)$.
	A necessary condition for $\Phi\pi$ to be 
	a regular Gibbs measure
	is that for every two measures $\mu_1,\mu_2\in\xSx{P}(\xSp{X},\sigma)$
	with $\mu_1(f)\neq\mu_2(f)$ it holds $\Phi\mu_1\neq\Phi\mu_2$.
\end{proposition}

\begin{example}[XOR map; Example~\ref{exp:xor:non-gibbs} continued]
\label{exp:xor:non-gibbs:contd}
	The argument of Example~\ref{exp:xor:non-gibbs} can be stretched to
	show that the iterations of the XOR map turn every Gibbs measure other than
	the uniform Bernoulli measure eventually to a non-Gibbs measure.
	More specifically, for every observable $f\in SV(\xSp{X})$ that is not physically equivalent to~$0$
	and every shift-invariant Gibbs measure $\pi$ for $\Delta_f$,
	there is an integer $n_0\geq 1$ such that for any $n\geq n_0$,
	the measure $\Phi^n\pi$ is not a regular Gibbs measure.
	
	This is a consequence of the self-similar behaviour of the XOR map.
	Namely, the map $\Phi$ satisfies
	$(\Phi^{2^k}x)(i) = x(i) + x(i+2^k) \pmod{2}$ for every $i\in\xZ$ and every $k\geq 1$.
	If $f$ is not physically equivalent to $0$,
	two periodic configurations $x,y\in\xSp{X}$
	with common period $2^k$ can be found such that %$2^{-k}\xSum_{[0,2^k)}f(x)\neq2^{-k}\xSum_{[0,2^k)}f(y)$.
	$2^{-k}\sum_{i=0}^{2^k-1}f(\sigma^i x)\neq 2^{-k}\sum_{i=0}^{2^k-1}f(\sigma^i y)$.
	If $\mu_x$ and $\mu_y$ denote, respectively, the shift-invariant measures concentrated
	at the shift orbits of $x$ and $y$, we obtain that $\mu_x(f)\neq \mu_y(f)$.
	Nevertheless, $\Phi^n x=\Phi^n y=\underline{\symb{0}}$ for all $n\geq 2^k$, implying that
	$\Phi^n \mu_x = \Phi^n \mu_y = \delta_{\underline{\symb{0}}}$.
	Therefore, according to Proposition~\ref{prop:gibbs-non-gibbs},
	the measure $\Phi^n\pi$
	cannot be a regular Gibbs measure. 
	\exampleqed
\end{example}

With the interpretation of the shift-ergodic measures
as the macroscopic states (see the \hyperref[sec:intro]{Introduction}),
the above proposition 
reads as follows:
a sufficient condition for the non-Gibbsianness of $\Phi\pi$ is that
there are two macroscopic states that are distinguishable by the density of $f$
and are mapped to the same state by $\Phi$.

If the induced map $\Phi:\xSx{P}(\xSp{X},\sigma)\to\xSx{P}(\xSp{Y},\sigma)$
is not one-to-one, then there are Gibbs measures (even Markov measures)
whose images are not Gibbs.
For, suppose $\mu_1,\mu_2\in\xSx{P}(\xSp{X},\sigma)$ are distinct measures
with $\Phi\mu_1=\Phi\mu_2$.
Then, there is a local observable $f\in K(\xSp{X})$ such that $\mu_1(f)\neq\mu_2(f)$.
Every shift-invariant Gibbs measure for $\Delta_f$ is mapped by $\Phi$
to a measure that is not regular Gibbs.

\begin{question}
	Let $\Phi:\xSp{X}\to\xSp{Y}$ be a pre-injective factor map between
	two strongly irreducible shifts of finite type $(\xSp{X},\sigma)$ and $(\xSp{Y},\sigma)$,
	and suppose that the induced map $\Phi:\xSx{P}(\xSp{X},\sigma)\to\xSx{P}(\xSp{Y},\sigma)$
	is injective.  Does $\Phi$ map every (regular, shift-invariant) Gibbs measure to a Gibbs measure?
\end{question}

%*****************************************************************************
% Algorithmics
%vvvvvvvvvvvvvvvvvvvvvvvvvvvvvvvvvvvvvvvvvvvvvvvvvvvvvvvvvvvvvvvvvvvvvvvvvvvvv
%\begin{question}
%	Let $\Phi:\xSp{X}\to\xSp{Y}$ be a pre-injective factor map
%	between two strongly irreducible shifts of finite type $(\xSp{X},\sigma)$ and $(\xSp{Y},\sigma)$.
%	Given a local observable $f$ on $\xSp{X}$, is it algorithmically decidable whether there exists
%	a (local/summable-variation/continuous) observable $g$ on $\xSp{Y}$
%	such that $g\xO\Phi$ is physically equivalent to $f$?
%\end{question}
%\begin{question}
%	Let $\Phi:\xSp{X}\to\xSp{X}$ be a pre-injective cellular automaton
%	on a strongly irreducible shift of finite type $(\xSp{X},\sigma)$.
%	Given a local observable $f$ on $\xSp{X}$, is it algorithmically decidable whether there exist
%	(local/summable-variation/continuous) observables $g_n$ (for $n>0$) on $\xSp{X}$
%	such that $g_n\xO\Phi^n$ are physically equivalent to $f$?
%\end{question}
%^^^^^^^^^^^^^^^^^^^^^^^^^^^^^^^^^^^^^^^^^^^^^^^^^^^^^^^^^^^^^^^^^^^^^^^^^^^^^

%xxxxxxxxxxxxxxxxxxxxxxxxxxxxxxxxxxxxxxxxxxxxxxxxxxxxxxxxxxxxxxxxxxxxxxxx
\section{Cellular Automata}
%------------------------------------------------------------------------
\label{sec:ca}

%========================================================================
\subsection{Conservation Laws}
%------------------------------------------------------------------------
\label{sec:ca:conservation-laws}

Let $\Phi:\xSp{X}\to\xSp{X}$ be a cellular automaton
on a strongly irreducible shift of finite type
$(\xSp{X},\sigma)$.
We say that $\Phi$ \emph{conserves} (the energy-like quantity formalized by) a Hamiltonian $\Delta$
if $\Delta(\Phi x,\Phi y)=\Delta(x,y)$ for every two asymptotic configurations $x,y\in\xSp{X}$.
If $\Delta=\Delta_f$ is the Hamiltonian generated by a local observable $f\in K(\xSp{X})$,
then we may also say that $\Phi$ conserves $f$ (in the aggregate).
More generally, we say that
a continuous observable $f\in C(\xSp{X})$ is conserved by $\Phi$ if
$f$ and $f\xO\Phi$ are physically equivalent.
According to Proposition~\ref{prop:physical-equivalence:continuous},
this is equivalent to the existence of a constant $c\in\xR$
such that $(\Phi\pi)(f)=\pi(f)+c$ for every
shift-invariant probability measure $\pi$.
However, in this case $c$ is always $0$.
\begin{proposition}
\label{prop:conservation:char}
	Let $\Phi:\xSp{X}\to\xSp{X}$ be a cellular automaton
	over a strongly irreducible shift of finite type $(\xSp{X},\sigma)$.
	A continuous observable $f\in C(\xSp{X})$ is conserved by $\Phi$
	if and only if $(\Phi\pi)(f)=\pi(f)$
	for every probability measure $\pi\in\xSx{P}(\xSp{X},\sigma)$.
\end{proposition}
\begin{proof}
	Let $c\in\xR$ be such that $(\Phi\pi)(f)=\pi(f)+c$
	for every $\pi\in\xSx{P}(\xSp{X},\sigma)$.
	Then, for every $n>0$,
	$(\Phi^n\pi)(f)=\pi(f)+nc$.
	However, every continuous function on a compact space is bounded.
	Therefore, $c=0$.
\end{proof}

If an observable $f$ is conserved by a cellular automaton $\Phi$,
we say that $f$ is bound by a \emph{conservation law} under $\Phi$.
There is also a concept of \emph{local} conservation law.
Let $D_0$ be a finite generating set for the group $\xL=\xZ^d$.
Suppose that $f\in C(\xSp{X})$ is an observable that is conserved by a cellular automaton $\Phi:\xSp{X}\to\xSp{X}$.
By Proposition~\ref{prop:conservation:char} and Lemma~\ref{lem:annihilators},
this means that $f\xO\Phi-f\in C(\xSp{X},\sigma)$, that is
\begin{align}
	f\xO\Phi-f &= \lim_{n\to\infty} \sum_{i\in D_0} (h_i^{(n)}\xO\sigma^i - h_i^{(n)})
\end{align}
for some $h_i^{(n)}\in K(\xSp{X})$ (for $i\in D_0$ and $n=0,1,2,\ldots$).
In other words, for every configuration $x\in\xSp{X}$ it holds
\begin{align}
	f(\Phi x) &= f(x) + \lim_{n\to\infty}\left[
			\sum_{i\in D_0} h_i^{(n)}(\sigma^i x) - \sum_{i\in D_0} h_i^{(n)}(x)
		\right] \;.
\end{align}
If furthermore $f\xO\Phi-f\in K(\xSp{X},\sigma)\subseteq C(\xSp{X},\sigma)$
(where $K(\xSp{X},\sigma)$ is defined as in~(\ref{eq:def:local-coboundaries})),
then we have the more intuitive equation
\begin{align}
\label{eq:cons-law:local}
	f(\Phi x) &= f(x) + \sum_{i\in D_0} h_i(\sigma^i x) - \sum_{i\in D_0} h_i(x)
\end{align}
for some $h_i\in K(\xSp{X})$.
In this case, we say that $f$ is \emph{locally conserved} by $\Phi$
(or satisfies a \emph{local conservation law} under $\Phi$).
The value $h_i(\sigma^k x)$ is then interpreted as the \emph{flow} (of the energy-like quantity captured by $f$)
from site $k$ to site $k-i$.  The latter equation is
a \emph{continuity equation}, stating that at each site $k$, the changes in the observed quantity after one step
should balance with the incoming and the outgoing flows.
If $\xSp{X}$ is a full shift, it is known that every conserved local observable
is locally conserved.  The proof is similar to that of Proposition~\ref{prop:Hamiltonian:observable:local}.

Local conservation laws enjoy a somewhat symmetric relationship with time and space.
Namely, an observable $f\in K(\xSp{X})$ is locally conserved by $\Phi$
if and only if the observable $\alpha\IsDef f\xO\Phi-f$ is
in $K(\xSp{X},\Phi)\cap K(\xSp{X},\sigma)$.
Moreover, to every observable $\alpha\in K(\xSp{X},\Phi)\cap K(\xSp{X},\sigma)$,
there corresponds at least one observable $f\in K(\xSp{X})$
such that $\alpha=f\xO\Phi-f$ and $f$ is locally conserved by $\Phi$.
In general, there might be several observables $f$ with the latter property.
If $\alpha=f\xO\Phi-f=f'\xO\Phi-f'$ for two observables $f,f'\in K(\xSp{X})$,
then $(f-f')=(f-f')\xO\Phi$; that is, $f-f'$ is \emph{invariant} under $\Phi$.
Every constant observable is invariant under any cellular automaton.  The following is an example
of a cellular automaton with non-constant invariant local observables.
\begin{example}[Invariant observables]
\label{exp:almost-equicontinuous}
	Let $\Phi:\{\symb{0},\symb{1},\symb{2}\}^\xZ\to\{\symb{0},\symb{1},\symb{2}\}^\xZ$
	be the cellular automaton with
	\begin{align}
		(\Phi x)(i) &\IsDef \begin{cases}
			\symb{2}							& \text{if $x(i)=\symb{2}$,}\\
			x(i) + x(i+1) \pmod{2}\qquad{}	& \text{otherwise}
		\end{cases}
	\end{align}
	(see Figure~\ref{fig:almost-equicontinuous}).
	The observable $f:\{\symb{0},\symb{1},\symb{2}\}^\xZ\to\xR$ defined by
	$f(x)\IsDef 1$ if $x(i)=\symb{2}$ and $f(x)\IsDef 0$ otherwise
	is obviously invariant.
	The Hamiltonian $\Delta_f$ counts the number of occurrences of symbol~$\symb{2}$
	and is conserved by $\Phi$.
	In fact, there are infinitely many linearly independent, physically non-equivalent
	observables that are invariant under $\Phi$.
	Namely, the relative position of the occurrences of~$\symb{2}$ remain unchanged,
	and hence,
	for any finite set $D\subseteq\xZ$, the logical conjunction of $f\xO\sigma^i$ for $i\in D$
	is invariant.
	It follows that $\Phi$ has infinitely many distinct (and linearly independent)
	conservation laws.
	
	Such abundance of conservation laws is common among all cellular automata
	having non-constant invariant local observables (see Lemma~2 of~\cite{ForKarTaa11}),
	and has been suggested as the reason behind the ``non-physical'' behavior
	in these cellular automata (see e.g.~\cite{Tak87}).
	Every surjective equicontinuous cellular automaton is periodic~\cite{BlaTis00,Gam06}
	and hence has non-constant invariant local observables.
	It follows that every surjective cellular automaton that has
	a non-trivial equicontinuous cellular automaton as factor has non-constant invariant local observables
	and an infinity of linearly independent conservation laws.
	\exampleqed
\end{example}
\begin{question}
	Does every surjective cellular automaton with equicontinuous points
	have non-constant local observables?
\end{question}
\begin{figure}
	\begin{center}
		\includegraphics[width=0.6\textwidth]{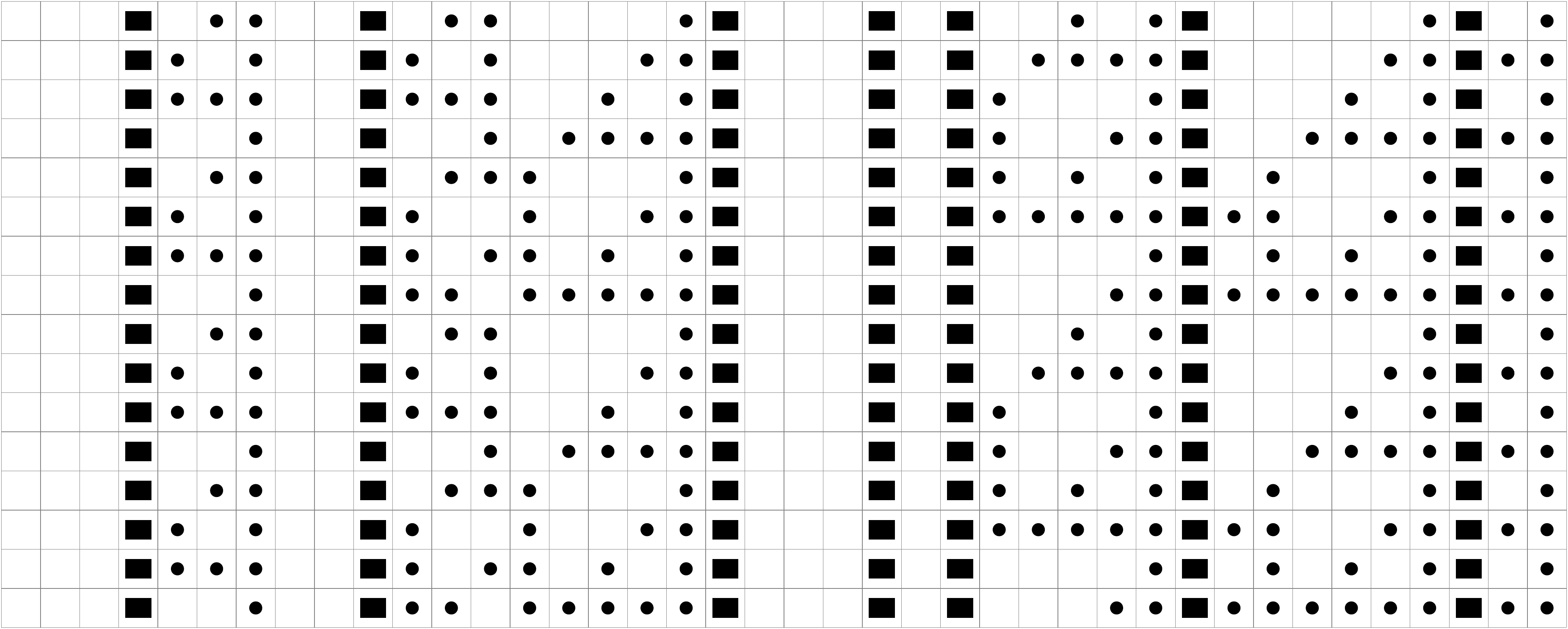}
	\end{center}
	\caption{%
		A cellular automaton with infinitely many distinct conservation laws
		(see Example~\ref{exp:almost-equicontinuous}).
		Black square represents state $\symb{2}$ and  black dot represents $\symb{1}$.
		Time goes downward.
	}
	\label{fig:almost-equicontinuous}
\end{figure}

Every cellular automaton $\Phi:\xSp{X}\to\xSp{X}$ conserves the trivial Hamiltonian $\Delta\equiv 0$ on $\xSp{X}$.
Furthermore, every observable $f\in C(\xSp{X})$ that is physically equivalent to $0$
(i.e., $f-c\in C(\xSp{X},\sigma)$ for some $c\in\xR$)
is \emph{trivially} conserved by $\Phi$.
Likewise, a local observable $f\in K(\xSp{X})$
is \emph{trivially} locally conserved by $\Phi$
if $f-c\in K(\xSp{X},\sigma)$ for some $c\in\xR$.
We shall say that two local observables $f,g\in K(\xSp{X})$
are \emph{locally physically equivalent}
if $f-g-c\in K(\xSp{X},\sigma)$ for some $c\in\xR$.
The following proposition is the analogue of Proposition~\ref{prop:conservation:char}.
\begin{proposition}
	Let $\Phi:\xSp{X}\to\xSp{X}$ be a cellular automaton
	over a strongly irreducible shift of finite type $(\xSp{X},\sigma)$.
	A local observable $f\in K(\xSp{X})$ is locally conserved by $\Phi$
	if and only if $f$ and $f\xO\Phi$ are locally physically equivalent.
\end{proposition}
%\begin{proof}
%	If $f$ is locally conserved by $\Phi$, then $f\xO\Phi$ and $f$ are locally physically equivalent by definition.
%	The converse is similar to Proposition~\ref{prop:conservation:char}.
%\end{proof}

%========================================================================
\subsection{Invariance of Gibbs Measures}
%------------------------------------------------------------------------
\label{sec:ca:invariance}

As a corollary of the results of Section~\ref{sec:entropy-preserving},
we obtain a correspondence between the conservation laws of a surjective cellular automata
and its invariant Gibbs measures.
It is well-known that every surjective cellular automaton
over a full shift preserves the uniform Bernoulli measure
(see~\cite{Hed69}, Theorem~5.4 and~\cite{MarKim76}).
The invariance of the uniform Bernoulli measure is sometimes called
the \emph{balance property} of (the local update rule of) the surjective cellular automata.
In case of surjective cellular automata over
strongly irreducible shifts of finite type, a similar property is known to hold:
every measure of maximum entropy is mapped to a measure of maximum entropy
(see~\cite{CovPau74}, Corollary~2.3 and~\cite{MeeSte01}, Theorems~3.3 and~3.6).
The following two theorems can be seen as further generalizations of
the balance property.  Indeed, choosing $f\equiv 0$ in either
of the two theorems implies that a surjective cellular automaton
maps each measure of maximum entropy
to a measure of maximum entropy.
An elementary proof of Theorem~\ref{thm:conservation-invariance:1}
in the special case of surjective cellular automata on one-dimensional full shifts
and single-site observables was earlier presented in~\cite{KarTaa11}.

\begin{theorem}
\label{thm:conservation-invariance:1}
	Let $\Phi:\xSp{X}\to\xSp{X}$ be a surjective cellular automaton
	over a strongly irreducible shift of finite type $(\xSp{X},\sigma)$,
	and let $f\in SV(\xSp{X})$ be an observable with summable variations.
	The following conditions are equivalent:
	\begin{enumerate}[ \rm a)]
		\item $\Phi$ conserves $f$.
		\item $\Phi$ maps the set $\xSx{E}_f(\xSp{X},\sigma)$
			of equilibrium measures for $f$ onto itself.
		\item There exist a measure in $\xSx{E}_f(\xSp{X},\sigma)$
			whose $\Phi$-image is also in $\xSx{E}_f(\xSp{X},\sigma)$.
	\end{enumerate}
	If $f\in C(\xSp{X})$ does not have summable variations, condition~(a)
	still implies the other two conditions.
\end{theorem}
\begin{proof}\ \\
	\begin{enumerate}[ a $\Rightarrow$ b)]
		\item[a $\Rightarrow$ b)]
			Suppose that $\Phi$ conserves $f$.
			By Proposition~\ref{prop:equilibrium:equivalence}
			and Corollary~\ref{thm:equilibrium:factor-map} we have
			$\pi\in\xSx{E}_f(\xSp{X},\sigma)$ if and only if
			$\Phi\pi\in\xSx{E}_f(\xSp{X},\sigma)$.
			Using Lemma~\ref{lem:onto-one-to-one}, we obtain
			$\Phi\xSx{E}_f(\xSp{X},\sigma)=\xSx{E}_f(\xSp{X},\sigma)$.
		
		\item[b $\Rightarrow$ c)] Trivial.
		
		\item[c $\Rightarrow$ a)]
			Let $f$ have summable variations.  Then, so does $f\xO\Phi$.
			Suppose that there exists a measure
			$\pi\in\xSx{E}_f(\xSp{X},\sigma)$
			such that $\Phi\pi\in\xSx{E}_f(\xSp{X},\sigma)$.
			By Corollary~\ref{thm:equilibrium:factor-map},
			we also have $\pi\in\xSx{E}_{f\xO\Phi}(\xSp{X},\sigma)$.
			Therefore, $\xSx{E}_f(\xSp{X},\sigma)\cap
			\xSx{E}_{f\xO\Phi}(\xSp{X},\sigma)\neq\varnothing$
			and by Proposition~\ref{prop:equilibrium:equivalence},
			$f$ and $f\xO\Phi$ are physically equivalent.
			That is, $\Phi$ conserves $f$.
	\end{enumerate}
\end{proof}

\begin{theorem}
\label{thm:conservation-invariance:2}
	Let $\Phi:\xSp{X}\to\xSp{X}$ be a surjective cellular automaton
	over a strongly irreducible shift of finite type $(\xSp{X},\sigma)$.
	Let $f\in C(\xSp{X})$ be an observable and $e\in\xR$.
	If $\Phi$ conserves $f$, then
	$\Phi$ maps $\xSx{E}_{\langle f\rangle=e}(\xSp{X},\sigma)$ onto itself.
\end{theorem}
%\begin{proof}
%	By Corollaries~\ref{cor:bayesian:entropy:factor-map}
%	and~\ref{thm:bayesian:solutions:factor-map}, we have
%	$\pi\in\xSx{E}_{\langle f\rangle=e}(\xSp{X},\sigma)$ if and only if
%	$\Phi\pi\in\xSx{E}_{\langle f\rangle=e}(\xSp{X},\sigma)$.
%	Using Lemma~\ref{lem:onto-one-to-one} we obtain
%	$\Phi\xSx{E}_{\langle f\rangle=e}(\xSp{X},\sigma)=
%	\xSx{E}_{\langle f\rangle=e}(\xSp{X},\sigma)$.
%\end{proof}

From Theorems~\ref{thm:conservation-invariance:1} and~\ref{thm:conservation-invariance:2}
it follows that each of the (convex and compact) sets $\xSx{E}_f(\xSp{X},\sigma)$ and
$\xSx{E}_{\langle f\rangle=e}(\xSp{X},\sigma)$
contains an invariant measure for $\Phi$, provided that $\Phi$ conserves $f$.
However, following the common reasoning of statistical mechanics
(see the \hyperref[sec:intro]{Introduction}),
such an invariant measure should not be considered as a macroscopic equilibrium state
unless it is shift-ergodic (see Example~\ref{exp:ising:ca:invariant-non-invariant} below).

In the implication (c $\Rightarrow$ a) of Theorem~\ref{thm:conservation-invariance:1},
the set $\xSx{E}_f(\xSp{X},\sigma)$ of equilibrium measures for $f$
can be replaced by the potentially larger set $\xSx{G}_{\Delta_f}(\xSp{X})$
of Gibbs measures for $\Delta_f$.
\begin{corollary}
\label{cor:conservation-invariance:1:non-shift-invariant}
	Let $\Phi:\xSp{X}\to\xSp{X}$ be a surjective cellular automaton
	over a strongly irreducible shift of finite type $(\xSp{X},\sigma)$,
	and let $f\in SV(\xSp{X})$ be an observable with summable variations.
	Suppose that there is a Gibbs measure for $\Delta_f$
	whose $\Phi$-image is also a Gibbs measure for $\Delta_f$.
	Then, $\Phi$ conserves~$\Delta_f$.
\end{corollary}
\begin{proof}
	Let $\pi$ be a probability measure on $\xSp{X}$ such that $\pi,\Phi\pi\in\xSx{G}_{\Delta_f}(\xSp{X})$.
	Let $\xSx{H}$ denote the closed convex hull of the measures $\sigma^k\pi$ for $k\in\xL$.
	Then, $\xSx{H}$ is a closed, convex, shift-invariant set, and therefore,
	contains a shift-invariant element $\nu$.
	Moreover, both $\xSx{H}$ and $\Phi\xSx{H}$ are subsets of $\xSx{G}_{\Delta_f}(\xSp{X})$.
	In particular, $\nu,\Phi\nu\in\xSx{G}_{\Delta_f}(\xSp{X})$.
	Hence, $\nu,\Phi\nu\in\xSx{E}_f(\xSp{X},\sigma)$,
	and the claim follows from Theorem~\ref{thm:conservation-invariance:1}.
\end{proof}
For reversible cellular automata,
Corollary~\ref{cor:gibbs:conjugacy} leads to a variant of Theorem~\ref{thm:conservation-invariance:1}
concerning all (not necessarily shift-invariant) Gibbs measures.
\begin{theorem}
\label{thm:conservation-invariance:reversible}
	Let $\Phi:\xSp{X}\to\xSp{X}$ be a reversible cellular automaton
	over a strongly irreducible shift of finite type $(\xSp{X},\sigma)$,
	and let $\Delta$ be a Hamiltonian on $\xSp{X}$.
	The following conditions are equivalent:
	\begin{enumerate}[ \rm a)]
		\item $\Phi$ conserves $\Delta$.
		\item A probability measure is in $\xSx{G}_\Delta(\xSp{X})$ if and only if
			its $\Phi$-image is in $\xSx{G}_\Delta(\xSp{X})$.
		\item There exists a measure in $\xSx{G}_\Delta(\xSp{X})$ whose $\Phi$-image
			is also in $\xSx{G}_\Delta(\xSp{X})$.
	\end{enumerate}
\end{theorem}
\begin{proof}
	If $\Phi$ conserves $\Delta$, we have, by definition, $\Phi^*\Delta=\Delta$,
	and Corollary~\ref{cor:gibbs:conjugacy} (and Lemma~\ref{lem:onto-one-to-one})
	imply that $\Phi^{-1}\xSx{G}_\Delta(\xSp{X})=\xSx{G}_\Delta(\xSp{X})$.
	Conversely, suppose that $\pi$ is a probability measure such that $\pi,\Phi\pi\in\xSx{G}_\Delta(\xSp{X})$.
	Then, by Corollary~\ref{cor:gibbs:conjugacy},
	$\pi\in\xSx{G}_\Delta(\xSp{X})\cap\xSx{G}_{\Phi^*\Delta}(\xSp{X})$,
	and it follows from the definition of a Gibbs measure that $\Phi^*\Delta=\Delta$.
	That is, $\Phi$ conserves $\Delta$.
\end{proof}

\begin{example}[Q2R cellular automaton]
\label{exp:ising:ca:invariant-non-invariant}
	The Q2R model discussed in the \hyperref[sec:intro]{Introduction}
	is not, strictly speaking, a cellular automaton (with the standard definition),
	as it involves alternate application of two maps that do not commute with the shift.
	Simple tricks can however be used to turn it into a standard cellular automaton
	(see e.g.~\cite{TofMar90}, Section~5.2).
	
	Let $\xSp{X}\IsDef\{\symb{+},\symb{-}\}^{\xZ^2}$ be the space of spin configurations,
	and denote by $\Phi_{\mathsf{e}}$ the mapping $\xSp{X}\to\xSp{X}$
	that updates the even sites.  That is,
	\begin{align}
		(\Phi_{\mathsf{e}} x)(i) &\IsDef \begin{cases}
			\overline{x_i}\qquad{}	& \text{if $i$ an even site and $n_i^{\symb{+}}(x)=n_i^{\symb{-}}(x)$,} \\
			x_i						& \text{otherwise,}
		\end{cases}
	\end{align}
	where the spin-flipping operation is denoted by overline, and $n_i^{\symb{+}}(x)$ (resp., $n_i^{\symb{-}}(x)$)
	represents the number of sites $j$ among the four immediate neighbors of $i$
	such that $x(j)=\symb{+}$ (resp., $x(j)=\symb{-}$).
	Similarly, let $\Phi_{\mathsf{o}}$ denotes the mapping that updates the odd sites.
	The composition $\Phi\IsDef \Phi_{\mathsf{o}}\Phi_{\mathsf{e}}$ commutes
	with the shifts $\sigma^k$, for $k$ in the sub-lattice $(2\xZ)^2$,
	and (after a recoding) could be considered as a cellular automaton.
	
	Let $f$ denote the energy observable defined in Example~\ref{exp:ising}.
	For every $\beta>0$, the Hamiltonian $\Delta_{\beta f}$ is conserved by $\Phi$.
	Therefore, according to Theorem~\ref{thm:conservation-invariance:reversible},
	the set $\xSx{G}_{\Delta_{\beta f}}(\xSp{X})$ of Gibbs measures for $\Delta_{\beta f}$
	is invariant under $\Phi$.
	In fact, in this example, it is easy to show that $\Phi$ preserves
	every individual Gibbs measure in $\xSx{G}_{\Delta_{\beta f}}(\xSp{X})$.
	
	It is natural to ask whether the preservation of individual elements
	of $\xSx{G}_{\Delta_{\beta f}}(\xSp{X})$
	holds in general.  This is however not the case.
	When $\beta$ large enough, it is known that $\xSx{G}_{\Delta_{\beta f}}(\xSp{X})$
	contains two distinct shift-ergodic measures,
	obtained from each other by a spin flip transformation (see Example~\ref{exp:ising-contour}).
	The cellular automaton $\Phi' x\IsDef \overline{\Phi x}$,
	which flips every spin after applying $\Phi$, conserves $\Delta_{\beta f}$
	but does not preserve either of the two distinct shift-ergodic Gibbs measures
	for $\Delta_{\beta f}$.
	\exampleqed
\end{example}

%========================================================================
\subsection{Absence of Conservation Laws}
%------------------------------------------------------------------------
\label{sec:ca:rigidity}
In light of the above connection,
every statement about conservation laws in surjective cellular automata
has an interpretation in terms of invariance of Gibbs measures, and vice versa.
In this section, we see an example of such reinterpretation
that leads to otherwise non-trivial results.
Namely, proving the abscence of conservation laws in
two relatively rich families of surjective and reversible cellular automata,
we obtain strong constraints on the invariant measures
of the cellular automata within each family.
Roughly speaking, strong chaotic behavior is incompatible
with the presence of conservation laws.
In contrast,
any surjective cellular automaton with a non-trivial equicontinuous factor has an infinity of
linearly independent 
conservation laws (see Example~\ref{exp:almost-equicontinuous}).

We say that a dynamical system $(\xSp{X},\Phi)$ is
\emph{strongly transitive}
if for every point $z\in\xSp{X}$,
the set $\bigcup_{i=0}^\infty \Phi^{-i}z$ 
is dense in $\xSp{X}$.
Strong transitivity is stronger than transitivity (!)\ and weaker than minimality.
A dynamical system $\Phi:\xSp{X}\to\xSp{X}$ is \emph{minimal} if it has no
non-trivial closed subsystems, and is \emph{transitive} if
for every pair of non-empty open sets $A,B\subseteq\xSp{X}$, there is an integer $n\geq 0$
such that $A\cap \Phi^{-n}B\neq\varnothing$.
In our setting (i.e., $\xSp{X}$ being compact), minimality is equivalent to the property that
the only closed sets $E\subseteq\xSp{X}$ with $E\subseteq\Phi E$ are $\varnothing$ and $\xSp{X}$,
which is easily seen to imply strong transitivity.
However, note that cellular automata over non-trivial strongly irreducible shifts of finite type
cannot be minimal.  This is because every strongly irreducible shift of finite type
has configurations that are periodic in at least one direction.
(More specifically, for each $k\in\xL\setminus\{0\}$, there is a configuration $x$
such that $\sigma^{pk}x=x$ for some $p>0$.)
Transitivity is often considered as one of the main indicators of chaos (see e.g.~\cite{Kol04,Bla09}).
Every transitive cellular automaton is known to be sensitive to initial conditions
(i.e., uniformly unstable)~\cite{CodMar96,Kur97}.\footnote{%
	In fact, transitive cellular automata are weakly mixing~\cite{Moo05},
	which, in addition to sensitivity, also implies chaos in the sense of Li and Yorke
	(see~\cite{Kol04,Bla09}).
}

\begin{example}[XOR cellular automata]
\label{exp:xor:ca}
	The $d$-dimensional \emph{XOR cellular automaton} with neighborhood $N\subseteq\xZ^d$
	is defined by the map $\Phi:\{\symb{0},\symb{1}\}^{\xZ^d}\to\{\symb{0},\symb{1}\}^{\xZ^d}$,
	where $(\Phi x)(i)\IsDef \sum_{i\in N} x(i) \pmod{2}$.
	To avoid the trivial case, we assume that the neighborhood has at least two elements.
	Examples~\ref{exp:xor:non-gibbs} and~\ref{exp:xor:non-gibbs:contd}
	were about the one-dimensional XOR cellular automaton with neighborhood $\{0,1\}$.
	Figure~\ref{fig:xor:transpose}a depicts a sample run of the one-dimensional model
	with neighborhood $\{-1,1\}$.
	
	The XOR cellular automaton is strongly transitive.
	An argument similar to that in
	Example~\ref{exp:xor:non-gibbs:contd} shows that
	the uniform Bernoulli measure is the only regular Gibbs measure that is
	invariant under an XOR cellular automaton.
	Note, however, that there are many other (non-Gibbs) invariant measures.
	For example, the Dirac measure concentrated at the uniform
	configuration with $\symb{0}$ everywhere is invariant.
	So is the (atomic) measure uniformly distributed over any
	jointly periodic orbit (i.e., a finite orbit of $(\sigma,\Phi)$).
	
	In fact, much more is known about the invariant measures
	of the XOR cellular automata,
	with a strong indication that the uniform Bernoulli measure is the only
	``state of macroscopic equilibrium''.
	For instance, the uniform Bernoulli measure on $\{0,1\}^\xZ$ is known to be
	the only shift-ergodic probability measure that is invariant 
	and of positive entropy for the XOR cellular automaton
	with neighborhood $\{0,1\}$~\cite{HosMaaMar03}.
	Another such result states that the only measures that are strongly mixing for the shift
	and invariant under the XOR cellular automaton with neighborhood $\{-1,1\}$
	are the uniform Bernoulli measure and the Dirac measure
	concentrated at the uniform configuration with $\symb{0}s$ everywhere~\cite{Miy79}.
	(Note that the one-dimensional Gibbs measures are all strongly mixing.)
	Similar results have been obtained for broad classes of cellular automata
	with algebraic structure (e.g.~\cite{Piv05,Sab07,Sob08}).
	See~\cite{Piv09} for a survey.
	\exampleqed
\end{example}

The following theorem is a slight generalization of Theorem~5 in~\cite{ForKarTaa11}.
\begin{theorem}
\label{thm:strongly-transitive-CA:no-cons-law}
	Let $\Phi:\xSp{X}\to\xSp{X}$ be a strongly transitive cellular automaton
	over a 
	shift of finite type $(\xSp{X},\sigma)$.
	Then, $\Phi$ does not conserve any non-trivial Hamiltonian.
\end{theorem}
\begin{proof}
	Let $0\in M\subseteq\xL$ be a finite window
	that witnesses the finite type gluing property of~$\xSp{X}$.
	
	Let $\Delta$ be a non-trivial Hamiltonian on $\xSp{X}$,
	and suppose there exist two asymptotic configurations
	$u$ and $v$ such that $\varepsilon\IsDef\Delta(u,v)>0$.
	By the continuity property of $\Delta$, there is a finite set $D\supseteq M(M^{-1}(\diff(u,v)))$
	such that for every two asymptotic configurations $u'\in [u]_D$ and $v'\in [v]_D$
	with $\diff(u',v')=\diff(u,v)$,
	$\Delta(u',v')\geq\varepsilon/2>0$.
	
	Let $z$ be a ground configuration for $\Delta$ (see Proposition~\ref{prop:Hamiltonian:ground}).
	Since $\Phi$ is strongly transitive, there is a configuration $x\in[v]_D$
	and a time $t\geq 0$ such that $\Phi^t x=z$.
	Construct a configuration $y\in\xSp{X}$ that agrees with $u$ on $D$
	and with $x$ outside $\diff(u,v)$.  In particular, $y\in[u]_D$.
	Then, $\Delta(y,x)\geq\varepsilon/2$, whereas $\Delta(\Phi^t y,\Phi^t x)\leq 0$.
	Therefore, $\Delta$ is not conserved by $\Phi$.
\end{proof}

\begin{corollary}
\label{cor:ca:strongly-transitive:rigidity}
	Let $\Phi:\xSp{X}\to\xSp{X}$ be a strongly transitive cellular automaton
	over a strongly irreducible shift of finite type $(\xSp{X},\sigma)$.
	Then, $\Phi$ does not preserve any regular Gibbs measure
	other than the Gibbs measures for the trivial Hamiltonian.
\end{corollary}
A special case of the above corollary (for the permutive cellular automata and
Bernoulli measures) is also proved in~\cite{BanChaChe11} (Corollary 3.6).
Let us recall that the \emph{shift-invariant} Gibbs measures for the trivial Hamiltonian on $\xSp{X}$
coincide with the measures of maximum entropy for $(\xSp{X},\sigma)$ (Theorem~\ref{thm:equilibrium:Gibbs}).
Therefore, according to Corollary~\ref{cor:ca:strongly-transitive:rigidity},
if $\Phi$ is strongly transitive,
the measures of maximum entropy for $(\xSp{X},\sigma)$ are the only
candidates for Gibbs measures that are invariant under both~$\sigma$ and~$\Phi$.
Since the set of measures with maximum entropy for $(\xSp{X},\sigma)$ is closed and convex,
and is preserved under $\Phi$, it follows that at least one measure with maximum entropy
is invariant under~$\Phi$.
However, this measure does not need to be ergodic for the shift.

Next, we are going to introduce a class of one-dimensional \emph{reversible} cellular automata
with no local conservation law.
The proof will be via reduction to Theorem~\ref{thm:strongly-transitive-CA:no-cons-law}.
Note that reversible cellular automata over non-trivial strongly irreducible shifts of finite type
cannot be strongly transitive:
the inverse of a strongly transitive system is minimal, and as mentioned above, cellular automata
over non-trivial strongly irreducible shifts of finite type cannot be minimal.
\begin{example}[Transpose of XOR]
\label{exp:xor:transpose}
	Figure~\ref{fig:xor:transpose}b depicts a sample space-time diagram of
	the reversible cellular automaton $\Phi$ on
	$(\{\symb{0},\symb{1}\}\times\{\symb{0},\symb{1}\})^\xZ$
	with neighbourhood $\{0,1\}$ and local rule $((a,b),(c,d))\mapsto (b,a+d)$,
	where the addition is modulo~$2$.
	Observe that rotating a space-time diagram of $\Phi$ by $90$ degrees,
	we obtain what is essentially a space-time diagram of
	the XOR cellular automaton with neighbourhood $\{-1,1\}$
	(see Figure~\ref{fig:xor:transpose}a and Example~\ref{exp:xor:ca}).
	
	As in Example~3 of~\cite{ForKarTaa11}, it is possible to show that $\Phi$ has
	no non-trivial finite-range conservation law.
	Below, we shall present an alternative proof
	(using its connection with the XOR cellular automaton) that covers
	a large class of similar reversible cellular automata.
	\exampleqed
\end{example}
\begin{figure}
	\begin{center}
		\begin{tabular}{ccc}
			\begin{minipage}[c]{0.35\textwidth}
				\centering
				\includegraphics[width=0.95\textwidth]{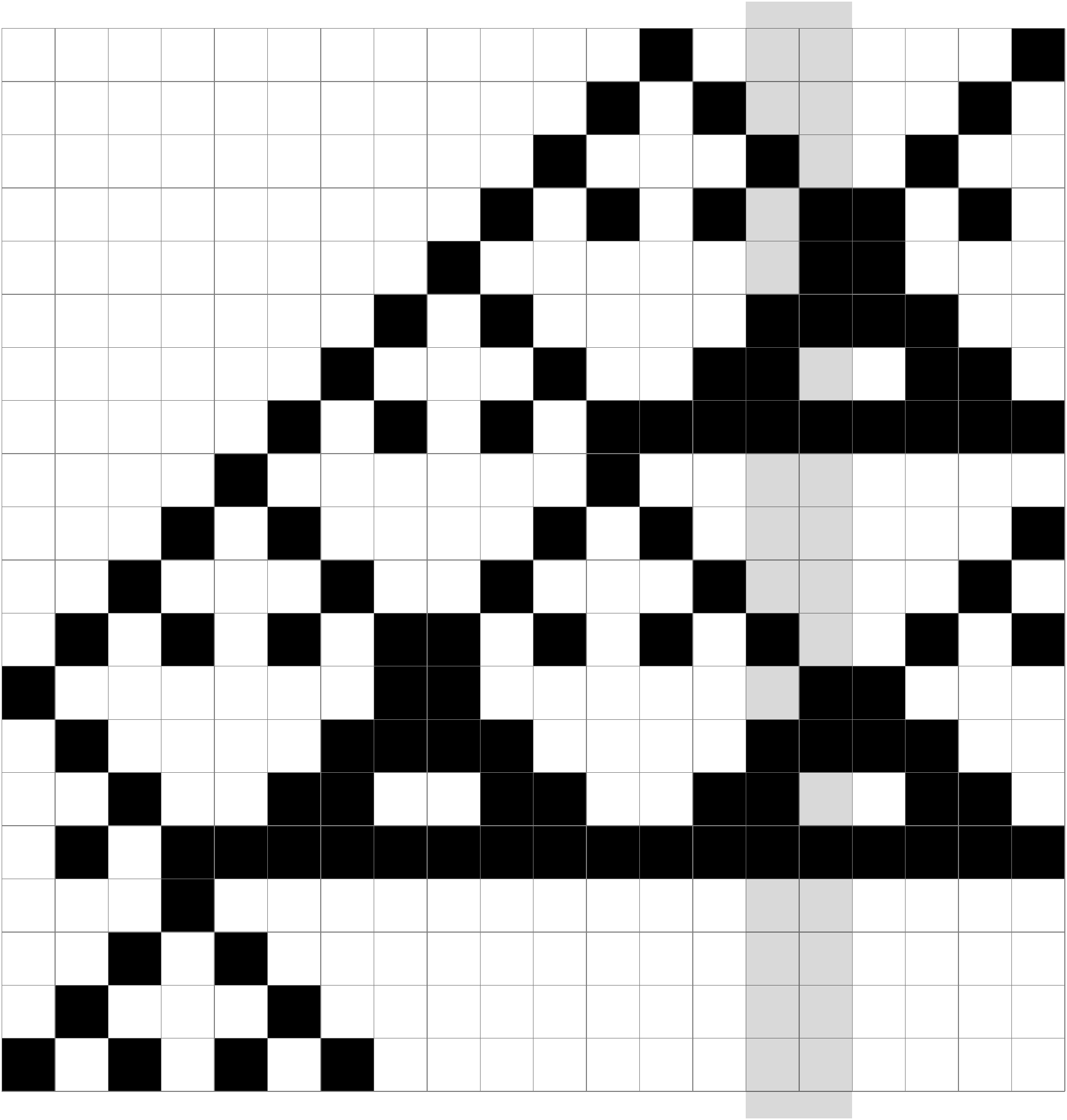}
			\end{minipage} & &
			\begin{minipage}[c]{0.35\textwidth}
				\centering
				\includegraphics[width=0.95\textwidth]{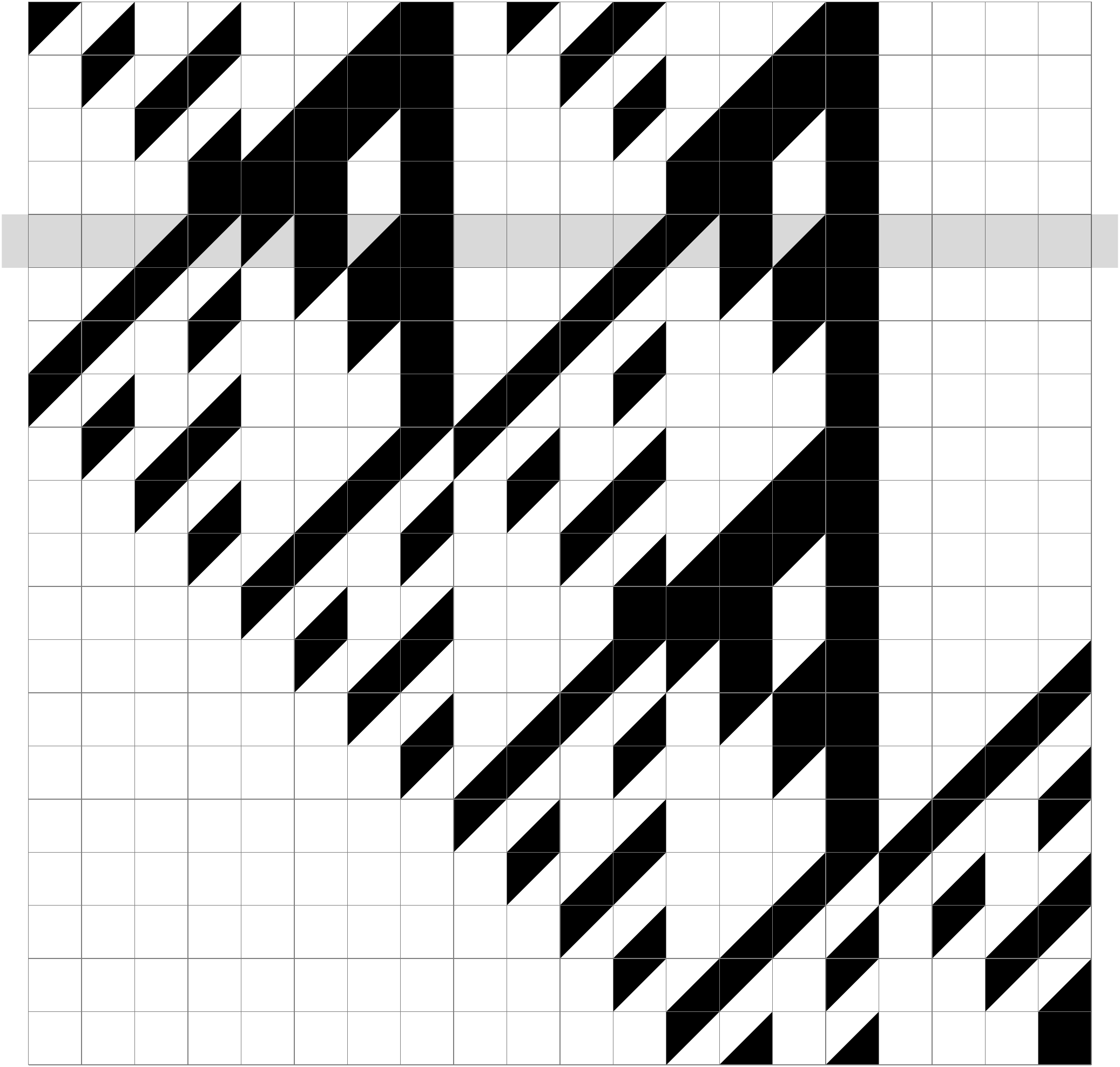}
			\end{minipage} \medskip\\
			(a) && (b)
		\end{tabular}
	\end{center}
	\caption{%
		Sample runs of (a) the one-dimensional XOR cellular automaton with
		neighborhood $\{-1,1\}$ (see Example~\ref{exp:xor:ca})
		and (b) its transpose (Example~\ref{exp:xor:transpose}).
		Time goes downward.
	}
	\label{fig:xor:transpose}
\end{figure}

We shall say that two surjective one-dimensional cellular automata are
\emph{transpose} of each other
if the bi-infinite space-time diagrams of each is obtained (up to a conjugacy)
from the bi-infinite space-time diagrams of the other by swapping the role of space and time.
To be more specific,
let $\Phi:\xSp{X}\to\xSp{X}$ be a surjective cellular automaton on a one-dimensional
mixing shift space of finite type $\xSp{X}\subseteq S^\xZ$.
Define the continuous map $\Theta:S^{\xZ\times\xZ}\to S^\xZ$
where $(\Theta z)(i)\IsDef z(i,0)$, and let $\tilde{\xSp{X}}$
be the two-dimensional shift space formed by all configurations $z\in S^{\xZ\times\xZ}$
such that
\begin{align}
	\ldots \,\xrightarrow[]{\Phi}\, \Theta\sigma^{(0,-1)}z \,\xrightarrow[]{\Phi}\,
		\Theta z \,\xrightarrow[]{\Phi}\, \Theta\sigma^{(0,1)} z \,\xrightarrow[]{\Phi}\, \ldots
\end{align}
is a bi-infinite orbit of $\Phi$, that is
$\Theta\sigma^{(0,k+1)} z = \Phi\Theta\sigma^{(0,k)} z$ for each $k\in\xZ$.
Set $\xVer\IsDef(0,1)$ and $\xHor\IsDef(1,0)$.
The dynamical system $(\tilde{\xSp{X}},\sigma^\xVer,\sigma^\xHor)$
(together with the map $\Theta$)
is the \emph{natural extension} of $(\xSp{X},\Phi,\sigma)$. 
Now, let $\Psi:\xSp{Y}\to\xSp{Y}$ be another surjective cellular automaton
on a one-dimensional mixing shift space of finite type $\xSp{Y}\subseteq T^\xZ$.
We say $\Psi$ is a \emph{transpose} of $\Phi$
if its natural extension is conjugate to $(\tilde{\xSp{X}},\sigma^\xHor,\sigma^\xVer)$.
The transpose of $\Phi$ (if it exists) is unique only up to conjugacy.
When there is no danger of confusion, we denote any representative
of the transpose conjugacy class by~$\Phi^\intercal$.

\begin{proposition}
\label{prop:transpose:mixing}
	A surjective cellular automaton on a one-dimensional mixing shift of finite type is mixing
	provided it has a transpose
	(acting on a mixing shift of finite type).
\end{proposition}
\begin{proof}
	A dynamical system is mixing if and only if its natural extension is mixing.
\end{proof}

Obviously, not every cellular automaton has a transpose.
A class of cellular automata that do have transposes
is the class of those that are positively expansive.
A dynamical system $(\xSp{X},\Phi)$ is \emph{positively expansive}
if there exists a real number $\varepsilon>0$ such that
for every two distinct points $x,y\in \xSp{X}$, there is a time $t\geq 0$
such that $\Phi^t x$ and $\Phi^t y$ have distance at least $\varepsilon$.
If $(\xSp{X},\sigma)$ is a mixing shift of finite type and $\Phi:\xSp{X}\to\xSp{X}$
is a positively expansive cellular automaton, then $\Phi$ is surjective, and it is known that a transpose of
$\Phi$ exists and is a reversible cellular automaton on a mixing shift of finite type
(see~\cite{Kur03b}, Section~5.5\footnote{%
	The proof in~\cite{Kur03b} is presented for the case that $(\xSp{X},\sigma)$ is a full shift,
	but the same proof, with slight adaptation, works for any arbitrary mixing shift of finite type.
	To prove the openness of $\Phi$, see~\cite{Nas83}, Theorems~6.3 and~6.4,
	and note that $\Phi$ is both left- and right-closing.
	
}).
If, furthermore, $(\xSp{X},\sigma)$ is a full shift, then the transpose of $\Phi$
also acts on a full shift (see~\cite{Nas95}, Theorem~3.12).

\begin{proposition}
	Every positively expansive cellular automaton on a one-dimensional
	mixing shift of finite type is strongly transitive.
\end{proposition}
\begin{proof}
	Any continuous map $\Phi:\xSp{X}\to\xSp{X}$ on a compact metric space
	that is transitive, open, and positively expansive is strongly transitive~\cite{Kam02}. 
	Every positively expansive cellular automaton on a mixing shift of finite type
	is itself mixing (see the above paragraph) and open (see~\cite{Kur03b}, Theorem~5.45).
	
	Alternatively, every positively expansive cellular automaton on a mixing shift of finite type
	is conjugate to a mixing \emph{one-sided} shift of finite type (see~\cite{Kur03b}, Theorem~5.49),
	and hence strongly transitive.
\end{proof}

The local conservation laws of a cellular automaton and its transpose
are in one-to-one correspondence.
\begin{theorem}
\label{thm:transpose:cons-law}
	Let $\Phi:\xSp{X}\to\xSp{X}$ and $\Phi^\intercal:\xSp{X}^\intercal\to\xSp{X}^\intercal$ be
	surjective cellular automata over one-dimensional mixing shifts of finite type
	$\xSp{X}$ and $\xSp{X}^\intercal$,
	and suppose that $\Phi$ and $\Phi^\intercal$ are transpose of each other.
	There is a one-to-one correspondence (up to local physical equivalence) between
	the observables $f\in K(\xSp{X})$
	that are locally conserved by $\Phi$ and the observables
	$f^\intercal\in K(\xSp{X}^\intercal)$ that
	are locally conserved by $\Phi^\intercal$.
	Moreover, $f$ is locally physically equivalent to $0$ if and only if $f^\intercal$ is so.
\end{theorem}
\begin{proof}
	Recall that an observable $f\in K(\xSp{X})$ is locally conserved by $\Phi$ if and only if
	it satisfies the continuity equation
	\begin{align}
		f\xO\Phi &= f + g\xO\sigma - g
	\end{align}
	for some observable $g\in K(\xSp{X})$, where the terms $g\xO\sigma$ and $g$
	are interpreted, respectively, as the flow pouring into a site from its right neighbour
	and the flow leaving that site towards its left neighbour.
	This equation may alternatively be written as
	\begin{align}
		g\xO\sigma &= g + f\xO\Phi - f \;,
	\end{align}
	which can be interpreted as the local conservation of the observable $g$
	when the role of $\Phi$ and $\sigma$ are exchanged.
	
	To specify the correspondence between local conservation laws of
	$\Phi$ and $\Phi^\intercal$ more precisely,
	let $\tilde{\xSp{X}}$ be the shift space of the space-time diagrams of $\Phi$, so that
	$(\tilde{\xSp{X}},\sigma^\xVer,\sigma^\xHor)$ is the natural extension
	of $(\xSp{X},\Phi,\sigma)$, and
	$(\tilde{\xSp{X}},\sigma^\xHor,\sigma^\xVer)$ is the natural extension
	of $(\xSp{X}^\intercal,\Phi^\intercal,\sigma)$,
	and let $\Theta:\tilde{\xSp{X}}\to\xSp{X}$ and
	$\Theta^\intercal:\tilde{\xSp{X}}\to\xSp{X}^\intercal$
	be, respectively, the corresponding factor maps, extracting (up to a conjugacy)
	the $0$th row and the $0$th column of $\tilde{\xSp{X}}$.
	
	Let us use the following notation.
	Suppose that local observables $f,g\in K(\xSp{X})$ and
	$f^\intercal,g^\intercal\in K(\xSp{X}^\intercal)$
	are such that $f\xO\Theta = g^\intercal\xO\Theta^\intercal$
	and $f^\intercal\xO\Theta^\intercal=g\xO\Theta$,
	and setting
	$\tilde{f}_\xVer\IsDef f\xO\Theta = g^\intercal\xO\Theta^\intercal$ and
	$\tilde{f}_\xHor\IsDef f^\intercal\xO\Theta^\intercal=g\xO\Theta$, it holds
	\begin{align}
		\tilde{f}_\xVer\xO\sigma^\xVer - \tilde{f}_\xVer &=
			\tilde{f}_\xHor\xO\sigma^\xHor - \tilde{f}_\xHor \;.
	\end{align}
	Then, we write $f_1\perp f^\intercal_1$ for any two local observables
	$f_1\in K(\xSp{X})$ and $f^\intercal_1\in K(\xSp{X})$
	that are locally physically equivalent to $f$ and $f^\intercal$, respectively.
	
	We verify that
	\begin{enumerate}[i)]
		\item a local observable $f\in K(\xSp{X})$ is locally conserved by $\Phi$
			if and only if $f\perp f^\intercal$ for some local observable
			$f^\intercal\in K(\xSp{X}^\intercal)$,
		\item the relation $\perp$ is linear, and
		\item $f\perp 0$ if and only if $f$ is locally physically equivalent to $0$.
	\end{enumerate}
	Note that these three statements (along with the similar statements obtained by
	swapping~$f$ and~$f^\intercal$) would imply that $\perp$
	is a one-to-one correspondence with the desired properties.
	
	To prove the first statement, suppose that $f\perp f^\intercal$.  Then,
	\begin{align}
		\tilde{f}_\xVer\xO\sigma^\xVer - \tilde{f}_\xVer &=
			\tilde{f}_\xHor\xO\sigma^\xHor - \tilde{f}_\xHor \;,
	\end{align}
	where $\tilde{f}_\xVer = f\xO\Theta$ and $\tilde{f}_\xHor=g\xO\Theta$ for some $g\in K(\xSp{X})$.
	Rewriting this equation as
	\begin{align}
		(f\xO\Phi - f)\xO\Theta &= (g\xO\sigma - g)\xO\Theta \;,
	\end{align}
	we obtain, using Lemma~\ref{lem:onto-one-to-one} and the surjectivity of $\Theta$, that
	$f$ is locally conserved by $\Phi$.
	Conversely, suppose that $\Phi$ locally conserves $f$,
	and let $g\in K(\xSp{X})$ be such that $f\xO\Phi - f=g\xO\sigma - g$.
	Therefore,
	\begin{align}
		f\xO\Theta\xO\sigma^\xVer - f\xO\Theta &= g\xO\Theta\xO\sigma^\xHor - g\xO\Theta \;.
	\end{align}
	Since $f\xO\Theta$ and $g\xO\Theta$ are local observables,
	there exists a finite region $D\subseteq\xZ\times\xZ$ such that
	$f\xO\Theta,g\xO\Theta\in K_D(\tilde{\xSp{X}})$.
	By the definition of natural extension, there is an integer $k>0$ such that
	$\xRest{z}{D}$ is uniquely and continuously determined by $\Theta^\intercal\sigma^{-k\xHor}z$
	(i.e., column $-k$ of the space-time shift of $\Phi$).
	Hence, there exist observables $f^\intercal,g^\intercal\in K(\xSp{X}^\intercal)$ such that
	$f\xO\Theta\xO\sigma^{k\xHor}=g^\intercal\xO\Theta^\intercal$ and
	$g\xO\Theta\xO\sigma^{k\xVer}=f^\intercal\xO\Theta^\intercal$.
	Now, setting $\tilde{f}_\xVer\IsDef f\xO\sigma^k\xO\Theta=g^\intercal\xO\Theta^\intercal$
	and $\tilde{f}_\xHor\IsDef g\xO\sigma^k\xO\Theta=f^\intercal\xO\Theta^\intercal$, we can write
	\begin{align}
		\tilde{f}_\xVer\xO\sigma^\xVer - \tilde{f}_\xVer &=
			\tilde{f}_\xHor\xO\sigma^\xHor - \tilde{f}_\xHor \;,
	\end{align}
	which means $f\xO\sigma^k \perp f^\intercal$.
	Finally, note that $f$ and $f\xO\sigma^k$ are locally physically equivalent.
	
	The linearity of $\perp$ and the fact that $0\perp 0$ are clear.
	It remains to show that if $f_1\in K(\xSp{X})$ is a local observable such that $f_1\perp 0$,
	then $f_1$ is locally physically equivalent to $0$.
	Suppose that $f_1\perp 0$.  Then, there is an observable $f\in K(\xSp{X})$
	locally physically equivalen to $f_1$, and an observable $f^\intercal\in K(\xSp{X}^\intercal)$
	locally physically equivalent to $0$ such that
	\begin{align}
		f\xO\Theta\xO\sigma^\xVer - f\xO\Theta &=
			f^\intercal\xO\Theta^\intercal\xO\sigma^\xHor - f^\intercal\xO\Theta^\intercal \;.
	\end{align}
	Since $f^\intercal$ is locally physically equivalent to $0$,
	it has the form $f^\intercal=h^\intercal\xO\sigma - h^\intercal + c$
	for some observable $h^\intercal\in K(\xSp{X}^\intercal)$ and some constant $c\in\xR$.
	Therefore,
	\begin{align}
	\label{eq:thm:transpose:1}
		f\xO\Theta\xO\sigma^\xVer - f\xO\Theta &=
			h^\intercal\xO\Theta^\intercal\xO\sigma^\xVer\xO\sigma^\xHor -
				h^\intercal\xO\Theta^\intercal\xO\sigma^\xVer -
				h^\intercal\xO\Theta^\intercal\xO\sigma^\xHor +
				h^\intercal\xO\Theta^\intercal \;.
	\end{align}
	Since $h^\intercal$ is a local observable, we can find, as before,
	an integer $l>0$ and a local observable $h\in K(\xSp{X})$
	such that $h^\intercal\xO\sigma^l\xO\Theta^\intercal = h\xO\Theta$.
	Therefore, composing both sides of~(\ref{eq:thm:transpose:1}) with $\sigma^{l\xVer}$ leads to
	\begin{align}
		f\xO\Phi^l\xO\Theta\xO\sigma^\xVer - f\xO\Phi^l\xO\Theta &=
			h\xO\Theta\xO\sigma^\xVer\xO\sigma^\xHor -
				h\xO\Theta\xO\sigma^\xVer -
				h\xO\Theta\xO\sigma^\xHor +
				h\xO\Theta \;,
	\end{align}
	which, together with Lemma~\ref{lem:onto-one-to-one}, gives
	\begin{align}
		f\xO\Phi^l\xO\Phi - f\xO\Phi^l &= h\xO\Phi\xO\sigma - h\xO\Phi - h\xO\sigma + h \;.
	\end{align}
	The latter equation can be rewritten as
	\begin{align}
		(f\xO\Phi^l - h\xO\sigma + h)\xO\Phi &= (f\xO\Phi^l - h\xO\sigma + h) \;,
	\end{align}
	which says that $f\xO\Phi^l - h\xO\sigma + h$ is invariant under $\Phi$.
	On the other hand, since $(\xSp{X}^\intercal,\sigma)$ is a mixing shift,
	it follows from Proposition~\ref{prop:transpose:mixing} that $(\xSp{X},\Phi)$ is also mixing.
	As a consequence, every continuous observable that is invariant under $\Phi$ is constant.
	In particular, $f\xO\Phi^l - h\xO\sigma + h=c'$ for some constant $c'\in \xR$,
	which means $f\xO\Phi^l$ is locally physically equivalent to $0$.
	Since $f$ is locally conserved by $\Phi$, the observable $f\xO\Phi^l$ is also
	locally physically equivalent to $f$, and this completes the proof.
\end{proof}

\begin{corollary}
\label{cor:ca:pexp-transpose:no-cons-law}
	Let $\Phi:\xSp{X}\to\xSp{X}$ be a reversible cellular automaton on
	a one-dimensional mixing shift of finite type $(\xSp{X},\sigma)$,
	and suppose that $\Phi$ has a positively expansive transpose.
	Then, $\Phi$ has no non-trivial local conservation law.
\end{corollary}

As mentioned in Section~\ref{sec:ca:conservation-laws},
for cellular automata on full shifts, every conserved local observable
is locally conserved.
\begin{corollary}
\label{cor:ca:pexp-transpose:rigidity}
	Let $\Phi:\xSp{X}\to\xSp{X}$ be a reversible cellular automaton on
	a one-dimensional full shift $(\xSp{X},\sigma)$,
	and suppose that $\Phi$ has a positively expansive transpose.
	The uniform Bernoulli measure is the only finite-range Gibbs measure
	($\equiv$ full-support Markov measure) that is invariant under~$\Phi$.
\end{corollary}

\begin{example}[Non-additive positively expansive]
\label{exp:permutive:transpose}
	Let $\Phi:\{\symb{0},\symb{1},\symb{2}\}^\xZ\to\{\symb{0},\symb{1},\symb{2}\}^\xZ$
	be the cellular automaton defined by neighborhood $N\IsDef\{-1,0,1\}$ and
	local update rule $\varphi:\{\symb{0},\symb{1},\symb{2}\}^N\to \{\symb{0},\symb{1},\symb{2}\}$
	defined by
	\begin{align}
		\varphi(a,b,c) &\IsDef \begin{cases}
			a+c+\symb{1} \pmod{3}\qquad\ 	& \text{if $b=\symb{2}$,}\\
			a+c \pmod{3}						& \text{otherwise.}
		\end{cases}
	\end{align}
	See Figure~\ref{fig:permutive:transpose}a for a sample run.
	Note that the local rule is both left- and right-\emph{permutive}
	(i.e., $a\mapsto\varphi(a,b,c)$ and $c\mapsto\varphi(a,b,c)$ are permutations).
	It follows that $\Phi$ is positively expansive, and hence also strongly transitive.
	Therefore, according to Theorem~\ref{thm:strongly-transitive-CA:no-cons-law} and
	Corollary~\ref{cor:ca:strongly-transitive:rigidity},
	$\Phi$ has no non-trivial conservation law and
	the uniform Bernoulli measure is the only regular Gibbs measure
	on $\{\symb{0},\symb{1},\symb{2}\}^\xZ$ that is invariant under $\Phi$.
	
	The cellular automaton $\Phi$ has a transpose
	$\Phi^\intercal:(\{\symb{0},\symb{1},\symb{2}\}\times\{\symb{0},\symb{1},\symb{2}\})^\xZ \to
		(\{\symb{0},\symb{1},\symb{2}\}\times\{\symb{0},\symb{1},\symb{2}\})^\xZ$
	defined 	with neighborhood $\{0,1\}$ and local rule
	\begin{align}
		((a,b), (a',b')) &\mapsto
			\begin{cases}
				(b, b'-a-\symb{1})\qquad\		&\text{if $b=\symb{2}$,}\\
				(b, b'-a)						&\text{otherwise,}\\
			\end{cases}
	\end{align}
	where the subtractions are modulo~$3$ (see Figure~\ref{fig:permutive:transpose}b).
	This is a reversible cellular automaton.
	It follows from Corollaries~\ref{cor:ca:pexp-transpose:no-cons-law}
	and~\ref{cor:ca:pexp-transpose:rigidity}
	that $\Phi^\intercal$ has no non-trivial local conservation law
	and no invariant full-support Markov measure other than the uniform Bernoulli measure.
	\exampleqed
\end{example}
\begin{figure}
	\begin{center}
		\begin{tabular}{ccc}
			\begin{minipage}[c]{0.35\textwidth}
				\centering
				\includegraphics[width=0.95\textwidth]{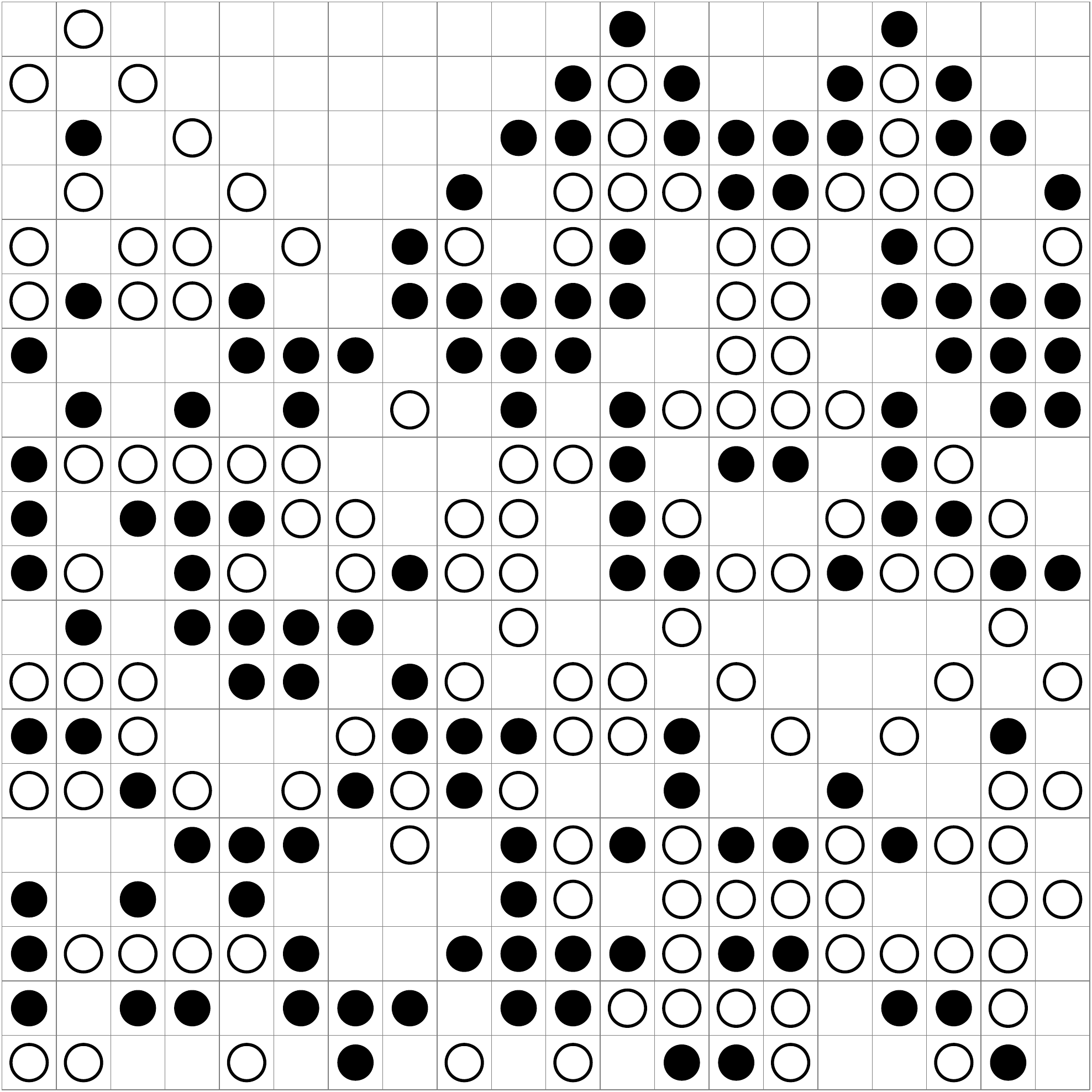}
			\end{minipage} & &
			\begin{minipage}[c]{0.35\textwidth}
				\centering
				\includegraphics[width=0.95\textwidth]{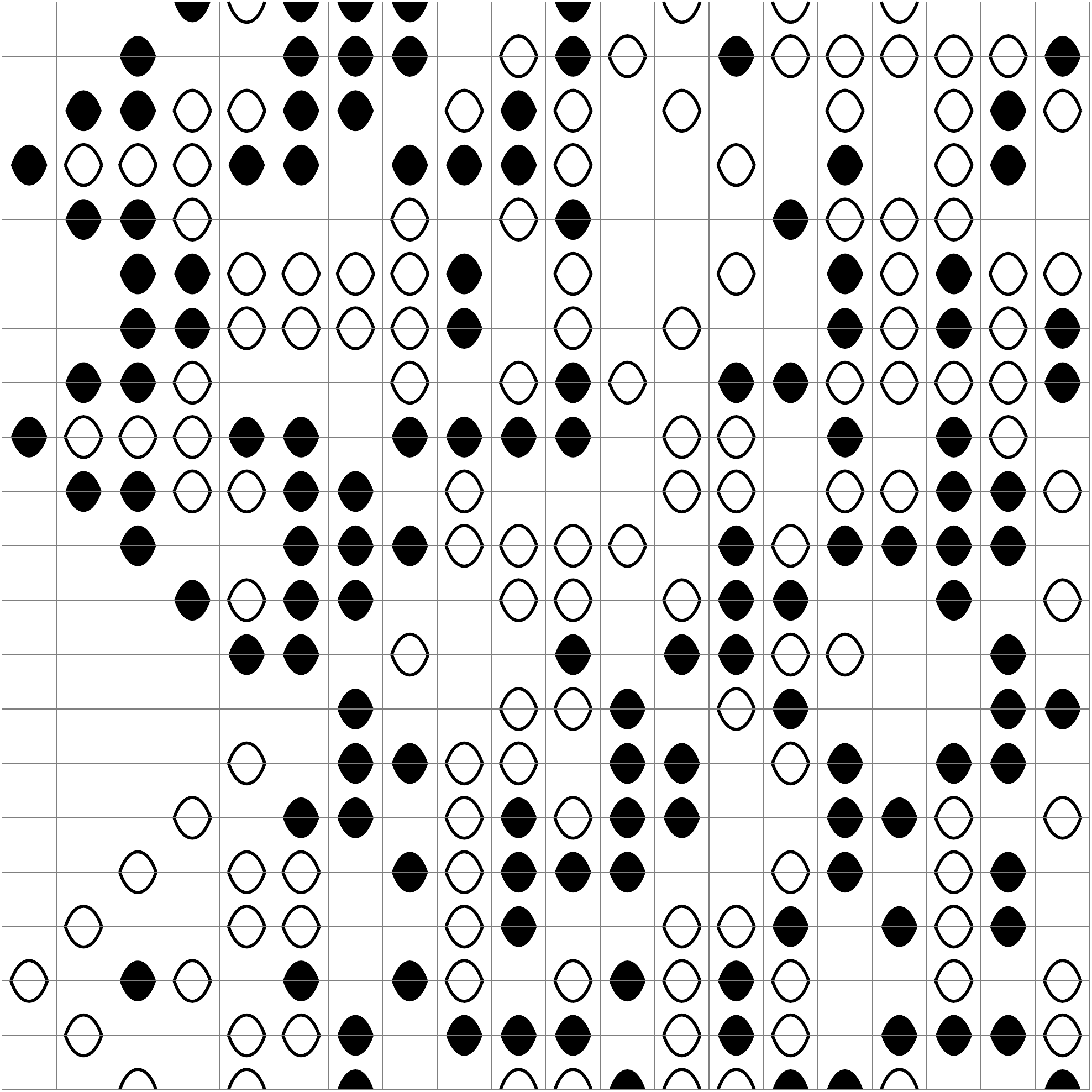}
			\end{minipage} \medskip\\
			(a) && (b)
		\end{tabular}
	\end{center}
	\caption{%
		Sample runs of (a) a one-dimensional bi-permutive cellular automaton
		and (b) its transpose.  See Example~\ref{exp:permutive:transpose}.
		Time goes downward.
		Black circle represents~$\symb{2}$ and white circle~$\symb{1}$.
	}
	\label{fig:permutive:transpose}
\end{figure}

According to Corollary~\ref{cor:ca:strongly-transitive:rigidity},
if $\Phi$ is a strongly transitive cellular automaton on a full shift $\xSp{X}$,
the uniform Bernoulli measure on $\xSp{X}$ is the only regular Gibbs measure
that is preserved by $\Phi$.
Likewise, Corollary~\ref{cor:ca:pexp-transpose:rigidity} states that
for a class of one-dimensional reversible cellular automata,
the uniform Bernoulli measure is the only invariant full-support Markov measure.
Note that even with these constraints, a cellular automaton in either of these
two classes still has a large collection of other invariant measures.
For example, for every $d$ linearly independent vectors $k_1,k_2,\ldots,k_d\in\xZ^d$,
the set of $d$-dimensional spatially periodic configurations having $k_i$ as periods
(i.e., $\{x: \sigma^{k_i} x=x \text{ for $i=1,2,\ldots,d$}\}$)
is finite and invariant under any cellular automaton,
and therefore any cellular automaton has an (atomic) invariant measure supported at such a set.
Nevertheless, if we restrict our attention to sufficiently ``smooth'' measures,
the uniform Bernoulli measure becomes the ``unique'' invariant measure for
a cellular automaton in either of the above classes.\footnote{
	The term ``smoothness'' here refers to
	the continuity of the conditional probabilities $\pi([p]_D\,|\,\family{F}_D)(z)$
	for Gibbs measures (which is a defining proeprty).
	Unfortunately, Corollary~\ref{cor:ca:strongly-transitive:rigidity}
	restricts only the invariance of \emph{regular} Gibbs measures
	(see Section~\ref{sec:hamiltonian-gibbs}).
	We do not know if every Hamiltonian is generated by an observable with
	summable variations.  However, see~\cite{Sul73} in this direction.
}
In this sense, Corollaries~\ref{cor:ca:strongly-transitive:rigidity}
and~\ref{cor:ca:pexp-transpose:rigidity}
may be interpreted as weak indications of ``absence of phase transition''
for cellular automata in the two classes in question.

\begin{question}
	Let $(\xSp{X},\sigma)$ be a strongly irreducible shift of finite type.
	Which shift-ergodic measures can be invariant under a strongly transitive cellular automaton?
	Can a shift-ergodic measure with positive but sub-maximum entropy on $(\xSp{X},\sigma)$
	be invariant under a strongly transitive cellular automaton?
\end{question}

%========================================================================
\subsection{Randomization and Approach to Equilibrium}
%------------------------------------------------------------------------
\label{sec:ca:randomization}

This section contains a few remarks and open questions regarding
the problem of approach to equilibrium in surjective cellular automata.

\begin{example}[Randomization in XOR cellular automata]
\label{exp:xor:randomization}
	The XOR cellular automata
	(Examples~\ref{exp:xor:ca},~\ref{exp:xor:non-gibbs} and~\ref{exp:xor:non-gibbs:contd})
	exhibit the same kind of ``approach to equilibrium''
	as observed in the Q2R model (see the \hyperref[sec:intro]{Introduction}).
	Starting from a biased Bernoulli random configuration,
	the system quickly reaches a uniformly random state, where it remains
	(see Figure~\ref{fig:xor:randomization}).
	A mathematical explanation of this behavior was first found independently
	by Miyamoto~\cite{Miy79} and Lind~\cite{Lin84} (following Wolfram~\cite{Wol83})
	and has since been extended and strengthened by others.	
	
	Let $\xSp{X}\IsDef\{\symb{0},\symb{1}\}^\xZ$,
	and consider the XOR cellular automaton $\Phi:\xSp{X}\to\xSp{X}$
	with neighborhood $\{0,1\}$.
	If $\pi$ is a shift-invariant probability measure on $\xSp{X}$,
	the convergence of $\Phi^t\pi$ as $t\to\infty$ fails
	as long as $\pi$ is strongly mixing 
	and different from the uniform Bernoulli measure and the Dirac measures
	concentrated at one of the two uniform configurations~\cite{Miy79,Miy94}.
	However, if $\pi$ is a non-degenerate Bernoulli measure,
	the convergence holds if a negligible set of time steps are ignored.
	More precisely, there is a set $J\subseteq\xN$ of density~$1$
	such that for every non-degenerate Bernoulli measure $\pi$,
	the sequence $\{\Phi^t\pi\}_{t\in J}$ converges, as $t\to\infty$,
	to the uniform Bernoulli measure~$\mu$.
	In particular, we have the convergence of the Ces\`aro averages
	\begin{align}
		\frac{1}{n}\sum_{t=0}^{n-1}\Phi^t\pi\to\mu
	\end{align}
	as $n\to\infty$~\cite{Miy79,Lin84}.
	
	The same type of convergence holds
	as long as $\pi$ is harmonically mixing~\cite{PivYas02}.
	Similar results have been obtained for
	a wide range of algebraic cellular automata
	(see e.g.~\cite{CaiLuo93,MaaMar98,FerMaaMarNey00,PivYas02,
	HosMaaMar03,PivYas04,PivYas06,Sob08}).
	In particular, the reversible cellular automaton of Example~\ref{exp:xor:transpose}
	has been shown to have the same randomizing effect~\cite{MaaMar99}.
	See~\cite{Piv09} for a survey.
	
	It is also worth mentioning a similar result due to Johnson and Rudolph~\cite{JohRud95}
	regarding maps of the unit circle $\xT\IsDef\xR/\xZ$.
	Namely, let $\pi$ be a Borel measure on $\xT$.
	They showed that if $\pi$ is invariant, ergodic
	and of positive entropy for the map $3\times: x\mapsto 3x \pmod{1}$,
	then it is randomized by the map $2\times: x\mapsto 2x \pmod{1}$,
	in the sense that $(2\times)^t\pi$ converges to the Lebesgue measure
	along a subsequence $J\subseteq\xN$ of density~$1$.
	\exampleqed
\end{example}
\begin{figure}
	\begin{center}
		\includegraphics[width=0.6\textwidth]{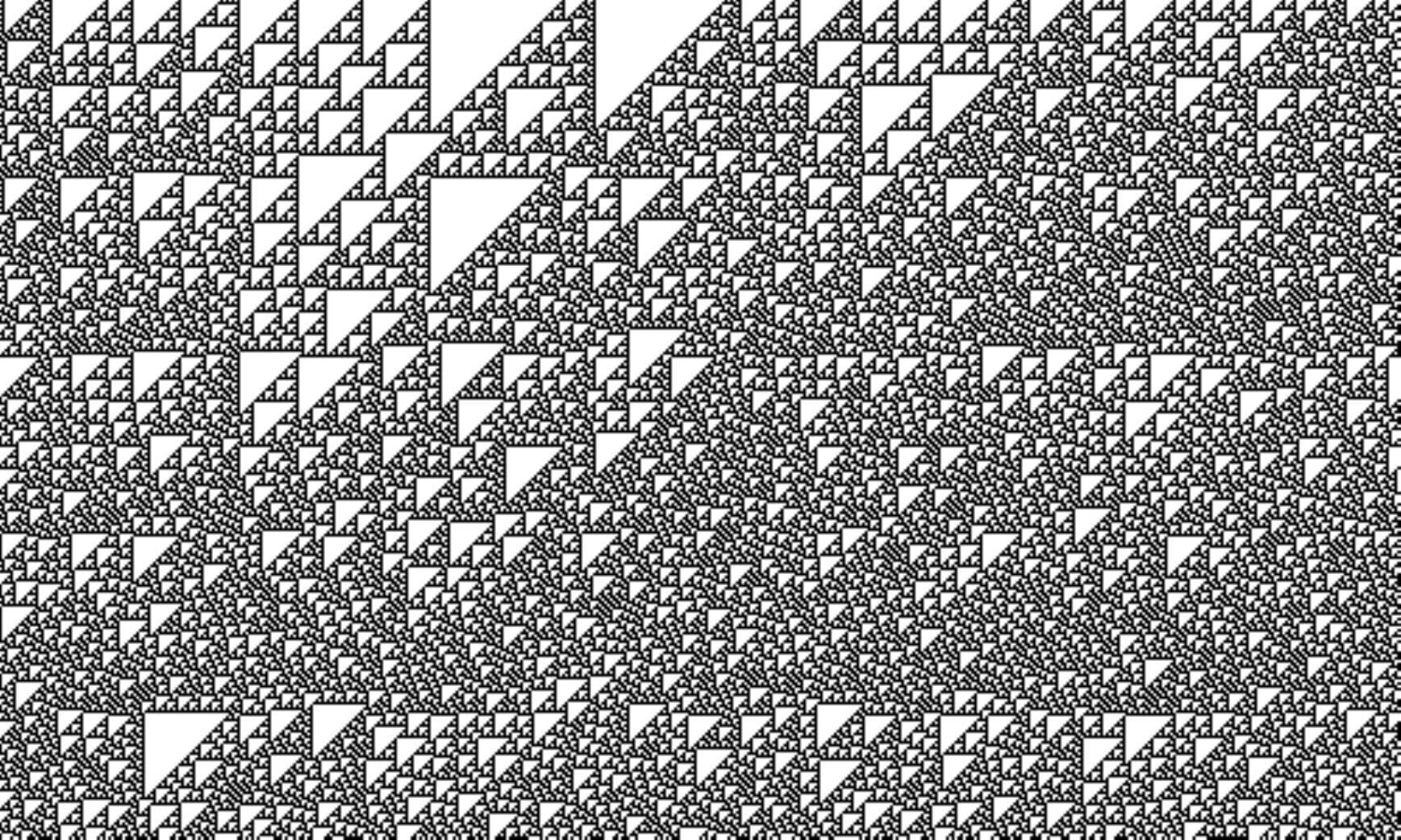}
	\end{center}
	\caption{%
		Randomization effect of the XOR cellular automaton (see Example~\ref{exp:xor:randomization}).
		Time goes downward.
		The initial configuration is chosen using coin flips with bias $1{\,:\,}9$.
		The initial density of symbol $\symb{1}$ is $0.104$.  %${\sim} 0.09$.
		The density at time $300$ is $0.504$.  %${\sim} 0.49$.
	}
	\label{fig:xor:randomization}
\end{figure}

Randomization behavior similar to that in the XOR cellular automaton
has been observed in simulations of other (non-additive) cellular automata,
but the mathematical results are so far limited to algebraic cellular automata.
The uniform Bernoulli measure is the unique measure with maximum entropy
on the full shift $(\xSp{X},\sigma)$ (i.e., the ``state of maximum randomness'').
The convergence (in density) of $\Phi^t\pi$ to the uniform Bernoulli measure
may thus be interpreted as a manifestation of 
the second law of thermodynamics~\cite{Piv09}.

We say that a cellular automaton $\Phi:\xSp{X}\to\xSp{X}$ 
(asymptotically) \emph{randomizes} a probability measure $\pi\in\xSx{P}(\xSp{X})$,
if there is a set $J\subseteq\xN$ of density~$1$ such that the weak limit
\begin{align}
	\Phi^\infty\pi &\IsDef \lim_{J\ni t\to\infty}\Phi^t\pi
\end{align}
exists and is a shift-invariant measure
with maximum entropy, that is, $h_{\Phi^\infty\pi}(\xSp{X},\sigma)=h(\xSp{X},\sigma)$.
The \emph{density} of a set $J\subseteq\xN$ is defined as
\begin{align}
	d(J) &\IsDef \lim_{n\to\infty} \frac{
		J \cap \{0,1,\ldots,n-1\}
	}{n} \;.
\end{align}
Note that the limit measure $\Phi^\infty\pi$ must be invariant under $\Phi$,
even if $(\xSp{X},\sigma)$ has multiple measures with maximum entropy.
If $\Phi$ randomizes a measure $\pi$, the Ces\`aro averages
$(\sum_{t=0}^{n-1} \Phi^t\pi)/n$ will also converge to $\Phi^\infty\pi$.
The converse is also true as long as $\pi$ is shift-invariant
and the limit measure is shift-ergodic:

\begin{lemma}[see~\cite{JohRud95}, Corollary~1.4] 
\label{lem:density-vs-cesaro}
	Let $\xSp{X}$ be a compact metric space
	and $\xSx{Q}\subseteq\xSx{P}(\xSp{X})$ a closed and convex set of
	probability measures on $\xSp{X}$.
	Let $\pi_1,\pi_2,\ldots$ be a sequence of elements in $\xSx{Q}$
	whose Ces\`aro averages $(\sum_{i=0}^{n-1} \pi_i)/n$
	converge to a measure $\mu$ as $n\to\infty$.
	If $\mu$ is extremal in $\xSx{Q}$,
	then there is a set $J\subseteq\xN$ of density~$1$
	such that $\pi_i\to\mu$ as $J\ni i\to\infty$.	
\end{lemma}

As mentioned in Example~\ref{exp:xor:randomization},
the stronger notion of randomization fails for the XOR cellular automaton.
We say that a cellular automaton $\Phi$ \emph{strongly randomizes} a measure
$\pi$ if $\Phi^t\pi$ converges to a measure with maximum entropy.
\begin{question}
	Are there examples of surjective or reversible cellular automata
	that strongly randomize all (say) Bernoulli measures?
	Is there a generic obstacle against strong randomization
	in surjective or reversible cellular automata?
\end{question}

If the cellular automaton $\Phi$ has non-trivial conservation laws,
the orbit of a measure $\pi$ will be entirely on the same ``energy level''.
Nevertheless, we could expect $\pi$ to be randomized within its energy level.
To evade an abundance of invariant measures, let us assume that
$\Phi$ has only finitely many linearly independent conservation laws.
More precisely, let $F=\{f_1,f_2,\ldots,f_n\}\subseteq C(\xSp{X})$
be a collection of observables conserved by $\Phi$ such that
every observable $g\in C(\xSp{X})$ conserved by $\Phi$
is physically equivalent to an element of the linear span of $F$.
The measures $\Phi^t\pi$ as well as their accumulation points are confined in the closed convex set
\begin{align}
	\{\nu\in\xSx{P}(\xSp{X}): \text{$\nu(f_i)=\pi(f_i)$ for $i=1,2,\ldots,n$} \} \;.
\end{align}
Let us say that $\Phi$ \emph{randomizes} $\pi$ \emph{modulo $F$}
if there is a set $J\subseteq\xN$ of density~$1$ such that
\begin{align}
	\Phi^\infty\pi &\IsDef \lim_{J\ni t\to\infty}\Phi^t\pi
\end{align}
exists, is shift-invariant, and has entropy $s_{f_1,f_2,\ldots,f_n}(\pi(f_1),\pi(f_2),\ldots,\pi(f_n))$,
where
\begin{align}
	s_{f_1,f_2,\ldots,f_n}(e_1,e_2,\ldots,e_n)
		&=
		\sup\{h_\nu(\xSp{X},\sigma):
			\text{$\nu\in\xSx{P}(\xSp{X},\sigma)$ and $\nu(f_i)=e_i$ for $i=1,2,\ldots,n$} \} \;.
\end{align}

\begin{question}
	What are some examples of non-algebraic cellular automata
	(with or without non-trivial conservation laws) having a randomization property?
\end{question}

Suitable candidates to inspect for the occurrence of a randomization behavior
are those that do not have any non-trivial conservation laws.

\begin{question}
	Do strongly transitive cellular automata randomize every Gibbs measure?
\end{question}

\begin{question}
	Does a one-dimensional reversible cellular automaton that has a positively expansive transpose
	randomize every Gibbs measure?
\end{question}

%%xxxxxxxxxxxxxxxxxxxxxxxxxxxxxxxxxxxxxxxxxxxxxxxxxxxxxxxxxxxxxxxxxxxxxxxx
%\section{Algorithmic Aspects}
%%------------------------------------------------------------------------

%xxxxxxxxxxxxxxxxxxxxxxxxxxxxxxxxxxxxxxxxxxxxxxxxxxxxxxxxxxxxxxxxxxxxxxxx
\section{Conclusions}
%------------------------------------------------------------------------

There is a wealth of open issues in connection with the statistical mechanics
of reversible and surjective cellular automata.
We have asked a few questions in this article.
From the modeling point of view,
there are at least three central problems that need to be addressed:
\begin{itemize}
	\item What is a good description of macroscopic equilibrium states?
	\item What is a satisfactory description of approach to equilibrium?
	\item How do physical phenomena such as phase transition appear
		in the dynamical setting of cellular automata?
\end{itemize}

By virtue of their symbolic nature,
various questions regarding cellular automata can be conveniently
approached using computational and algorithmic methods.
Nevertheless, many fundamental global properties of cellular automata
have turned out to be algorithmically undecidable, at least in two and higher dimensions.
For example, the question of whether a given two-dimensional cellular automaton is
reversible (or surjective) is undecidable~\cite{Kar94a}.
Similarly, \emph{all} non-trivial properties of the limit sets of cellular automata
are undecidable, even when restricted to the one-dimensional case~\cite{Kar94b,GuiRic10}
(see also~\cite{Del11}).
Whether a given cellular automaton on a full shift conserves a given
local observable can be verified using a simple algorithm~\cite{HatTak91},
but whether a (one-dimensional) cellular automaton has any non-trivial
local conservation law is undecidable~\cite{ForKarTaa11}.
It is an interesting open problem whether the latter undecidability statement
remains true when restricted to the class of reversible (or surjective) cellular automata.
We hope to address this and other algorithmic questions related to
the statistical mechanics of cellular automata in a separate study.

Problems similar to those studied here have been addressed
in different but related settings and with various motivations.
Simple necessary and sufficient conditions have been obtained
that characterize when a one-dimensional probabilistic cellular automaton
has a Bernoulli or Markov invariant measure~\cite{TooVasStaMitKurPir90,MaiMar12}.
The equivalence of parts~(b) and~(c) in Theorem~\ref{thm:conservation-invariance:1}
is also true for positive-rate probabilistic cellular automata~\cite{DaiLuiRoe02}.
For positive-rate probabilistic cellular automata, however, the existence
of an invariant Gibbs measure implies that all shift-invariant invariant Gibbs measures
are Gibbs for the same Hamiltonian!
The ergodicity problem of the probabilistic cellular automata (see e.g.~\cite{TooVasStaMitKurPir90})
has close similarity with the problem of randomization in surjective cellular automata.

%xxxxxxxxxxxxxxxxxxxxxxxxxxxxxxxxxxxxxxxxxxxxxxxxxxxxxxxxxxxxxxxxxxxxxxxx
\section*{Acknowledgements}
%------------------------------------------------------------------------
We would like to thank Aernout van Enter,  Nishant Chandgotia,
Felipe Garc\'ia-Ramos, Tom Kempton and Marcus Pivato
for helpful comments and discussions.

\addcontentsline{toc}{section}{References}
\bibliographystyle{plain}
\bibliography{files/bibliography}

\appendix

%xxxxxxxxxxxxxxxxxxxxxxxxxxxxxxxxxxxxxxxxxxxxxxxxxxxxxxxxxxxxxxxxxxxxxxxx
\section*{Appendix}
\addcontentsline{toc}{section}{Appendix}
\addtocounter{section}{1}
%------------------------------------------------------------------------

%========================================================================
\subsection{Equivalence of the Definitions of a Gibbs Measure}
%------------------------------------------------------------------------
\label{apx:gibbs:def:equivalence}

\begin{proposition}
	Let $\Delta$ be a Hamiltonian on a shift space $\xSp{X}$
	and $\pi\in\xSx{P}(\xSp{X})$ a probability measure.
	The following conditions are equivalent:
	\begin{enumerate}[ \rm a)]
		\item \label{item:gibbs:def:equivalence:1}
			For every finite set $D\subseteq\xL$,
			\begin{align}
			\label{eq:gibbs:def:original}
				\pi([q]_D\,|\,\family{F}_{\xCmpl{D}})(z) &=
					\xe^{-\Delta\left(p\vee\xRest{z}{\xCmpl{D}}\;,\;q\vee\xRest{z}{\xCmpl{D}}\right)}\,
					\pi([p]_D\,|\,\family{F}_{\xCmpl{D}})(z)
			\end{align}
			for $\pi$-almost every $z\in\xSp{X}$
			and every two patterns $p,q\in L_D(\xSp{X}\,|\, z)$.
		\item \label{item:gibbs:def:equivalence:2}
			For every finite set $D\subseteq\xL$,
			\begin{align}
			\label{eq:gibbs:def:ours:uniform}
				\frac{\pi([q\vee\xRest{z}{\xCmpl{D}}]_E)}{\pi([p\vee\xRest{z}{\xCmpl{D}}]_E)}
					&\to \xe^{-\Delta\left(p\vee\xRest{z}{\xCmpl{D}}\;,\;q\vee\xRest{z}{\xCmpl{D}}\right)}
			\end{align}
			uniformly in $z\in\supp(\pi)$ and $p,q\in L_D(\xSp{X}\,|\,z)$
			as $E\nearrow\xL$ along the directed family of finite subsets of $\xL$.
		\item \label{item:gibbs:def:equivalence:3}
			For every configuration $x\in\xSp{X}$ that is in the support of $\pi$
			and every configuration $y\in\xSp{X}$
			that is asymptotic to~$x$,
			\begin{align}
			\label{eq:gibbs:def:ours}
				\frac{\pi([y]_E)}{\pi([x]_E)}
					&\to \xe^{-\Delta(x,y)} \;,
			\end{align}
			as $E\nearrow\xL$ along the directed family of finite subsets of $\xL$.
	\end{enumerate}
\end{proposition}
\begin{proof}\ \\
	\begin{enumerate}[(a)$\Rightarrow$(b)]
		\item[(\ref{item:gibbs:def:equivalence:1})$\Rightarrow$(\ref{item:gibbs:def:equivalence:2})]
			Assume that condition~(\ref{item:gibbs:def:equivalence:1}) is satisfied.
			Let $D\subseteq\xL$ be a finite set.
			Integrating~(\ref{eq:gibbs:def:original}), for any finite $E\supseteq D$ we get
			\begin{align}
				\pi([q\lor\xRest{z}{\xCmpl{D}}]_E) &= \int_{[z]_{E\setminus D}}
					\xe^{-\Delta\left(p\vee\xRest{\zeta}{\xCmpl{D}}\;,\;q\vee\xRest{\zeta}{\xCmpl{D}}\right)}
					\pi([p]_D\,|\,\family{F}_{\xCmpl{D}})(\zeta)\pi(\xd\zeta) \;.
			\end{align}
			Setting
			\begin{align}
				\delta(z,\zeta) &\IsDef
					\xe^{-\Delta\left(p\vee\xRest{\zeta}{\xCmpl{D}}\;,\;q\vee\xRest{\zeta}{\xCmpl{D}}\right)} -
					\xe^{-\Delta\left(p\vee\xRest{z}{\xCmpl{D}}\;,\;q\vee\xRest{z}{\xCmpl{D}}\right)} \;,
			\end{align}
			we can write
			\begin{align}
				\pi([q\lor\xRest{z}{\xCmpl{D}}]_E) &= \int_{[z]_{E\setminus D}}
					\left(
						\xe^{-\Delta\left(p\vee\xRest{z}{\xCmpl{D}}\;,\;q\vee\xRest{z}{\xCmpl{D}}\right)} +
						\delta(z,\zeta)
					\right)
					\pi([p]_D\,|\,\family{F}_{\xCmpl{D}})(\zeta)\pi(\xd\zeta) \\
				&= \label{eq:gibbs:def:proof:a-b}
					\begin{multlined}[t]
						\xe^{-\Delta\left(p\vee\xRest{z}{\xCmpl{D}}\;,\;q\vee\xRest{z}{\xCmpl{D}}\right)}
						\pi([p\lor\xRest{z}{\xCmpl{D}}]_E) \\
						\qquad + \int_{[z]_{E\setminus D}}
							\delta(z,\zeta)\,\pi([p]_D\,|\,\family{F}_{\xCmpl{D}})(\zeta)\pi(\xd\zeta) \;.
					\end{multlined}
			\end{align}
			Now, let $p,q\in L_D(\xSp{X})$ be fixed patterns with $\pi([p]_D),\pi([q]_D)>0$,
			and let $\varepsilon>0$.
			By the uniform continuity of $z\mapsto\Delta(p\lor\xRest{z}{\xCmpl{D}},q\lor\xRest{z}{\xCmpl{D}})$,
			there is a sufficiently large finite set $E_\varepsilon\subseteq\xL$
			such that, for every $E\supseteq E_\varepsilon$ and every $z,\zeta$ with
			$\xRest{z}{E\setminus D}=\xRest{\zeta}{E\setminus D}$, we have
			$\abs{\delta(z,\zeta)}<\varepsilon$.
			In particular, for every $E\supseteq E_\varepsilon$ and every $z\in\supp(\pi)$
			satisfying $p,q\in L_D(\xSp{X}\,|\,z)$, we get
			\begin{align}
				\abs{
					\int_{[z]_{E\setminus D}}
						\delta(z,\zeta)\,\pi([p]_D\,|\,\family{F}_{\xCmpl{D}})(\zeta)\pi(\xd\zeta)
				} &\leq
					\int_{[z]_{E\setminus D}}
						\abs{\delta(z,\zeta)}\pi([p]_D\,|\,\family{F}_{\xCmpl{D}})(\zeta)\pi(\xd\zeta) \\
				&<
					\varepsilon\, \pi([p\lor\xRest{z}{\xCmpl{D}}]_E) \;.
			\end{align}
			Substituting in~(\ref{eq:gibbs:def:proof:a-b}), we obtain, for every $z\in\supp(\pi)$
			satisfying $p,q\in L_D(\xSp{X}\,|\,z)$, that
			\begin{align}
				\abs{
					\pi([q\lor\xRest{z}{\xCmpl{D}}]_E) -
					\xe^{-\Delta\left(p\vee\xRest{z}{\xCmpl{D}}\;,\;q\vee\xRest{z}{\xCmpl{D}}\right)}
					\pi([p\lor\xRest{z}{\xCmpl{D}}]_E)
				} &<
					\varepsilon\, \pi([p\lor\xRest{z}{\xCmpl{D}}]_E)
			\end{align}
			provided $E\supseteq E_\varepsilon$.
			Dividing by $\pi([p\lor\xRest{z}{\xCmpl{D}}]_E)$
			and letting $\varepsilon\to 0$ proves the claim.
		\item[(\ref{item:gibbs:def:equivalence:2})$\Rightarrow$(\ref{item:gibbs:def:equivalence:3})]
			Trivial.
		\item[(\ref{item:gibbs:def:equivalence:3})$\Rightarrow$(\ref{item:gibbs:def:equivalence:1})]
			Suppose that $\pi$ satisfies condition~(\ref{item:gibbs:def:equivalence:3}).
			Let $I_1\subseteq I_2\subseteq\cdots$ be an arbitrary chain of finite subsets of $\xL$
			with $\bigcup_n I_n=\xL$.
			Let $z\in\xSp{X}$ be a configuration in the support of $\pi$ and $D\subseteq\xL$ a finite set.
			For every two patterns $p,q\in L_D(\xSp{X}\,|\,z)$,
			we have
			\begin{align}
				\frac{\pi([q]_D\,|\,[z]_{I_n\setminus D})}{\pi([p]_D\,|\,[z]_{I_n\setminus D})}
				= \frac{\pi([q\vee\cramped{\xRest{z}{\xCmpl{D}}}]_{I_n})}{%
					\pi([p\vee\cramped{\xRest{z}{\xCmpl{D}}}]_{I_n})
				}
				&\to
					\xe^{-\Delta\left(p\vee\xRest{z}{\xCmpl{D}}\;,\;q\vee\xRest{z}{\xCmpl{D}}\right)}
			\end{align}
			as $n\to\infty$, implying that
			\begin{align}
			\label{eq:gibbs:def:proof:c-a:1}
				\pi([q]_D\,|\,[z]_{I_n\setminus D}) - 
				\xe^{-\Delta\left(p\vee\xRest{z}{\xCmpl{D}}\;,\;q\vee\xRest{z}{\xCmpl{D}}\right)}
				\pi([p]_D\,|\,[z]_{I_n\setminus D})
					&\to 0
			\end{align}
			as $n\to\infty$.
			
			Note that the $\sigma$-algebra $\family{F}_{\xCmpl{D}}$
			is generated by the filtration
			$\family{F}_{I_1\setminus D}\subseteq \family{F}_{I_2\setminus D}\subseteq\cdots$.
			Therefore, by the martingale convergence theorem,
			for $\pi$-almost every $z$, and every $p\in L_D(\xSp{X}\,|\, z)$,
			\begin{align}
			\label{eq:gibbs:def:proof:c-a:2}
				\pi([p]_D\,|\,[z]_{I_n\setminus D}) =
				\pi([p]_D\,|\,\family{F}_{I_n\setminus D})(z) &\to \pi([p]_D\,|\,\family{F}_{\xCmpl{D}})(z)
			\end{align}
			as $n\to\infty$.
			
			Combining~(\ref{eq:gibbs:def:proof:c-a:1}) and~(\ref{eq:gibbs:def:proof:c-a:2}),
			we obtain
			\begin{align}
				\pi([q]_D\,|\,\family{F}_{\xCmpl{D}})(z) &=
				\xe^{-\Delta\left(p\vee\xRest{z}{\xCmpl{D}}\;,\;q\vee\xRest{z}{\xCmpl{D}}\right)}
				\pi([p]_D\,|\,\family{F}_{\xCmpl{D}})(z)
			\end{align}
			for $\pi$-almost every $z\in\xSp{X}$ and every $p,q\in L_D(\xSp{X}\,|\,z)$.
	\end{enumerate}
\end{proof}

\end{document}